\RequirePackage{fix-cm}
\documentclass[smallextended]{svjour3}       
\smartqed  
\usepackage{graphicx}
\begin{document}

\title{Affine matrix rank minimization problem via non-convex fraction function penalty}


\author{Angang Cui$^{1}$\and
        Jigen Peng$^{\ast, 1}$\and
        Haiyang Li$^{2}$\and
        Chengyi Zhang$^{2}$\and
       Yongchao Yu$^{1}$
}


\institute{$\ast$ Corresponding author\\
            Jigen Peng \at
              \email{jgpengxjtu@126.com}\\
           1 School of Mathematics and Statistics, Xi'an Jiaotong University, Xi'an, 710049, China \\
           2 School of Science, Xi'an Polytechnic University, Xi'an, 710048, China
}

\date{Received: date / Accepted: date}

\maketitle

\begin{abstract}
Affine matrix rank minimization problem is a fundamental problem with a lot of important applications in many fields. It is well known that this problem is
combinatorial and NP-hard in general.
In this paper, a continuous promoting low rank non-convex fraction function is studied to replace the rank function in this NP-hard problem. Inspired by our former work in
compressed sensing, an iterative singular value thresholding algorithm is proposed to solve the regularization transformed affine matrix rank minimization
problem. For different $a>0$, we could get a much better result by adjusting the different value of $a$, which is one of the advantages for the iterative
singular value thresholding algorithm compared with some state-of-art methods. Some convergence results are established and numerical experiments show that this
thresholding algorithm is feasible for solving the regularization transformed affine matrix rank minimization problem. Moreover, we proved that the value of the
regularization parameter $\lambda>0$ can not be chosen too large. Indeed, there exists $\bar{\lambda}>0$ such that the
optimal solution of the regularization transformed affine matrix rank minimization problem is equal to zero for any $\lambda>\bar{\lambda}$. Numerical experiments
on matrix completion problems show that our method performs powerful in finding a low-rank matrix and the numerical experiments about image inpainting problems
show that our algorithm has better performances than some state-of-art methods.

\end{abstract}

\section{Introduction}
\label{intro}
In recent years, affine matrix rank minimization problem (ARMP) has attracted much attention in many important application fileds such as collaborative
filtering in recommender systems [1,2], minimum order system and  low-dimensional Euclidean embedding in control theory [3,4], network localization [5],
and so on. ARMP can be viewed as the following mathematical form
\begin{equation}\label{r1}
(\mathrm{ARMP})\ \ \ \ \ \ \min_{X\in \mathcal{R}^{m\times n}} \ \mbox{rank}(X)\ \ s.t. \ \  \mathcal{A}(X)=b
\end{equation}
where the linear map $\mathcal{A}: \mathcal{R}^{m\times n}\mapsto \mathcal{R}^{d}$ and the vector $b\in \mathcal{R}^{d}$ are given. Without loss of generality,
we assume $m\leq n$. The matrix completion problem
\begin{equation}\label{r2}
\min_{X\in \mathcal{R}^{m\times n}} \ \mbox{rank}(X)\ \ s.t. \ \  X_{i,j}=M_{i,j},\ \ (i,j)\in \Omega
\end{equation}
is a special case of ARMP, where $X$ and $M$ are both $m\times n$ matrices and $\Omega$ is a subset of indexes set of all pairs $(i,j)$. If the projector
$\mathcal{P}_{\Omega}: \mathcal{R}^{m\times n}\rightarrow \mathcal{R}^{m\times n}$ is defined as
\begin{equation}\label{r3}
[\mathcal{P}_{\Omega}X]_{i,j}=\left\{
                         \begin{array}{ll}
                           X_{i,j}, & \hbox{if $(i,j)\in \Omega$;} \\
                           0, & \hbox{if $(i,j)\notin \Omega$,}
                         \end{array}
                       \right.
\end{equation}
and the resulting matrix is $X_{\Omega}=\mathcal{P}_{\Omega}X$, then the matrix completion problem can be reformulated as
\begin{equation}\label{r4}
\min_{X\in \mathcal{R}^{m\times n}} \ \mbox{rank}(X)\ \ s.t. \ \ \mathcal{P}_{\Omega}X=X_{\Omega}.
\end{equation}

Usually, ARMP could be transformed into the following regularization problem
\begin{equation}\label{r5}
\min_{X\in \mathcal{R}^{m\times n}} \Big\{\|\mathcal{A}(X)-b\|_{2}^{2}+\lambda\cdot\mathrm{rank}(X)\Big\}
\end{equation}
where $\lambda>0$, which is called the regularization parameter. Unfortunately, although $\mathrm{rank}(X)$ characterizes the rank
of the matrix $X$, and it is a challenging non-convex optimization problem and is known as NP-hard [6].

Nuclear-norm affine matrix rank minimization problem (NuARMP) is the most popular alternative [1,4,6-9], and the minimization has the following form
\begin{equation}\label{r6}
(\mathrm{NuARMP})\ \ \ \ \ \ \min_{X\in \mathcal{R}^{m\times n}} \ \|X\|_{\ast}\ \ s.t. \ \  \mathcal{A}(X)=b
\end{equation}
for the constrained problem and
\begin{equation}\label{r7}
(\mathrm{RNuARMP})\ \ \ \ \ \ \min_{X\in \mathcal{R}^{m\times n}} \Big\{\|\mathcal{A}(X)-b\|_{2}^{2}+\lambda\cdot\|X\|_{\ast}\Big\}
\end{equation}
for the regularization problem, where $\|X\|_{\ast}=\sum_{i=1}^{m}\sigma_{i}(X)$ is called nuclear-norm of matrix $X$, and
$\sigma_{i}(X)$ presents the $i$-th largest singular value of matrix $X$ arranged in descending order.

As the compact convex relaxation of the NP-hard ARMP, NuARMP possesses many theoretical and algorithmic advantages [10-13].
However, it may be suboptimal for recovering a real low-rank matrix. In fact, NuARMP or RNuARMP may yield a matrix with much
higher rank and need more observations to recover a real low-rank matrix [1,11]. Moreover, the nuclear-norm is a loose approximation
of the rank function and tends to lead to biased estimation by shrinking all the singular values toward to zero simultaneously, and sometimes
results in over-penalization in its regularization model as the $\l_{1}$-norm in compressed sensing [14].

With recent development of non-convex relaxation approach in sparse signal recovery problems, many researchers have shown that using a
continuous non-convex function to approximate the $\l_{0}$-norm is a better choice than using the $\l_{1}$-norm [15-26]. This brings our
attention back to the non-convex functions and we replace the rank function by a continuous promoting low-rank non-convex function.
Through this transformation, ARMP can be translated into a transformed ARMP (TrARMP) which has the following form
\begin{equation}\label{r8}
(\mathrm{TrARMP})\ \ \ \ \ \ \ \ \ \ \min_{X\in \mathcal{R}^{m\times n}} \ P(X)\ \ s.t. \ \  \mathcal{A}(X)=b
\end{equation}
for the constrained problem and
\begin{equation}\label{r9}
(\mathrm{RTrARMP})\ \ \ \ \ \ \min_{X\in \mathcal{R}^{m\times n}} \Big\{\|\mathcal{A}(X)-b\|_{2}^{2}+\lambda\cdot P(X)\Big\}
\end{equation}
for the regularization problem, where the continuous promoting low rank non-convex function $P(X)$ is in terms of singular values of matrix $X$,
typically, $P(X)=\sum_{i=1}^{m}\rho(\sigma_{i}(X))$.

In particular, if $\|\mathcal{A}(X)-b\|_{2}^{2}$ depends only on the diagonal entries of matrix $X$, RTrARMP reduces to a regularization vector
minimization problem (RVMP) which is based on non-convex function in the form of
\begin{equation}\label{r10}
(\mathrm{RVMP})\ \ \ \ \ \ \ \ \ \ \ \  \ \ \min_{x\in \mathcal{R}^{\l}} \Big\{\|Ax-b\|_{2}^{2}+\lambda\cdot P(x)\Big\}
\end{equation}
where $A\in \mathcal{R}^{d\times \l}$, $l=\min\{m,n\}$, and $P(x)=\sum_{i=1}^{l}\rho_{a}(x_{i})$ for any $x\in \mathcal{R}^{l}$.

In this paper, a continuous promoting low-rank non-convex function
\begin{equation}\label{r11}
P(X)=\sum_{i=1}^{m}\rho_{a}(\sigma_{i}(X))
\end{equation}
is studied to replace the $\mathrm{rank}(X)$ in RTrARMP, where the non-convex function
\begin{equation}\label{r12}
\rho_{a}(t)=\frac{a|t|}{a|t|+1}
\end{equation}
is called the non-convex fraction function, and the parameter $a\in(0,+\infty)$. It is easy to verify that $\rho_{a}(t)$ is increasing and concave in
$t\in(-\infty, +\infty)$, and

\begin{equation}\label{r13}
\lim_{a\rightarrow+\infty}\rho_{a}(t)=\left\{
    \begin{array}{ll}
      0, & {\ \ \mathrm{if} \ t=0;} \\
      1, & {\ \ \mathrm{if} \ t\neq 0.}
    \end{array}
  \right.
\end{equation}

The non-convex fraction function $\rho_{a}(t)$ is called "strictly non-interpolating" in [27], and widely used in image restoration. German
in [27] showed that the non-convex fraction function gave rise to a step-shaped estimate from ramp-shaped data. And in [28], Nikolova demonstrated that for
almost all data, the strongly homogeneous zones recovered by the non-convex fraction function were preserved constant under any small perturbation of the data.

Inspired by our former work in compressed sensing [29], an iterative singular value thresholding algorithm (ISVTA) is proposed to solve RTrARMP in this paper.
For different $a>0$, we could get a much better result by adjusting different values of parameter $a$, which is one of the advantages for ISVTA compared with
some state-of-art methods.

\begin{figure}
 \centering
 \includegraphics[width=0.75\textwidth]{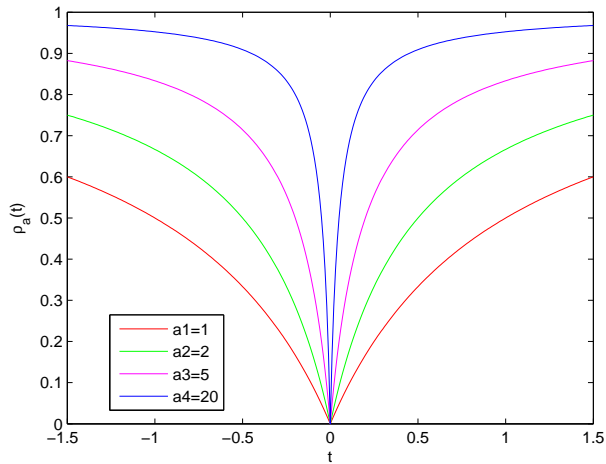}
\caption{The behavior of the fraction function $\rho_{a}(t)$ for various values of $a$.}
\label{fig:1}       
\end{figure}

The rest of this paper is organized as the following. In Section 2, we recall the iterative thresholding algorithm (ITA) of RVMP, and ISVTA is proposed
to solve RTrARMP in Section 3. In section 4, the convergence of ISVTA is established. In Section 5, we demonstrate some numerical experiments
on matrix completion problems and image inpainting problems, and numerical results show that ISVTA performs powerful in finding a low-rank matrix and the
numerical results for image inpainting problems show that our algorithm performs much better than some state-of-art methods. Some conclusion
remarks are presented in Section 6.

\section{Iterative thresholding algorithm (ITA) for solving RVMP} \label{prelimi-sec}

In this section, the iterative thresholding algorithm (ITA) of RVMP is recalled for all positive parameter $a>0$, which underlies the algorithm to be proposed.
Before the analytic expression of ITA, some crucial results need to be introduced for later use.
\begin{lemma}\label{le1} {\rm [29]}
Define three iterative thresholding values
$$t_{1}^{\ast}=\frac{\sqrt[3]{\frac{27}{8}\lambda a^{2}}-1}{a}, \ \ \ \ t_{2}^{\ast}=\frac{\lambda}{2}a, \ \ \ \ t_{3}^{\ast}=\sqrt{\lambda}-\frac{1}{2a}$$
for any positive parameters $\lambda$ and $a$, then the inequalities $t_{1}^{\ast}\leq t_{3}^{\ast}\leq t_{2}^{\ast}$ hold. Furthermore,
they are equal to $\frac{1}{2a}$ when $\lambda=\frac{1}{a^{2}}$.
\end{lemma}

Define a function of $\beta\in \mathcal{R}$ as
$$f_{\lambda}(\beta)=(\beta-\gamma)^{2}+\lambda\cdot\rho_{a}(\beta)$$
and
$$\beta^{\ast}=\arg\min_{\beta\in \mathcal{R}}f_{\lambda}(\beta).$$

\begin{lemma}\label{le2} {\rm [29]}
The optimal solution to $\beta^{\ast}=\displaystyle \arg\min_{\beta\in \mathcal{R}} f_{\lambda}(\beta)$ is the threshold function defined as
\begin{equation}\label{r14}
\beta^{\ast}=\left\{
    \begin{array}{ll}
      g_{\lambda}(\gamma), & \ \ \mathrm{if} \ {|\gamma|> t^{\ast};} \\
      0, & \ \ \mathrm{if} \ {|\gamma|\leq t^{\ast}.}
    \end{array}
  \right.
\end{equation}
where $g_{\lambda}(\gamma)$ is defined as
\begin{equation}\label{r15}
g_{\lambda}(\gamma)=sign(\gamma)\bigg(\frac{\frac{1+a|\gamma|}{3}(1+2\cos(\frac{\phi(\gamma)}{3}-\frac{\pi}{3}))-1}{a}\bigg),
\end{equation}
$$\phi(\gamma)=\arccos\Big(\frac{27\lambda a^{2}}{4(1+a|\gamma|)^{3}}-1\Big),$$
and the threshold value satisfies
\begin{equation}\label{r16}
t^{\ast}=\left\{
    \begin{array}{ll}
      t_{2}^{\ast}, & \ \ \mathrm{if} \ {\lambda\leq \frac{1}{a^{2}};} \\
      t_{3}^{\ast}, & \ \ \mathrm{if} \ {\lambda>\frac{1}{a^{2}}.}
    \end{array}
  \right.
\end{equation}
\end{lemma}

\begin{figure}
  \begin{minipage}[t]{0.5\linewidth}
  \centering
  \includegraphics[width=1\textwidth]{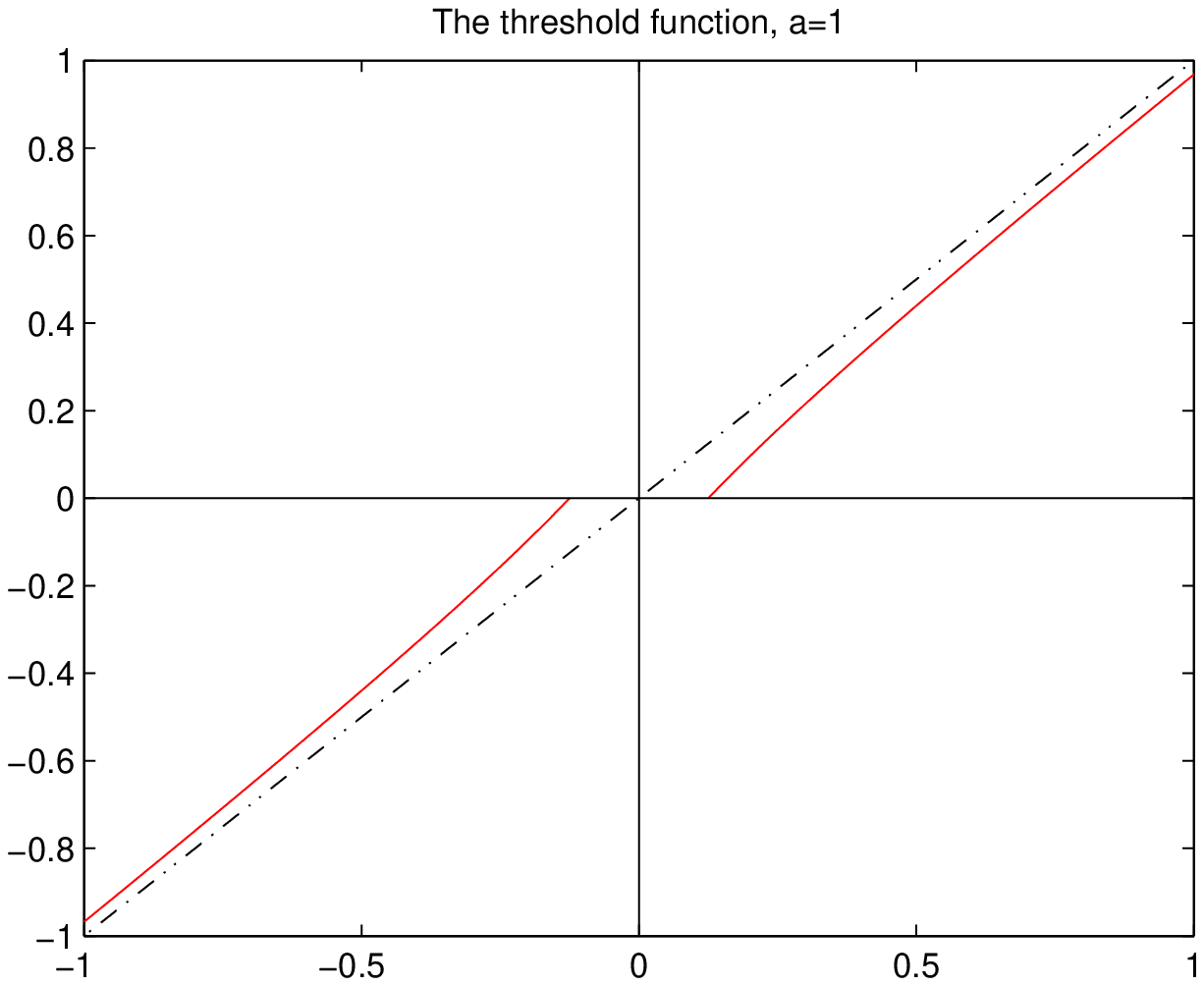}
  \end{minipage}
  \begin{minipage}[t]{0.5\linewidth}
  \centering
  \includegraphics[width=1\textwidth]{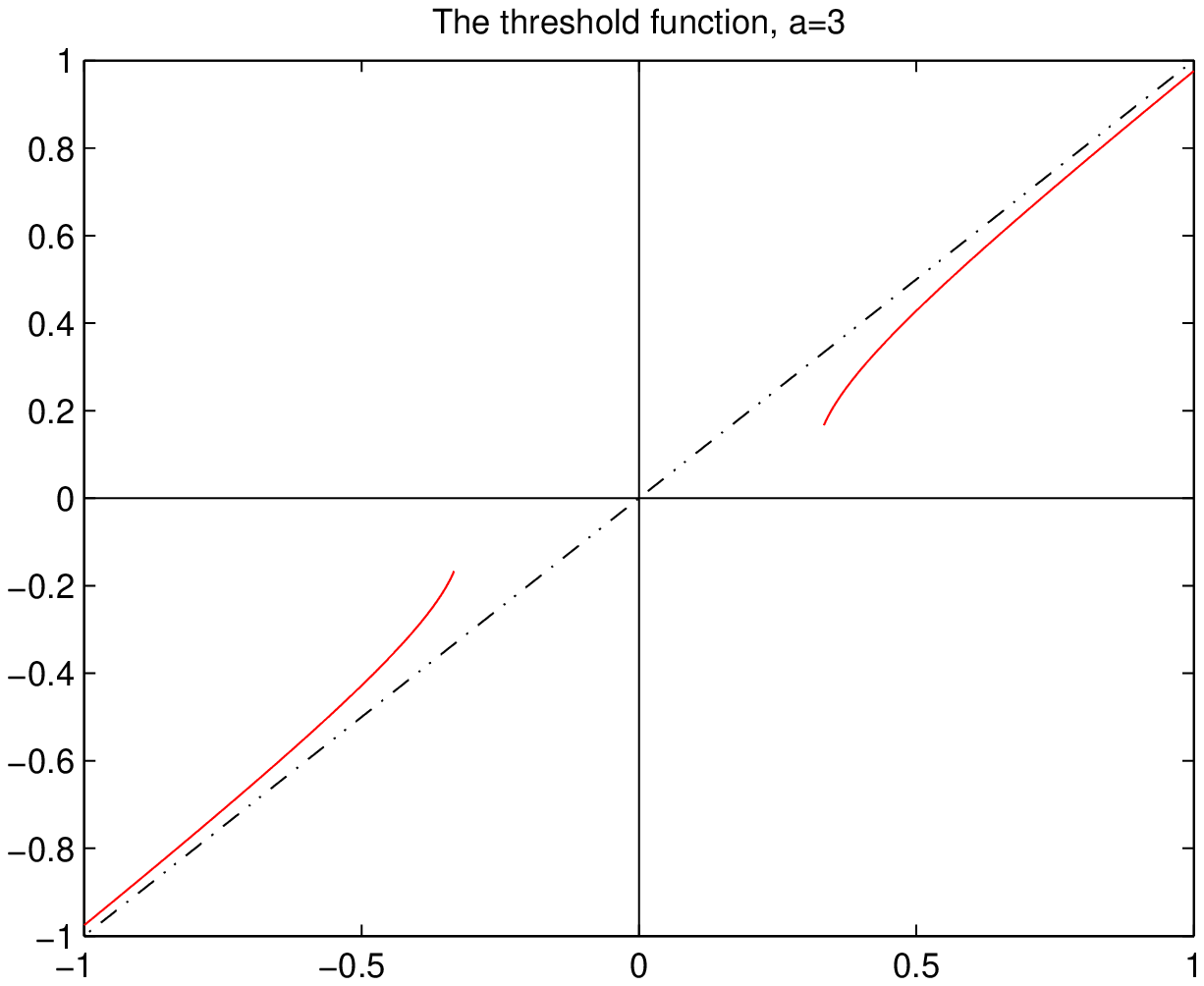}
  \end{minipage}
  \begin{minipage}[t]{0.5\linewidth}
  \centering
  \includegraphics[width=1\textwidth]{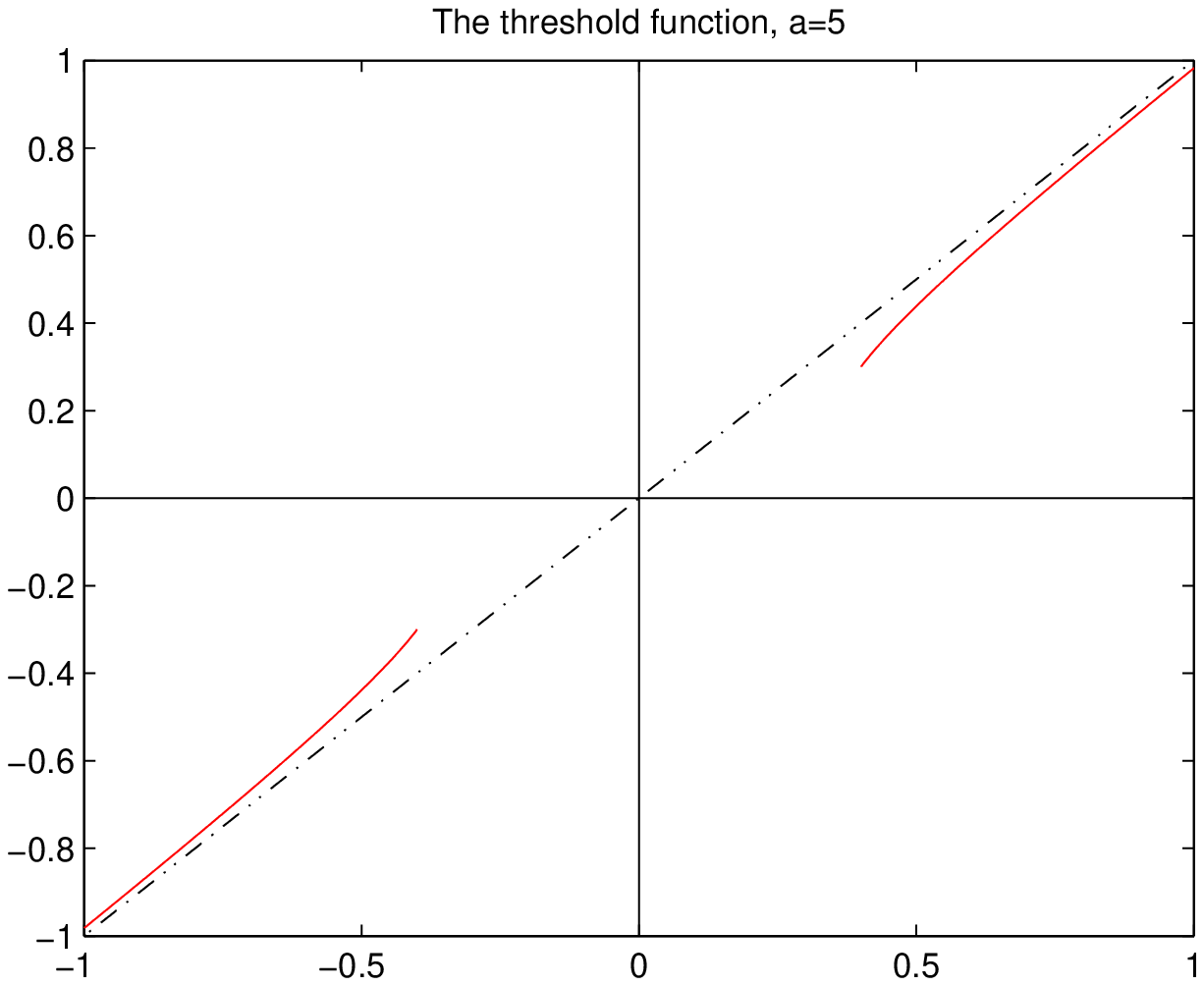}
  \end{minipage}
  \begin{minipage}[t]{0.5\linewidth}
  \centering
  \includegraphics[width=1\textwidth]{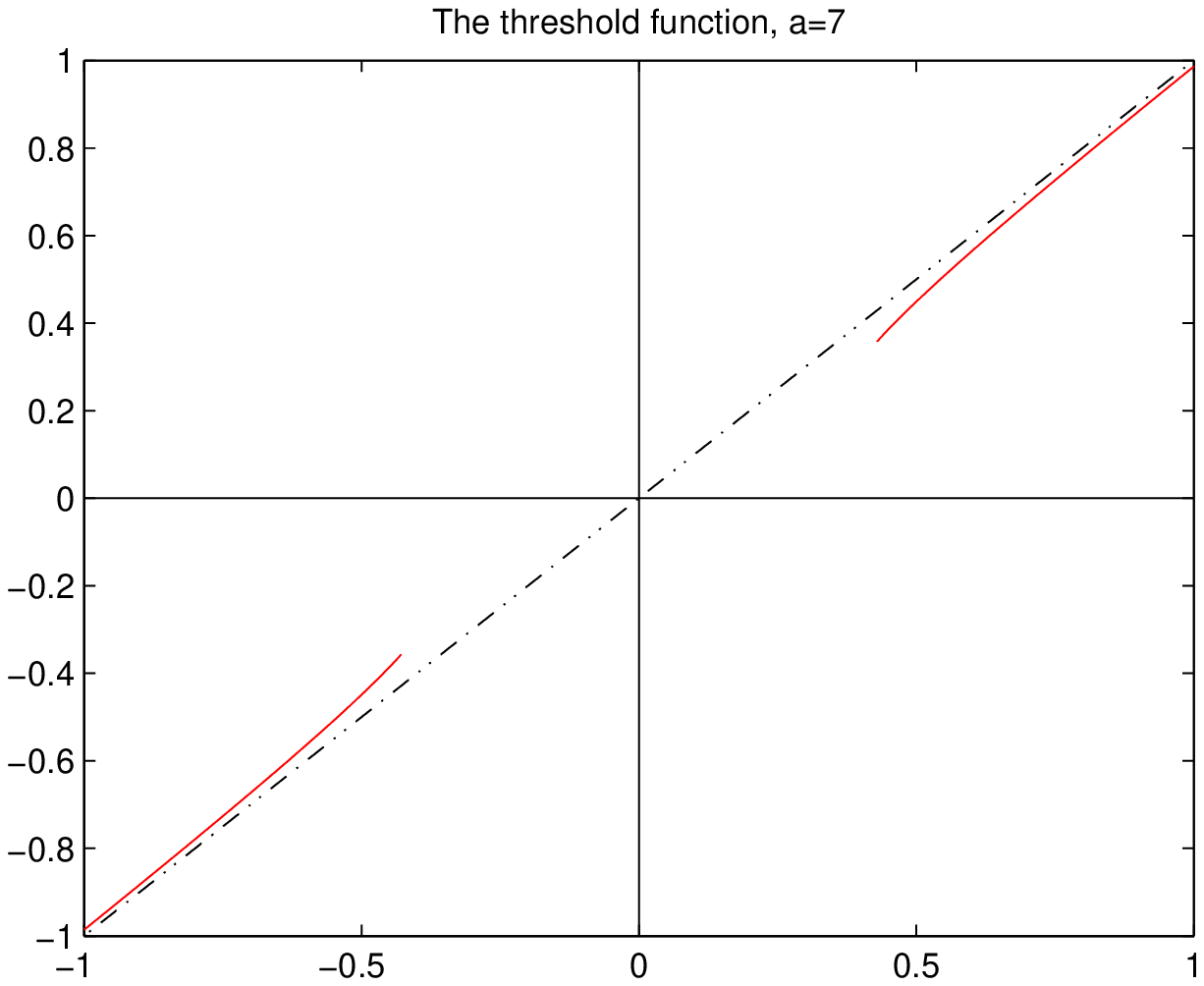}
  \end{minipage}
  \begin{minipage}[t]{0.5\linewidth}
  \centering
  \includegraphics[width=1\textwidth]{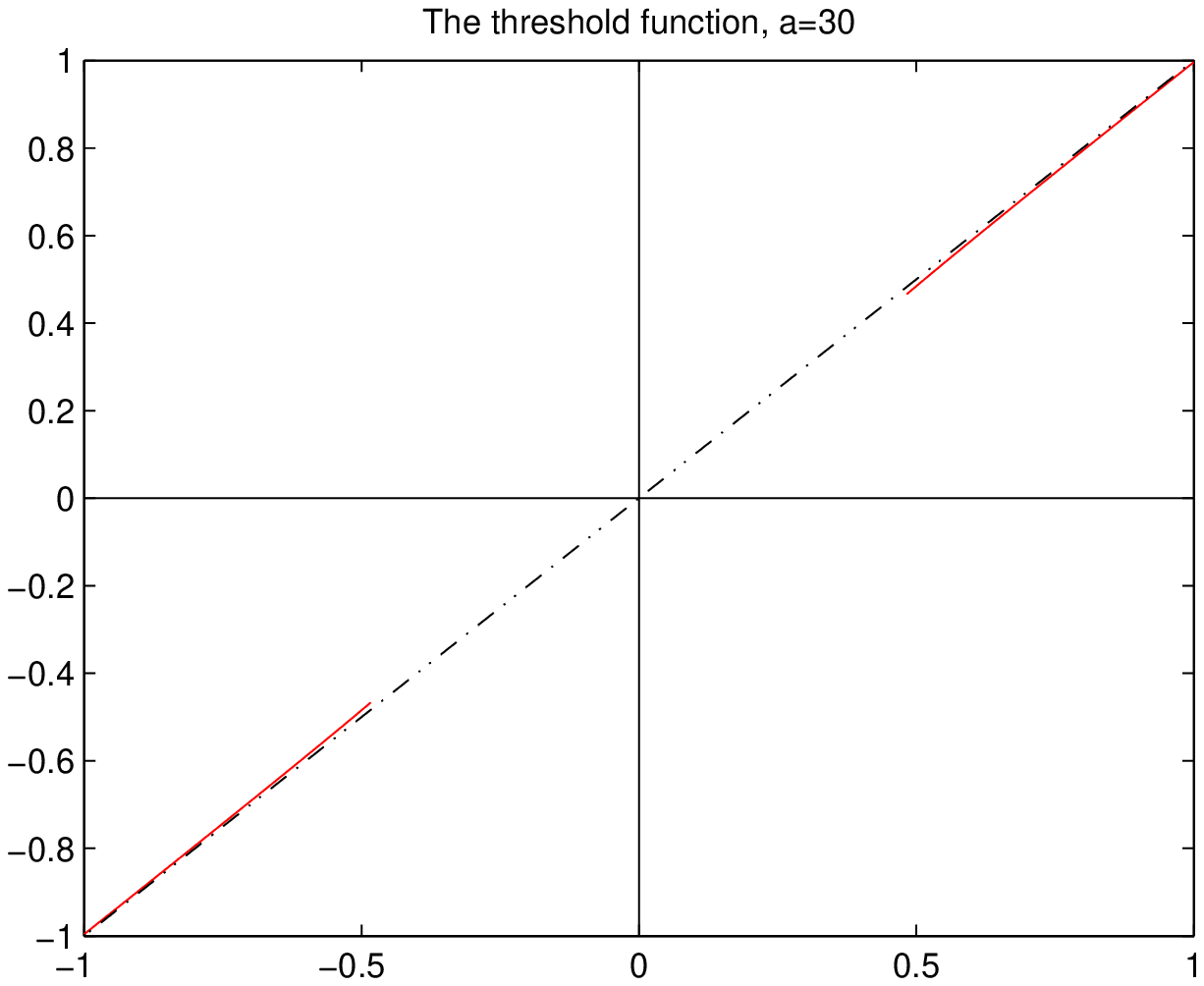}
  \end{minipage}
  \begin{minipage}[t]{0.5\linewidth}
  \centering
  \includegraphics[width=1\textwidth]{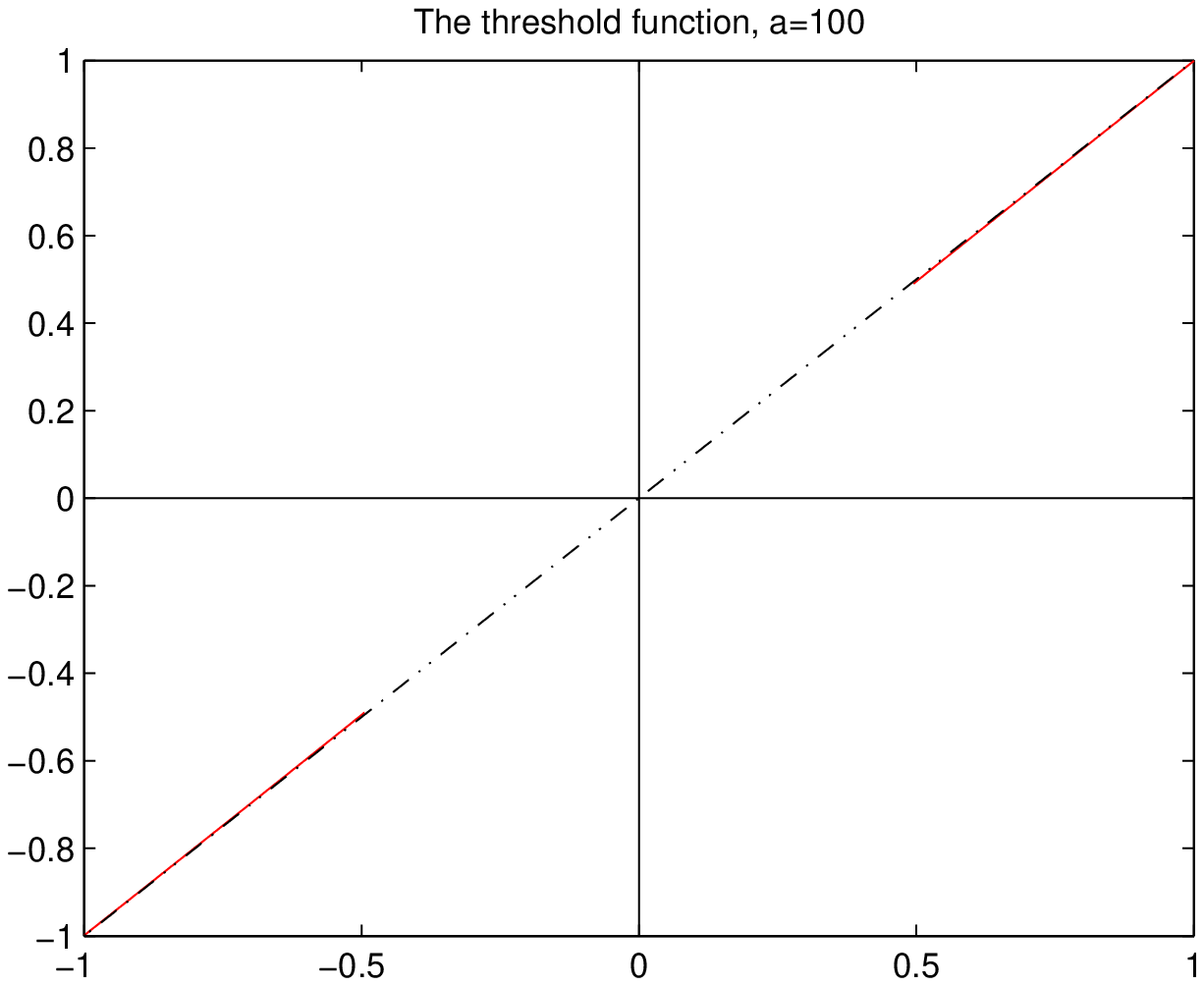}
  \end{minipage}
  \caption{The plots of the threshold functions for a=1, 3, 5, 7, 30, 100, and $\lambda=0.25$.} \label{fig:2}
\end{figure}

\begin{lemma}\label{le3}
The optimal solution $\beta^{\ast}$ defined in Lemma 2 is monotone.
\end{lemma}

\begin{proof}
For different $\gamma_{1}$ and $\gamma_{2}$, let
$$\beta_{1}\in\displaystyle\arg\min_{\beta\in \mathcal{R}}\Big\{(\beta-\gamma_{1})^{2}+\lambda\cdot\rho_{a}(\beta)\Big\},$$
$$\beta_{2}\in\displaystyle\arg\min_{\beta\in \mathcal{R}}\Big\{(\beta-\gamma_{2})^{2}+\lambda\cdot\rho_{a}(\beta)\Big\}.$$
Then, we have
$$(\beta_{2}-\gamma_{1})^{2}+\lambda\cdot \rho_{a}(\beta_{2})\geq(\beta_{1}-\gamma_{1})^{2}+\lambda\cdot \rho_{a}(\beta_{1}),$$
$$(\beta_{1}-\gamma_{2})^{2}+\lambda\cdot \rho_{a}(\beta_{1})\geq(\beta_{2}-\gamma_{2})^{2}+\lambda\cdot \rho_{a}(\beta_{2}).$$
By adding them, we have
$$(\beta_{2}-\gamma_{1})^{2}+(\beta_{1}-\gamma_{2})^{2}\geq(\beta_{1}-\gamma_{1})^{2}+(\beta_{2}-\gamma_{2})^{2}$$
and
$$(\beta_{1}-\beta_{2})(\gamma_{1}-\gamma_{2})\geq0.$$
This completes the proof.
\end{proof}

\begin{definition}\label{de1}
The iterative thresholding operator $R_{\lambda, P}$ is a diagonally nonlinear analytically expressive operator, and can be specified by
\begin{equation}\label{r17}
R_{\lambda, P}(x)=(g_{\lambda}(x_{1}), \cdots, g_{\lambda}(x_{l}))^{T},
\end{equation}
where $g_{\lambda}(x_{i})$ is defined in Lemma 2.
\end{definition}

Nextly, we will show that the optimal solution to RVMP can be expressed as an iterative thresholding operation.

For any $\lambda$, $\mu\in(0,+\infty)$ and $z\in \mathcal{R}^{\l}$, let
\begin{equation}\label{r18}
C_{\lambda}(x)=\|Ax-b\|_{2}^{2}+\lambda P(x),
\end{equation}
\begin{equation}\label{r19}
C_{\mu}(x, z)=\mu\Big[C_{\lambda}(x)-\|Ax-Az\|_{2}^{2}\Big]+\|x-z\|_{2}^{2}
\end{equation}
and
$$B_{\mu}(x)=x+\mu A^{T}(b-Ax).$$
Clearly, $C_{\mu}(x, x)=\mu C_{\lambda}(x)$.
\begin{theorem}\label{th1}{\rm [29]}
For positive parameters $\mu$, $a$ and $\lambda$, if $x^{s}=(x_{1}^{s}, x_{2}^{s},\cdots, x_{\l}^{s})^{T}$ is a local optimal solution
to $\displaystyle\min_{x\in \mathcal{R}^{n}}C_{\mu}(x,z)$, then
$$x_{i}^{s}=0\Leftrightarrow |(B_{\mu}(z))_{i}|\leq t^{\ast}$$
and
$$x_{i}^{s}=g_{\lambda\mu}((B_{\mu}(z))_{i})\Leftrightarrow |(B_{\mu}(z))_{i}|>t^{\ast}$$
where parameter $t^{\ast}$ is defined in (16) and $g_{\lambda\mu}$ is obtained by replacing $\lambda$ with $\lambda\mu$ in $g_{\lambda}$.
\end{theorem}

\begin{theorem}\label{th2}{\rm [29]}
If $x^{\ast}=(x_{1}^{\ast}, x_{2}^{\ast}, \cdots, x_{l}^{\ast})^{T}$ is an optimal solution to RVMP and $0<\mu<\|A\|_{2}^{-2}$,
then the optimal solution $x^{\ast}$ satisfies the fixed point equation
\begin{equation}\label{r20}
x_{i}^{\ast}=\left\{
    \begin{array}{ll}
      g_{\lambda\mu}([B_{\mu}(x^{\ast})]_{i}), &\ \ \mathrm{if} \ {|[B_{\mu}(x^{\ast})]_{i}|>t^{\ast};} \\
      0, &\ \ \mathrm{if} \ {|[B_{\mu}(x^{\ast})]_{i}|\leq t^{\ast},}
    \end{array}
  \right.
\end{equation}
where
\begin{equation}\label{r21}
\begin{array}{llll}
t^{\ast}=\left\{
    \begin{array}{ll}
      t_{2}^{\ast}=\frac{\lambda\mu}{2}a, &\ \ \mathrm{if} \ {\lambda\leq\frac{1}{a^{2}\mu};} \\
      t_{3}^{\ast}=\sqrt{\lambda\mu}-\frac{1}{2a}, &\ \ \mathrm{if} \ {\lambda>\frac{1}{a^{2}\mu}.}
    \end{array}
  \right.
\end{array}
\end{equation}
\end{theorem}

By Definition 1, Theorem 1 and Theorem 2, the ITA for solving RVMP can be naturally defined as
\begin{equation}\label{r22}
\begin{array}{llll}
x^{k+1}&=&R_{\lambda \mu, P}(B_{u}(x^{k}))\\
&=&R_{\lambda \mu, P}(x^{k}+\mu A^{T}(b-Ax^{k}))
\end{array}
\end{equation}
where $B_{\mu}(x^{k})=x^{k}+\mu A^{T}(b-Ax^{k})$ and $R_{\lambda\mu}$ is obtained by replacing $\lambda$ with $\lambda\mu$ in $R_{\lambda}$.
The more detailed accounts of ITA for solving RVMP can be seen in [29].

\section{Iterative singular value thresholding algorithm (ISVTA) for solving RTrARMP}

Inspired by ITA for solving RMVP in Section 2, iterative singular value thresholding algorithm (ISVTA) is proposed to solve RTrARMP
in this section. We begin with the definition of a key building block, namely, the iterative singular value thresholding operator of RTrARMP.

\begin{definition}\label{de2}
The iterative singular value thresholding operator $G_{\lambda, P}$ is a diagonally nonlinear analytically expressive operator, and can be specified by
$$G_{\lambda, P}(X)=UG_{\lambda, P}(\Sigma)V^{T}=U\mathrm{Diag}\Big(g_{\lambda}(\sigma_{i}(X))\Big)V^{T}$$
where $g_{\lambda}$ is defined in Lemma 2, and
$$X=U\mathrm{Diag}(\sigma_{1}(X), \sigma_{2}(X), \cdots, \sigma_{m}(X))V^{T}$$
is the singular value decomposition (SVD) of matrix $X\in \mathcal{R}^{m\times n}$, $\sigma_{1}(X)\geq\sigma_{2}(X)\geq\cdots\geq\sigma_{m}(x)\geq0$ are
the singular values of matrix $X$.
\end{definition}

Before computing ISVTA for solving RTrARMP, we need the following von Neumann's trace inequality which plays a key role in our later analysis.
\begin{lemma}\label{le4}
(von Neumann's trace inequality) For any matrices $X, Y\in \mathcal{R}^{m\times n}$ $(m\leq n)$, $Tr(X^{T}Y)\leq\sum_{i=1}^{m}\sigma_{i}(X)\sigma_{i}(Y)$, where $\sigma_{1}(X)\geq\sigma_{2}(X)\geq\cdots\geq\sigma_{m}(x)\geq0$ and $\sigma_{1}(Y)\geq\sigma_{2}(Y)\geq\cdots\geq\sigma_{m}(Y)\geq0$ are the singular value
of matrices $X$ and $Y$ respectively. The equality holds if and only if there exist unitary matrices $U$ and $V$ that such $X=UDiag(\sigma_{1}(X), \sigma_{2}(X), \cdots, \sigma_{m}(X))V^{T}$ and $Y=UDiag(\sigma_{1}(Y), \sigma_{2}(Y), \cdots, \sigma_{m}(Y))V^{T}$ as the singular value decompositions (SVDs) of matrices $X$ and $Y$
simultaneously.
\end{lemma}

Define a function of matrix $Y\in \mathcal{R}^{m\times n}$ as
\begin{equation}\label{r23}
f_{\lambda}(Y)=\|Y-X\|_{F}^{2}+\lambda\cdot P(Y)
\end{equation}
and
\begin{equation}\label{r24}
Y^{\ast}=\arg\min_{Y\in \mathcal{R}^{m\times n}}\Big\{\|Y-X\|_{F}^{2}+\lambda\cdot P(Y)\Big\}.
\end{equation}

According to the von Neumann's trace inequality in Lemma 4, the iterative singular value thresholding operator of matrix $Y^{\ast}$ can be
presented in the following description.

\begin{theorem}\label{th3}
Let $X=U\mathrm{Diag}(\sigma_{1}(X), \sigma_{2}(X), \cdots, \sigma_{m}(X))V^{T}$ be the SVD of matrix $X\in \mathcal{R}^{m\times n}$.
Then the iterative singular value thresholding operator of matrix $Y^{\ast}$ with the parameter $t^{\ast}>0$ can be expressed as
\begin{equation}\label{r25}
Y^{\ast}=G_{\lambda\mu, P}(X)=UG_{\lambda\mu, P}(\Sigma)V^{T}=U\mathrm{Diag}\Big(g_{\lambda\mu}(\sigma_{i}(X))\Big)V^{T}
\end{equation}
and
\begin{eqnarray*}
\sigma_{i}(Y^{\ast})=
\left\{
    \begin{array}{ll}
      g_{\lambda\mu}(\sigma_{i}(X)), &\ \ {\mathrm{if}\ \sigma_{i}(X)>t^{\ast};} \\
      0, & \ \ {\mathrm{if}\ \sigma_{i}(X)\leq t^{\ast},}
    \end{array}
  \right.
\end{eqnarray*}
where the parameter $t^{\ast}$ is the threshold value defined in (21).
\end{theorem}

\begin{proof}
Since $\sigma_{1}(X)\geq\sigma_{2}(X)\geq\cdots\geq\sigma_{m}\geq0$ are the singular values of matrix $X$, the minimization problem
\begin{equation}\label{r26}
\min_{Y\in \mathcal{R}^{m\times n}}\Big\{\|Y-X\|_{F}^{2}+\lambda\cdot P(Y)\Big\}
\end{equation}
can be rewritten as
$$\min_{\tau:\tau_{1}\geq\tau_{2}\geq\cdots\geq\tau_{m}\geq0}\bigg\{\min_{\sigma(Y)=\tau}\Big\{\|Y-X\|_{F}^{2}+\lambda\cdot \sum_{i=1}^{m}\rho_{a}(\tau_{i})\Big\}\bigg\}.$$
By using the trace inequality in Lemma 4, we have
\begin{eqnarray*}
\|Y-X\|_{F}^{2}&=&Tr(Y^{T}Y)-2Tr(Y^{T}X)+Tr(X^{T}X)\\
&=&\sum_{i=1}^{m}\sigma_{i}^{2}(Y)-2Tr(Y^{T}X)+\sum_{i=1}^{m}\sigma_{i}^{2}(X)\\
&\geq&\sum_{i=1}^{m}\sigma_{i}^{2}(Y)-2\sum_{i=1}^{m}\sigma_{i}(Y)\sigma_{i}(X)+\sum_{i=1}^{m}\sigma_{i}(X)\\
&=&\sum_{i=1}^{m}(\sigma_{i}(Y)-\sigma_{i}(X))^{2}.
\end{eqnarray*}

Noting that the above equality holds when $Y$ admits the singular value decomposition
$$Y=U\mathrm{Diag}(\sigma_{1}(Y), \sigma_{2}(Y), \cdots, \sigma_{m}(Y))V^{T}$$
where $U$ and $V$ are the left and right orthonormal matrices in the SVD of matrix $X$. In this case, the optimization
problem (24) reduces to
\begin{equation}\label{r27}
\min_{\sigma(Y)}\sum_{i=1}^{m}\Big\{(\sigma_{i}(Y)-\sigma_{i}(X))^{2}+\lambda\cdot \sum_{i=1}^{m}\rho_{a}(\sigma_{i}(Y))\Big\},
\end{equation}
which is consistent with Lemma 2, and since the function $g_{\lambda\mu}(\cdot)$ is monotone by Lemma 3.

So, there exists
$$\sigma_{i}(Y)=g_{\lambda\mu}(\sigma_{i}(X))$$
and
$$\sigma_{1}(Y)\geq\sigma_{2}(Y)\geq\cdots\geq\sigma_{m}(Y)\geq0$$
for $\sigma_{1}(X)\geq\sigma_{2}(X)\geq\cdots\geq\sigma_{m}(X)\geq0$.

Noting that $g_{\lambda\mu}(\cdot)$ acts only on the nonnegative part of the real line since all the $\sigma_{i}(X)$ are nonnegative.
Hence, we can see that the iterative singular value thresholding operator of the problem (24) has the form of (25).\\
This completes the proof.
\end{proof}

The iterative singular value thresholding operator $G_{\lambda\mu, P}$ simply applies the iterative thresholding operator $R_{\lambda\mu,P}$
defined in Section 2 to the singular values of $X$, and effectively shrinking them towards zero. This is the reason why we refer to this
transformation as the singular value thresholding operator for the non-convex fraction function. In some sense, the thresholding operator
$G_{\lambda\mu, P}$ is a straightforward extension of the iterative thresholding operator $R_{\lambda\mu, P}$. It is clear that
if many of the singular values of matrix $X$ are below the threshold value $t^{\ast}$, the rank of $G_{\lambda\mu, P}(X)$ may be considerably
lower than the rank of matrix $X$, like the iterative thresholding operator which is applied in RVMP to sparse outputs whenever some entries
of the input are below the threshold value $t^{\ast}$.

In the following process, we will show that the optimal solution to RTrARMP can also be expressed as an iterative singular
value thresholding operation.

For any positive parameters $\lambda>0$, $\mu>0$ and $Z\in \mathcal{R}^{m\times n}$, let
\begin{equation}\label{r28}
C_{\lambda}(X)=\|\mathcal{A}(X)-b\|_{2}^{2}+\lambda\cdot P(X),
\end{equation}
\begin{equation}\label{r29}
\begin{array}{llll}
C_{\mu}(X, Z)&=&\mu\Big[C_{\lambda}(X)-\|\mathcal{A}(X)-\mathcal{A}(Z)\|_{2}^{2}\Big]+\|X-Z\|_{F}^{2}
\end{array}
\end{equation}
and
$$B_{\mu}(X)=X+\mu \mathcal{A}^{\ast}(b-\mathcal{A}(X)).$$
Clearly, $C_{\mu}(X, X)=\mu C_{\lambda}(X)$.

\begin{theorem}\label{th4}
For any fixed $\lambda>0$ and $0<\mu<\frac{1}{\|\mathcal{A}\|_{2}^{2}}$. If $X^{\ast}$ is an optimal solution of
$\displaystyle\min_{X\in \mathcal{R}^{m\times n}}C_{\lambda}(X)$, then $X^{\ast}$ is also an optimal solution of $\displaystyle\min_{X\in
\mathcal{R}^{m\times n}}C_{\mu}(X,X^{\ast})$, that is
$$C_{\mu}(X^{\ast},X^{\ast})\leq C_{\mu}(X,X^{\ast})$$
for any $X\in \mathcal{R}^{m\times n}$.
\end{theorem}

\begin{proof}
By the definition of $C_{\mu}(X, Z)$, we have
\begin{eqnarray*}
C_{\mu}(X,X^{\ast})&=&\mu\Big[C_{\lambda}(X)-\|\mathcal{A}(X)-\mathcal{A}(X^{\ast})\|_{2}^{2}\Big]+\|X-X^{\ast}\|_{F}^{2}\\
&=&\mu\Big[\|\mathcal{A}(X)-b\|_{2}^{2}+\lambda P(X)\Big]+\|X-X^{\ast}\|_{F}^{2}\\
&&-\mu\|\mathcal{A}(X)-\mathcal{A}(X^{\ast})\|_{2}^{2}\\
&\geq&\mu\Big[\|\mathcal{A}(X)-b\|_{2}^{2}+\lambda\cdot P(X)\Big]\\
&=&\mu C_{\lambda}(X)\\
&\geq&\mu C_{\lambda}(X^{\ast})\\
&=&C_{\mu}(X^{\ast},X^{\ast})
\end{eqnarray*}
where the first inequality holds by the fact that
\begin{eqnarray*}
\|\mathcal{A}(X)-\mathcal{A}(X^{\ast})\|_{2}^{2}&=&\|Avec(X)-Avec(X^{\ast})\|_{2}^{2}\\
&\leq&\|A\|_{2}^{2}\cdot\|vec(X)-vec(X^{\ast})\|_{2}^{2}\\
&\leq&\|\mathcal{A}\|_{2}^{2}\cdot\|X-X^{\ast}\|_{F}^{2}.
\end{eqnarray*}
This completes the proof.
\end{proof}

\begin{theorem}\label{th5}
For any fixed $\lambda>0$, $\mu>0$ and matrix $Z\in \mathcal{R}^{m\times n}$,
then $\displaystyle\min_{X\in \mathcal{R}^{m\times n}}C_{\mu}(X,Z)$ is equivalent to
$$\min_{X\in \mathcal{R}^{m\times n}}\Big\{\|X-B_{\mu}(Z)\|_{F}^{2}+\lambda\mu P(X)\Big\}$$
where $B_{\mu}(Z)=Z-\mu \mathcal{A}^{\ast}\mathcal{A}(Z)+\mu \mathcal{A}^{\ast}(b)$.
\end{theorem}

\begin{proof}
In accordance with the definition, $C_{\mu}(X,Z)$ can be rewritten as
\begin{eqnarray*}
C_{\mu}(X,Z)&=&\|X-(Z-\mu \mathcal{A}^{\ast}\mathcal{A}(Z)+\mu \mathcal{A}^{\ast}(b))\|_{F}^{2}+\lambda\mu P(X)+\mu\|b\|_{2}^{2}\\
&&+\|Z\|_{F}^{2}-\mu\|\mathcal{A}(Z)\|_{2}^{2}-\|Z-\mu \mathcal{A}^{\ast}\mathcal{A}(Z)+\mu \mathcal{A}^{\ast}(b)\|_{F}^{2}\\
&=&\|X-B_{\mu}(Z)\|_{F}^{2}+\lambda\mu P(X)+\mu\|b\|_{2}^{2}+\|Z\|_{F}^{2}-\mu\|\mathcal{A}(Z)\|_{2}^{2}\\
&&-\|B_{\mu}(Z)\|_{F}^{2},
\end{eqnarray*}
which implies that $\displaystyle\min_{X\in \mathcal{R}^{m\times n}}C_{\mu}(X,Z)$ for any fixed $\lambda>0$, $\mu>0$ and matrix
$Z\in \mathcal{R}^{m\times n}$ is equivalent to
$$\min_{X\in \mathcal{R}^{m\times n}}\Big\{\|X-B_{\mu}(Z)\|_{F}^{2}+\lambda\mu P(X)\Big\}.$$
Therefore, $X^{\ast}$ is an optimal solution of $C_{\mu}(X,Z)$ if and only if $\sigma_{i}(X^{\ast})$ solves the problem
$$\min_{\sigma_{i}(X)\in \mathcal{R}^{+}}\Big\{(\sigma_{i}(X)-\sigma_{i}(B_{\mu}(Z)))^{2}+\lambda\mu \rho_{a}(\sigma_{i}(X))\Big\}.$$
This completes the proof.\\
\end{proof}

Theorem 4 shows that $X^{\ast}$ is an optimal solution to $\displaystyle\min_{X\in \mathcal{R}^{m\times n}}C_{\lambda}(X)$ if and only if
$X^{\ast}$ is an optimal solution of $\displaystyle\min_{X\in \mathcal{R}^{m\times n}}C_{\mu}(X,X^{\ast})$ with $Z=X^{\ast}$. Moreover,
combining Lemma 2 and Theorem 2, 3, 4, 5, we can immediately conclude that the thresholding representation of RTr-ARMP can be exactly given by
\begin{equation}\label{r30}
X^{\ast}=G_{\lambda\mu,P}(B_{\mu}(X^{\ast})).
\end{equation}
Assume that the SVD for matrix $B_{\mu}(X^{\ast})$ is $U^{\ast}\mbox{Diag}(\sigma(B_{\mu}(X^{\ast})))(V^{\ast})^{T}$ and
$\sigma(B_{\mu}(X^{\ast}))$ represents the singular value vector of matrix $B_{\mu}(X^{\ast})$ and $\sigma_{i}(B_{\mu}(X^{\ast}))$
represents the $i$-th largest entries of the singular value vector $\sigma(B_{\mu}(X^{\ast}))$, then
\begin{equation}\label{r31}
X^{\ast}=U^{\ast}\mathrm{Diag}(\sigma(B_{\mu}(X^{\ast})))(V^{\ast})^{T}
\end{equation}
which means that the singular values of the matrix $X^{\ast}$ satisfy

\begin{equation}\label{r32}
\sigma_{i}(X^{\ast})=
\left\{
    \begin{array}{ll}
      g_{\lambda\mu}(\sigma_{i}(B_{\mu}(X^{\ast}))), & \ \ {\mathrm{if}\ \sigma_{i}(B_{\mu}(X^{\ast}))>t^{\ast};} \\
      0, & \ \ {\mathrm{if}\ \sigma_{i}(B_{\mu}(X^{\ast}))\leq t^{\ast},}
    \end{array}
  \right.
\end{equation}
for $i=1,\cdots,m$, where $t^{\ast}$ is the threshold value and it is defined in (21).

We are now in the position to introduce the ISVTA which is proposed to solve RTrARMP.

Starting with $X^{0}$, inductively define for $k=0,1,2,\cdots,$
\begin{eqnarray*}
X^{k+1}&=&\displaystyle G_{\lambda\mu, P}(B_{\mu}(X^{k}))\\
&=&U^{k}\mathrm{Diag}(g_{\lambda\mu}(\sigma_{i}(B_{\mu}(X^{k}))))(V^{k})^{T}
\end{eqnarray*}
until a stopping criterion is reached, and
$$\sigma_{i}(X^{k+1})=
\left\{
    \begin{array}{ll}
      g_{\lambda\mu}(\sigma_{i}(B_{\mu}(X^{k}))), & \ \ {\mathrm{if}\ \sigma_{i}(B_{\mu}(X^{k}))>t^{\ast};} \\
      0, & \ \ {\mathrm{if}\ \sigma_{i}(B_{\mu}(X^{k}))\leq t^{\ast},}
    \end{array}
  \right.$$
where $U^{k}$ and $V^{k}$ are unitary matrices, the singular value vector $\sigma(B_{\mu}(X^{k}))$ comes
from the SVD of matrix $B_{\mu}(X^{k})$, $\sigma_{i}(B_{\mu}(X^{k}))$ represents the $i$-th largest
entries of the singular value vector $\sigma(B_{\mu}(X^{k}))$, and $t^{\ast}$ is the threshold value
which is defined in (21).

It is well known that the quantity of the solution of a regularization problem depends seriously on the setting of the regularization
parameter $\lambda>0$. However, the selection of the proper regularization parameters is a very hard problem. In most and general cases,
an ¡±trial and error¡± method, say, the cross-validation method, is still an accepted or even unique choice. Nevertheless, when some prior
information is known for a problem, it is realistic to set the regularization parameter more reasonably and intelligently.

To make it clear, we suppose that the matrix $X^{\ast}$ of rank $r$ is the optimal solution to RTrARMP, and the singular values of matrix
$B_{\mu}(X^{\ast})$ are denoted as
$$\sigma_{1}(B_{\mu}(X^{\ast}))\geq\sigma_{2}(B_{\mu}(X^{\ast}))\geq\cdots\geq\sigma_{m}(B_{\mu}(X^{\ast})).$$
Then by Theorem 3, the following inequalities hold
$$\sigma_{i}(B_{\mu}(X^{\ast}))>t^{\ast}\Leftrightarrow i\in\{1,2,\cdots,r\},$$
$$\sigma_{i}(B_{\mu}(X^{\ast}))\leq t^{\ast}\Leftrightarrow i\in\{r+1,r+2,\cdots,m\}$$
where $t^{\ast}$ is the threshold value which is defined in (21).

According to $t_{3}^{\ast}\leq t_{2}^{\ast}$, we have
\begin{equation}\label{r33}
\left\{
  \begin{array}{ll}
   \sigma_{r}(B_{\mu}(X^{\ast}))\geq t^{\ast}\geq t_{3}^{\ast}=\sqrt{\lambda\mu}-\frac{1}{2a}; \\
   \sigma_{r+1}(B_{\mu}(X^{\ast}))<t^{\ast}\leq t_{2}^{\ast}=\frac{\lambda\mu}{2}a,
  \end{array}
\right.
\end{equation}
which implies
\begin{equation}\label{r34}
\frac{2\sigma_{r+1}(B_{\mu}(X^{\ast}))}{a\mu}\leq\lambda\leq\frac{(2a\sigma_{r}(B_{\mu}(X^{\ast}))+1)^{2}}{4a^{2}\mu}.
\end{equation}

For convenience, we denote $\lambda_{1}$ and $\lambda_{2}$ the left and the right of above inequality respectively.
Above estimate helps to set optimal regularization parameter. A choice of $\lambda$ is
$$\lambda=\left\{
            \begin{array}{ll}
              \lambda_{1}, & \ \ {\mathrm{if}\ \lambda_{1}\leq\frac{1}{a^{2}\mu};} \\
              \lambda_{2},  &\ \ {\mathrm{if}\ \lambda_{1}>\frac{1}{a^{2}\mu}.}
            \end{array}
          \right.
$$
In practice, we approximate $\sigma_{i}((B_{\mu}(X^{\ast})))$ by $\sigma_{i}((B_{\mu}(X^{k})))$ in (34), say, we can take
\begin{equation}\label{r35}
\begin{array}{llll}
\lambda^{\ast}=\left\{
            \begin{array}{ll}
              \frac{2\sigma_{r+1}(B_{\mu}(X^{k}))}{a\mu},  & \ \ {\mathrm{if}\ \frac{2\sigma_{r+1}(B_{\mu}(X^{k}))}{a\mu}\leq\frac{1}{a^{2}\mu};} \\
              \frac{(2a\sigma_{r}(B_{\mu}(X^{k}))+1)^{2}}{4a^{2}\mu},  & \ \ {\mathrm{if}\ \frac{2\sigma_{r+1}(B_{\mu}(X^{k}))}{a\mu}>\frac{1}{a^{2}\mu}.}
            \end{array}
          \right.
\end{array}
\end{equation}
in applications. When doing so, the ISVTA will be adaptive and free from the choice of
regularization parameter. Noting that (35) is valid for any $\mu$ satisfying $0<\mu<\frac{1}{\|\mathcal{A}\|_{2}^{2}}$.
In general, we can take $\mu=\mu_{0}=\frac{1-\varepsilon}{\|\mathcal{A}\|_{2}^{2}}$ with any small $\varepsilon\in(0,1)$ below.

Incorporated with different parameter-setting strategies, defines different implementation schemes
of the ISVTA. For example, we can have the following\\

Scheme 1: $\mu=\mu_{0}$, $\lambda_{n}$ is chosen by cross-validation and $a=a_{0}$.\\

Scheme 2: $\mu=\mu_{0}$, $\lambda_{n}=\lambda^{\ast}$ defined in (34) and $a=a_{0}$.\\\\
There is one more thing needed to be mentioned that the threshold value $t^{\ast}=t_{2}^{\ast}$ when the parameter
$\lambda_{n}=\lambda_{1}$ and the threshold value $t^{\ast}=t_{3}^{\ast}$ when the parameter $\lambda_{n}=\lambda_{2}$ in Scheme 2.

Moreover, we also proved that the value of $\lambda$ can not be chosen too large. Indeed,
there exists $\bar{\lambda}>0$ such that the optimal solution of RTrARMP is equal to zero for any $\lambda>\bar{\lambda}$.
We should declare that the results derived in this following discussion are worst-case ones, implying that the kind of guarantees
we obtain are over-pessimistic for all possibilities. Before we embark to this discussion, we first discuss some useful results of
RTrARMP which play a key role in analysis.

\begin{lemma} \label{lemma5}
Let $X^{\ast}$ be the optimal solution of RTrARMP, $\mathrm{rank}(X^{\ast})=r$, $X^{\ast}=U^{\ast}\Sigma^{\ast}(V^{\ast})^{T}$
and $\mathcal{A}(X^{\ast})=Avec(X^{\ast})= A(V^{\ast}\otimes U^{\ast})vec(\Sigma^{\ast})$. Then the columns in matrix
$B^{\ast}=A(V^{\ast}\otimes U^{\ast})$ corresponding to the support of vector $y^{\ast}=vec(\Sigma^{\ast})$ are linearly
independent.
\end{lemma}

\begin{proof}
By the optimality of $X^{\ast}$ and $X^{\ast}=U^{\ast}\Sigma^{\ast}(V^{\ast})^{T}$, we have
$$\mathcal{A}(X^{\ast})=Avec(X^{\ast})=A(V^{\ast}\otimes U^{\ast})vec(\Sigma^{\ast})=B^{\ast}y^{\ast}$$
where $A=(vec(\mathcal{A}_{1}), vec(\mathcal{A}_{2}),\cdots, vec(\mathcal{A}_{d}))^{T}\in \mathcal{R}^{d\times mn}$,
$B^{\ast}=A(V^{\ast}\otimes U^{\ast})\in \mathcal{R}^{r\times mn}$, $y^{\ast}=vec(\Sigma^{\ast})\in \mathcal{R}^{mn}$,
and $\mathrm{rank}(X^{\ast})=\|y^{\ast}\|_{0}=r$.

Without loss of generality, we assume
$$y^{\ast}=(\sigma_{1}(X^{\ast}), \sigma_{2}(X^{\ast}), \cdots, \sigma_{r}(X^{\ast}), 0,\cdots, 0)^{T},$$
and
$$\sigma_{1}(X^{\ast})\geq\sigma_{2}(X^{\ast})\geq\cdots\geq \sigma_{r}(X^{\ast})>0.$$
Let
$$z^{\ast}=(\sigma_{1}(X^{\ast}), \sigma_{2}(X^{\ast}), \cdots, \sigma_{r}(X^{\ast}))^{T}$$
and $C\in \mathcal{R}^{d\times r}$ be the sub-matrix of $B^{\ast}$, whose columns in matrix $B^{\ast}$ corresponding
to $z^{\ast}$.

Define a function $g:\mathcal{R}^{d\times r}\mapsto \mathcal{R}$ by
\begin{equation}\label{r36}
g(z^{\ast})=\|Cz^{\ast}-b\|_{2}^{2}+\lambda\cdot P(z^{\ast}).
\end{equation}
We have
\begin{equation}\label{37}
\begin{array}{llll}
f(X^{\ast})&=&\displaystyle\|\mathcal{A}(X^{\ast})-b\|_{2}^{2}+\lambda\cdot P(X^{\ast})\\
&=&\displaystyle\|Avec(X^{\ast})-b\|_{2}^{2}+\lambda\cdot P(X^{\ast})\\
&=&\displaystyle\|Avec(U^{\ast}\Sigma^{\ast}(V^{\ast})^{T})-b\|_{2}^{2}+
\lambda\cdot P(U^{\ast}\Sigma^{\ast}(V^{\ast})^{T})\\
&=&\displaystyle\|A(V^{\ast}\otimes U^{\ast})vec(\Sigma^{\ast})-b\|_{2}^{2}+
\lambda\cdot P(\Sigma^{\ast})\\
&=&\displaystyle\|B^{\ast}y^{\ast}-b\|_{2}^{2}+\lambda\cdot P(y^{\ast})\\
&=&\displaystyle\|Cz^{\ast}-b\|_{2}^{2}+\lambda\cdot P(z^{\ast})\\
&=&\displaystyle\mathrm{g}(z^{\ast}).
\end{array}
\end{equation}
Since $\sigma_{i}(X^{\ast})>0$, $i=1, 2, \cdots, r$, the function $g$ is continuously differentiable at $z^{\ast}$.
Moreover, in a neighborhood of $z^{\ast}$,
\begin{equation}\label{r38}
\begin{array}{llll}
g(z^{\ast})=\displaystyle f(X^{\ast})
&\leq&\displaystyle\min\Big\{f(X)|\sigma_{i}(X)=0,i=r+1,r+2,\cdots,m\Big\}\\
&=&\displaystyle\min\Big\{g(z)|z\in \mathcal{R}^{r}\Big\},
\end{array}
\end{equation}
which implies that $z^{\ast}$ is a local minimizer of the function $g$. Hence, the second order necessary condition for
$$\min_{z\in \mathcal{R}^{r}}g(z)$$
holds at $z^{\ast}$.
The second order necessary condition at $z^{\ast}$ gives that the matrix
$$2C^{T}C-\mathrm{Diag}\bigg(\frac{\lambda a^{2}}{(a\sigma_{i}(X^{\ast})+1)^{3}}\bigg), \ \ \ i=1, 2, \cdots, r$$
is positive semi-definite, and the matrix $M=\mathrm{Diag}\Big(\frac{\lambda a^{2}}{(a\sigma_{i}(X^{\ast})+1)^{3}}\Big)$
is positive, the matrix $C^{T}C$ must be positive definite. Hence the columns of $C$ must be linearly independent.\\
This completes the proof.
\end{proof}

\begin{theorem} \label{th6}
Let $X^{\ast}$ be the optimal solution of RTrARMP and $\mathrm{rank}(X^{\ast})=r$. Then the following statements hold.\\

(1) If $\lambda>\|b\|_{2}^{2}$, then
$$\|\sigma(X^{\ast})\|_{\infty}\leq\frac{\|b\|_{2}^{2}}{a(\lambda-\|b\|_{2}^{2})};$$

(2) Denote by $\bar{\lambda}$ the constant
$$\|b\|_{2}^{2}+\frac{\|A^{T}b\|_{2}+\sqrt{\|A^{T}b\|_{2}+2a\|b\|_{2}^{2}\|A^{T}b\|_{2}}}{a}.$$
Then for all $\lambda\geq\bar{\lambda}$, $X^{\ast}=0$.
\end{theorem}

\begin{proof}
(1) Let $X^{\ast}$ be the optimal solution of RTrARMP. Then we have
$$f(X^{\ast})=\|\mathcal{A}(X^{\ast})-b\|_{2}^{2}+\lambda\cdot P(X^{\ast})\leq f(0)=\|b\|_{2}^{2}.$$
Hence $\lambda\cdot P(X^{\ast})\leq \|b\|_{2}^{2}$, which implies that
$$\frac{a\|\sigma(X^{\ast})\|_{\infty}}{a\|\sigma(X^{\ast})\|_{\infty}+1}\leq\frac{\|b\|_{2}^{2}}{\lambda}.$$
If $\lambda>\|b\|_{2}^{2}$, then
$$\|\sigma(X^{\ast})\|_{\infty}\leq\frac{\|b\|_{2}^{2}}{a(\lambda-\|b\|_{2}^{2})}.$$

(2) By Lemma 5, the first order necessary condition for
$$\min_{z\in \mathcal{R}^{r}}\mathrm{g}(z)$$
at $z^{\ast}$ gives
\begin{equation}\label{r39}
2C^{T}(Cz^{\ast}-b)+\frac{\lambda a}{(az^{\ast}+1)^{2}}=0.
\end{equation}
Multiplying by $(z^{\ast})^{T}$ both sides of equality above yield
$$2(z^{\ast})^{T}C^{T}Cz^{\ast}-2(z^{\ast})^{T}C^{T}b+(z^{\ast})^{T}\frac{\lambda a}{(az^{\ast}+1)^{2}}=0.$$

Because the columns of $C$ are linearly independent, $C^{T}C$ is positive definite (see the proof of Lemma 5), and hence
$$-2(z^{\ast})^{T}C^{T}b+(z^{\ast})^{T}\frac{\lambda a}{(az^{\ast}+1)^{2}}<0,$$
equivalently,
\begin{equation}\label{r40}
\sum_{i=1}^{r}\bigg(\frac{\lambda az_{i}^{\ast}}{(az_{i}^{\ast}+1)^{2}}-2(C^{T}b)_{i}z_{i}^{\ast}\bigg)<0.
\end{equation}
Since
$$\lambda>\|b\|_{2}^{2}+\frac{\|A^{T}b\|_{2}+\sqrt{\|A^{T}b\|_{2}+2a\|b\|_{2}^{2}\|A^{T}b\|_{2}}}{a},$$
we obtain
\begin{equation}\label{r41}
a\lambda^{2}-2(a\|b\|_{2}^{2}+\|A^{T}b\|_{2})\lambda+a\|b\|_{2}^{4}\geq0,
\end{equation}
which implies that
\begin{equation}\label{r42}
\frac{a(\lambda-\|b\|_{2}^{2})}{\lambda}\geq2\|A^{T}b\|_{2}.
\end{equation}

Together with
$$\frac{\lambda a}{(az_{i}^{\ast}+1)^{2}}\geq\frac{a(\lambda-\|b\|_{2}^{2})}{\lambda}$$
and
\begin{equation}\label{r43}
\begin{array}{llll}
|(C^{T}b)_{i}|&\leq&\displaystyle\|(B^{\ast})^{T}b\|_{\infty}\\
&\leq&\|(B^{\ast})^{T}b\|_{2}\\
&=&\|(V^{\ast}\otimes U^{\ast})^{T}A^{T}b\|_{2}\\
&=&\|A^{T}b\|_{2},
\end{array}
\end{equation}
we obtain that
\begin{equation}\label{r44}
\frac{\lambda a}{(1+az_{i}^{\ast})^{2}}-2|(C^{T}b)_{i}|\geq0.
\end{equation}

Hence, for any $i\in \{1, 2, \cdots, r\}$,
$$\frac{\lambda az_{i}^{\ast}}{(1+az_{i}^{\ast})^{2}}-2((C^{T}b)_{i})z_{i}^{\ast}\geq0,$$
which is a contradiction with (40), as claimed.\\
This completes the proof.
\end{proof}

\section{Convergence analysis for ISVTA} \label{applications-sec}

In this section, we mainly study the convergence of ISVTA to a stationary point of the iteration (30) under some certain conditions.
\begin{theorem} \label{th7}
Let $\{X^{k}\}$ be the sequence generated by the ISVTA with the step size $\mu$ satisfying $0<\mu<\frac{1}{\|\mathcal{A}\|_{2}^{2}}$. Then
\begin{description}
  \item[$\mathrm{1)}$] The sequence $C_{\lambda}(X^{k})$ is decreasing.
  \item[$\mathrm{2)}$] $\{X^{k}\}$ is asymptotically regular, i.e., $\lim_{k\rightarrow\infty}\|X^{k+1}-X^{k}\|_{2}=0$.
  \item[$\mathrm{3)}$] $\{X^{k}\}$ converges to a stationary point of the iteration (30).
\end{description}
\end{theorem}

\begin{proof}
1) By the proof of Theorem 5, we have
$$C_{\mu}(X^{k+1}, X^{k})=\min_{X\in \mathcal{R}^{m\times n}} C_{\mu}(X, X^{k}).$$
Combined with the definition of $C_{\lambda}(X)$ and $C_{\mu}(X, Z)$, we have
$$C_{\lambda}(X^{k+1})=\frac{1}{\mu}[C_{\mu}(X^{k+1}, X^{k})-\|X^{k+1}-X^{k}\|_{F}^{2}]+\|\mathcal{A}(X^{k+1})-\mathcal{A}(X^{k})\|_{2}^{2}.$$
Since $0<\mu<\frac{1}{\|\mathcal{A}\|_{2}^{2}}$, we get
\begin{equation}\label{45}
\begin{array}{llll}
C_{\lambda}(X^{k+1})&=&\frac{1}{\mu}[C_{\mu}(X^{k+1}, X^{k})-\|X^{k+1}-X^{k}\|_{F}^{2}]+\|\mathcal{A}(X^{k+1})-\mathcal{A}(X^{k})\|_{2}^{2}\\
&\leq&\frac{1}{\mu}[C_{\mu}(X^{k}, X^{k})-\|X^{k+1}-X^{k}\|_{F}^{2}]+\|\mathcal{A}(X^{k+1})-\mathcal{A}(X^{k})\|_{2}^{2}\\
&=&C_{\lambda}(X^{k})-\frac{1}{\mu}\|X^{k+1}-X^{k}\|_{F}^{2}+\|\mathcal{A}(X^{k+1})-\mathcal{A}(X^{k})\|_{2}^{2}\\
&\leq&C_{\lambda}(X^{k}).
\end{array}
\end{equation}
That is, the sequence $\{X^{k}\}$ is a minimization sequence of function $C_{\lambda}(X)$, and $C_{\lambda}(X^{k+1})\leq C_{\lambda}(X^{k})$ for all $k\geq0$.\\

2) Let $\theta=1-\mu\|\mathcal{A}\|_{2}^{2}$. Then $\theta\in(0, 1)$ and
\begin{equation}\label{46}
\mu\|\mathcal{A}(X^{k+1}-X^{k})\|_{2}^{2}\leq(1-\theta)\|X^{k+1}-X^{k}\|_{F}^{2}.
\end{equation}
By (45), we have
\begin{equation}\label{47}
\frac{1}{\mu}\|X^{k+1}-X^{k}\|_{F}^{2}-\|\mathcal{A}(X^{k+1})-\mathcal{A}(X^{k})\|_{2}^{2}\leq C_{\lambda}(X^{k})-C_{\lambda}(X^{k+1}).
\end{equation}
Combing (46) and (47), we get
\begin{eqnarray*}
\sum_{k=1}^{N}\{\|X^{k+1}-X^{k}\|_{F}^{2}\}&\leq&\frac{1}{\theta}\sum_{k=1}^{N}\{\|X^{k+1}-X^{k}\|_{F}^{2}\}\\
&&-\frac{1}{\theta}\sum_{k=1}^{N}\{\mu\|\mathcal{A}(X^{k+1})-\mathcal{A}(X^{k})\|_{2}^{2}\}\\
&\leq&\frac{\mu}{\theta}\sum_{k=1}^{N}\{C_{\lambda}(X^{k})-C_{\lambda}(X^{k+1})\}\\
&=&\frac{\mu}{\theta}(C_{\lambda}(X^{1})-C_{\lambda}(X^{N+1}))\\
&\leq&\frac{\mu}{\theta}C_{\lambda}(X^{1}).
\end{eqnarray*}
Thus, the series $\sum_{k=1}^{\infty}\|X^{k+1}-X^{k}\|_{F}^{2}$ is convergent, which implies that
$$\|X^{k+1}-X^{k}\|_{F}^{2}\rightarrow 0 \ \ \mathrm{as}\ \ k\rightarrow\infty.$$

3) Denote
$$T_{\lambda, \mu}(Z, X)=\|Z-B_{\mu}(X)\|_{F}^{2}+\lambda\mu\cdot P(Z)$$
and let
$$D_{\lambda, \mu}(X)=T_{\lambda, \mu}(X, X)-\min_{Z\in \mathcal{R}^{m\times n}}T_{\lambda, \mu}(Z, X)$$
(similar to [30] in sparse signal recovery problems). Then
$$D_{\lambda, \mu}(X)\geq 0$$
and by (30), we have
$$D_{\lambda, \mu}(X)=0\ \mathrm{if}\ \mathrm{and}\ \mathrm{only}\ \mathrm{if}\ X=G_{\lambda\mu,P}(B_{\mu}(X)).$$
Assume that $X^{\ast}$ is a limit point of $\{X^{k}\}$, then there exists a subsequence of $\{X^{k}\}$, which is denoted as $\{X^{k_{j}}\}$ such that
$X^{k_{j}}\rightarrow X^{\ast}$ as $j\rightarrow \infty$. Since the iterative scheme $$X^{k_{j+1}}=G_{\lambda\mu,P}(B_{\mu}(X^{k_{j}})),$$
we have
\begin{eqnarray*}
D_{\lambda, \mu}(X^{k_{j}})&=&T_{\lambda, \mu}(X^{k_{j}}, X^{k_{j}})-T_{\lambda, \mu}(X^{k_{j+1}}, X^{k_{j}})\\
&=&\lambda\mu(P_{a}(X^{k_{j}})-P_{a}(X^{k_{j+1}}))-\|X^{k_{j+1}}-X^{k_{j}}\|_{F}^{2}\\
&&+2\langle\mu \mathcal{A}^{\ast}(b-\mathcal{A}(X^{k_{j}})), X^{k_{j+1}}-X^{k_{j}}\rangle,
\end{eqnarray*}
which implies that
\begin{equation}\label{48}
\begin{array}{llll}
\lambda P_{a}(X^{k_{j}})-\lambda P_{a}(X^{k_{j+1}})&=&\frac{1}{\mu}\|X^{k_{j+1}}-X^{k_{j}}\|_{F}^{2}+\frac{1}{\mu}D_{\lambda, \mu}(X^{k_{j}})\\
&&-2\langle \mathcal{A}^{\ast}(b-\mathcal{A}(X^{k_{j}})), X^{k_{j+1}}-X^{k_{j}}\rangle.
\end{array}
\end{equation}
By (48), it follows that
\begin{eqnarray*}
C_{\lambda}(X^{k_{j}})-C_{\lambda}(X^{k_{j+1}})&=&\|\mathcal{A}(X^{k_{j}})-b\|_{2}^{2}+\lambda\cdot P_{a}(X^{k_{j}})-\|\mathcal{A}(X^{k_{j+1}})-b\|_{2}^{2}\\
&&-\lambda P_{a}(X^{k_{j+1}}))\\
&=&\frac{1}{\mu}\|X^{k_{j+1}}-X^{k_{j}}\|_{F}^{2}-\|\mathcal{A}(X^{k_{j}}-X^{k_{j+1}})\|_{2}^{2}\\
&&+\frac{1}{\mu}D_{\lambda, \mu}(X^{k_{j}})\\
&\geq&(\frac{1}{\mu}-\|\mathcal{A}\|_{2}^{2})\|X^{k_{j}}-X^{k_{j+1}}\|_{F}^{2}+\frac{1}{\mu}D_{\lambda, \mu}(X^{k_{j}}).
\end{eqnarray*}
Since $0<\mu<\frac{1}{\|\mathcal{A}\|_{2}^{2}}$, we get
$$D_{\lambda, \mu}(X^{k_{j}})\leq\mu(C_{\lambda}(X^{k_{j}})-C_{\lambda}(X^{k_{j+1}})).$$
Combining the following fact that
$$C_{\lambda}(X^{k_{j}})-C_{\lambda}(X^{k_{j+1}})\rightarrow 0\ \ as\ j\rightarrow\infty,$$
we have
$$D_{\lambda, \mu}(X^{\ast})=0.$$
This implies that the limit point $X^{\ast}$ of the sequence $\{X^{k}\}$ satisfies the equation
$$X^{\ast}=G_{\lambda\mu,P}(B_{\mu}(X^{\ast})).$$
This completes the proof.
\end{proof}

\section{Numerical experiments} \label{applications-sec}
In this section, we first present numerical results of ISVTA for matrix completion problems, and then compare it with some state-of-art methods
(singular value thresholding algorithm (SVTA) and singular value projection algorithm (SVPA) respectively proposed in [11] and [31]) for image
inpainting problems. Numerical experiments on matrix completion problems show that our method performs powerful in finding a low-rank matrix and
the numerical experiments about image inpainting problems show that our algorithm has better performances than SVTA and SVPA. Among all of the
experiments, differing from the Scheme 2, we set $u=u_{0}=5$, and
$$
\lambda^{\ast}=\left\{
            \begin{array}{ll}
              \frac{2\sigma_{r+1}(B_{\mu}(X^{k}))}{a},  & \ \ {\mathrm{if}\ \frac{2\sigma_{r+1}(B_{\mu}(X^{k}))}{a}\leq\frac{1}{a^{2}\mu};} \\
              \frac{(2a\sigma_{r}(B_{\mu}(X^{k}))+1)^{2}}{4a^{2}},  & \ \ {\mathrm{if}\ \frac{2\sigma_{r+1}(B_{\mu}(X^{k}))}{a}>\frac{1}{a^{2}\mu}.}
            \end{array}
          \right.
$$
The experiments are all conducted on a personal computer ( Intel(R) Core (TM) i5-6200U with CPU at 2.30GHz, 8.0 GB RAM under 64-bit Ubuntu system)
with MATLAB 8.0 programming platform (R2012b).

\subsection{Completion of random matrices}
In this subsection, we carry out a series of experiments to demonstrate the performance of the ISVTA. All the experiments here are conducted
by applying our algorithm to a typical ARMP, i.e., random low rank matrix completion problems. We generate $m\times n$ matrices $M$
of rank $r$ as the matrix products of two low rank matrices $M_{1}$ and $M_{2}$ where $M_{1}\in \mathcal{R}^{m\times r}$, $M_{2}\in \mathcal{R}^{r\times n}$
are generated with independent identically distributed Gaussian entries and the matrix $M=M_{1}M_{2}$ has rank at most $r$. The set of observed
entries $\Omega$ is sampled uniformly at random among all sets of cardinality $s$. We denote the following quantities and they help to quantify the difficulty
of the low rank matrix recovery problems 
\begin{itemize}
  \item \emph{Sampling ratio}: $\mathrm{SR}=s/mn$.
  \item \emph{Freedom ration}: $\mathrm{FR}=s/r(m+n-r)$, which is the ratio between the number of sampled entries and the
'true dimensionality' of a $m\times n$ matrix of rank $r$, and it is a good quantity as the information oversampling ratio. 
\end{itemize}
The stopping criterion is usually as following
$$\frac{\|X_{k}-X_{k-1}\|}{\|X_{k}\|}\leq \mathrm{Tol}$$
where $X_{k}$ and $X_{k-1}$ are numerical results from two continuous iterative steps and $\mathrm{Tol}$ is a given small number.
In addition, we measure the accuracy of the generated solution $X_{opt}$ of our algorithms by the relative error ($\mathrm{RE}$)
defined as following
$$\mathrm{RE}=\frac{\|X_{opt}-M\|_{F}}{\|M\|_{F}}.$$
In these experiments, we test ISVTA on random low-rank matrix completion problems with different parameter '$a$',
and set $a=1,3,5,7,30,100$, respectively.
\begin{table}[!hbp]\scriptsize
\centering
\begin{tabular}{|c||l|l|l|l|l|l|}\hline
Problem&\multicolumn{2}{c}{a=1}&\multicolumn{2}{|c}{a=3}&\multicolumn{2}{|c|}{a=5}\\
\hline
(n,\,rank,\,FR)&RE&Time&RE&Time&RE&Time\\
\hline
$(100,\,11,\,1.9240)$&9.99e-05& 2.58& 9.93e-05& 1.82& 9.99e-05& 1.87\\
\hline
$(100,\,12,\,1.7730)$&9.97e-05& 2.23& 9.93e-05& 3.15& 9.99e-05& 3.29\\
\hline
$(100,\,13,\,1.6454)$&9.99e-05& 5.08&  9.99e-05& 2.30& 9.99e-05& 3.47\\
\hline
$(100,\,14,\,1.5361)$&9.96e-05& 3.49& 9.98e-05& 4.01& 9.97e-05& 3.86\\
\hline
$(100,\,15,\,1.4414)$&9.95e-05& 3.76& 9.98e-05& 4.26& 9.98e-05& 4.51\\
\hline
$(100,\,16,\,1.3587)$&9.98e-05& 5.87& 9.97e-05& 6.16& 9.99e-05& 13.06\\
\hline
$(100,\,17,\,1.2858)$&9.97e-05& 6.82& 9.97e-05& 7.41& 9.99e-05& 10.98\\
\hline
$(100,\,18,\,1.2210)$&9.99e-05& 11.43& 9.97e-05& 10.04& 9.99e-05& 11.67\\
\hline
$(100,\,19,\,1.1631)$&9.99e-05& 17.91& 9.99e-05& 17.60& 9.99e-05& 51.31\\
\hline
$(100,\,20,\,1.1111)$&9.99e-05& 35.36& 9.99e-05& 37.01& 9.99e-05& 40.19\\
\hline
$(100,\,21,\,1.0641)$&9.99e-05& 110.24& 9.99e-05& 118.52& 9.99e-05& 111.66\\
\hline
$(100,\,22,\,1.0215)$&---& ---& ---& ---& ---& ---\\
\hline
\end{tabular}
\caption{\scriptsize Numerical results of ISVTA for matrix completion problems with different rank and FR but fixed $n$, SR=0.40.}\label{table1}
\end{table}

\begin{table}[!hbp]\scriptsize
\centering
\begin{tabular}{|c||l|l|l|l|l|l|}\hline
Problem&\multicolumn{2}{c}{a=7}&\multicolumn{2}{|c}{a=30}&\multicolumn{2}{|c|}{a=100}\\
\hline
(n,\,rank,\,FR)&RE&Time&RE&Time&RE&Time\\
\hline
$(100,\,11,\,1.9240)$&9.88e-05& 1.78& 9.91e-05& 1.92& 9.96e-05& 1.74\\
\hline
$(100,\,12,\,1.7730)$&9.96e-05& 2.76& 9.95e-05& 2.05& 9.99e-05& 2.17\\
\hline
$(100,\,13,\,1.6454)$&9.98e-05& 3.96& 9.96e-05& 2.61& 9.97e-05& 3.83\\
\hline
$(100,\,14,\,1.5361)$&9.98e-05& 5.12& 9.99e-05& 4.63& 9.98e-05& 4.02 \\
\hline
$(100,\,15,\,1.4414)$&9.96e-05& 3.88& 9.96e-05& 4.40& 9.99e-05& 4.84\\
\hline
$(100,\,16,\,1.3587)$&9.99e-05& 6.68& 9.98e-05& 5.62& 9.99e-05& 7.58\\
\hline
$(100,\,17,\,1.2858)$&9.99e-05& 11.70& 9.98e-05& 11.74& 9.99e-05& 8.83\\
\hline
$(100,\,18,\,1.2210)$&9.99e-05& 10.31& 9.98e-05& 14.29& 9.99e-05& 29.18\\
\hline
$(100,\,19,\,1.1631)$&9.99e-05& 17.86& 9.99e-05& 45.61& 9.98e-05& 66.40\\
\hline
$(100,\,20,\,1.1111)$&9.99e-05& 42.13& 9.99e-05& 173.38& 1.00e-04& 262.07\\
\hline
$(100,\,21,\,1.0641)$&9.99e-05& 420.64& ---& --- & ---& ---\\
\hline
$(100,\,22,\,1.0215)$&---& ---& ---& ---& ---& ---\\
\hline
\end{tabular}
\caption{\scriptsize Numerical results of ISVTA for matrix completion problems with different rank and FR but fixed $n$, SR=0.40.}\label{table2}
\end{table}

\begin{table} [htbp]
\centering
\begin{tabular}{|c||l|l|l|l|l|l|}\hline
Problem&\multicolumn{2}{c}{a=1}&\multicolumn{2}{|c}{a=3}&\multicolumn{2}{|c|}{a=5}\\
\hline
(n,\,rank,\,FR)&RE&Time&RE&Time&RE&Time\\
\hline
$(100,\,11,\,1.9240)$&9.99e-05& 2.58& 9.87e-05& 1.68& 9.99e-05& 1.87\\
\hline
$(200,\,11,\,3.7392)$&9.92e-05& 2.54& 9.67e-05& 5.89& 9.66e-05& 2.20\\
\hline
$(300,\,11,\,5.5564)$&9.53e-05& 6.37& 9.62e-05& 7.42& 9.88e-05& 6.77\\
\hline
$(400,\,11,\,7.3741)$&9.87e-05& 9.17& 9.16e-05& 9.22& 9.86e-05& 9.99\\
\hline
$(500,\,11,\,9.1920)$&9.15e-05& 11.41& 9.69e-05& 13.52& 9.45e-05& 12.14\\
\hline
$(600,\,11,\,11.0100)$&9.69e-05& 15.52& 9.74e-05& 18.73& 9.83e-05& 15.31\\
\hline
$(700,\,11,\,12.8281)$&9.08e-05& 21.32& 9.15e-05& 24.09& 9.30e-05& 21.27\\
\hline
$(800,\,11,\,14.6461)$&9.94e-05& 30.20& 9.85e-05& 33.95& 9.28e-05& 30.23\\
\hline
$(900,\,11,\,16.4643)$&8.92e-05& 40.34& 9.38e-05& 47.87& 9.64e-05& 39.02\\
\hline
$(1000,\,11,\,18.2824)$&9.30e-05& 57.75& 9.65e-05& 62.60& 9.58e-05& 55.77\\
\hline
$(1100,\,11,\,20.1005)$&9.29e-05& 74.22& 9.20e-05& 81.80& 9.30e-05& 72.30\\
\hline
$(1200,\,11,\,21.9186)$& 9.20e-05& 90.74& 9.26e-05& 100.01& 9.97e-05& 87.63\\
\hline
\end{tabular}
\caption{\scriptsize Numerical results of ISVTA for matrix completion problems with different $n$
and FR but fixed rank, SR=0.40.}\label{table3}
\end{table}

\begin{table}[htbp]
\centering
\begin{tabular}{|c||l|l|l|l|l|l|}\hline
Problem&\multicolumn{2}{|c}{a=7}&\multicolumn{2}{|c}{a=30}&\multicolumn{2}{|c|}{a=100}\\
\hline
(n,\,rank,\,FR)&RE&Time&RE&Time&RE&Time\\
\hline
$(100,\,11,\,1.9240)$&9.88e-05& 1.78& 9.91-05& 1.92& 9.96e-05& 1.74\\
\hline
$(200,\,11,\,3.7392)$&9.89e-05& 2.27& 9.75e-05& 2.04& 9.97e-05& 1.90\\
\hline
$(300,\,11,\,5.5564)$&9.49e-05& 5.94&  9.86e-05& 6.31& 9.93e-05& 5.51\\
\hline
$(400,\,11,\,7.3741)$&9.88e-05& 8.97& 9.46e-05& 8.24& 9.86e-05& 7.75\\
\hline
$(500,\,11,\,9.1920)$&9.99e-05& 11.47&  9.27e-05& 11.04& 9.37e-05& 11.65\\
\hline
$(600,\,11,\,11.0100)$&9.33e-05& 15.96& 9.56e-05& 15.38& 9.96e-05& 12.63\\
\hline
$(700,\,11,\,12.8281)$&9.68e-05& 20.72& 9.22e-05& 20.24& 9.26e-05& 20.22\\
\hline
$(800,\,11,\,14.6461)$&9.24e-05& 30.57& 9.30e-05& 27.85& 9.44e-05& 27.69\\
\hline
$(900,\,11,\,16.4643)$&9.53e-05& 38.21& 9.89e-05& 38.57& 9.96e-05& 37.80\\
\hline
$(1000,\,11,\,18.2824)$&9.28e-05& 54.28& 9.62e-05& 53.71& 9.11e-05& 54.92\\
\hline
$(1100,\,11,\,20.1005)$&9.34e-05& 72.29& 8.98e-05& 77.04& 8.72e-05& 70.93\\
\hline
$(1200,\,11,\,21.9186)$&9.85e-05& 87.62& 9.17e-05& 87.84& 9.64e-05& 86.26\\
\hline
\end{tabular}
\caption{\scriptsize Numerical results of ISVTA for matrix completion problems with different $n$
and FR but fixed rank, SR=0.40.}\label{table4}
\end{table}
Table 1, 2 report the numerical results of ISVTA for the random low-rank matrix completion problems with $\mathrm{SR}=0.40$ when we fix $n=100$
and vary the rank of the matrix $M$ from $11$ to $22$ with step size $1$. Table 3, 4 present the numerical results of ISVTA in the case
where the rank is fixed to $11$ and $n$ is varied from $100$ to $1200$ with step size $100$. By the performances of ISVTA for completion
of random low rank matrices compared with different $a$ and $\mathrm{FR}$. Table 1, 2, 3, 4 show that for known rank scheme, our method performs
powerful in finding a low-rank matrix, and $a=1$ is the optimal strategy when $\mathrm{FR}$ is close to 1.

\subsection{Image inpainting}
In this subsection, we demonstrate performances of ISVTA on image inpainting problems. The ISVTA is tested on some medical grace images
($255\times192$ Brain angiography image (BAI), $395\times549$ Hand angiography image (HAI) and $419\times400$ Intracranial venous image (IVI)).
We use the SVD to obtain their approximated low-rank images with rank $r=30, 40, 30$, respectively. Numerical results of ISVTA for theses 
low-rank image inpainting problems are reported in Table 5, 6, 7, 8.

\begin{figure}[htbp]
  \centering
  \begin{minipage}[t]{0.4\linewidth}
  \centering
  \includegraphics[width=1\textwidth]{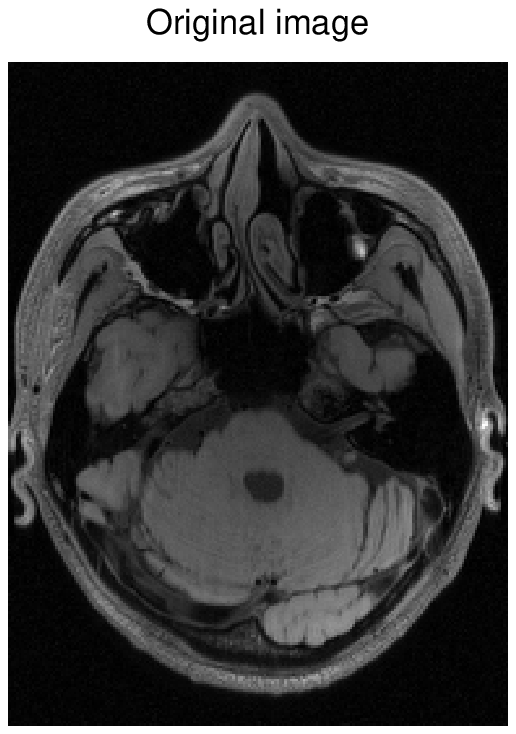}
  \end{minipage}
  \begin{minipage}[t]{0.4\linewidth}
  \centering
  \includegraphics[width=1\textwidth]{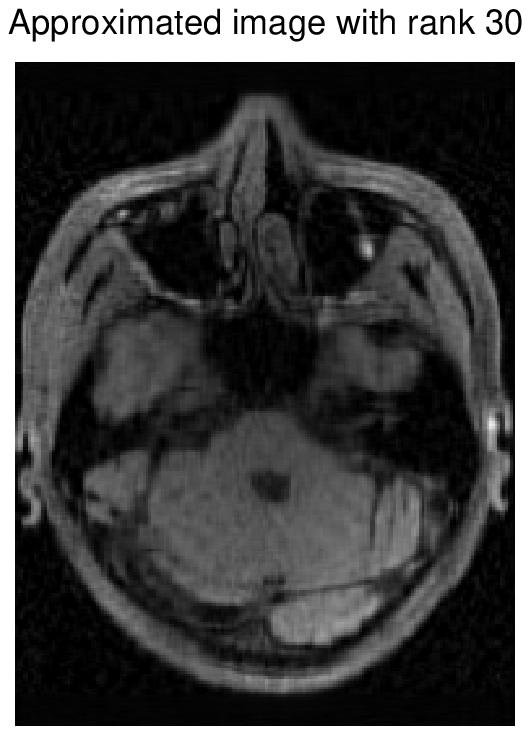}
  \end{minipage}
  \caption{Original $255\times192$ BAI and its approximation with rank 30.} \label{figure3}
\end{figure}

\begin{figure}[h]
  \centering
  \begin{minipage}[t]{0.4\linewidth}
  \centering
  \includegraphics[width=1\textwidth]{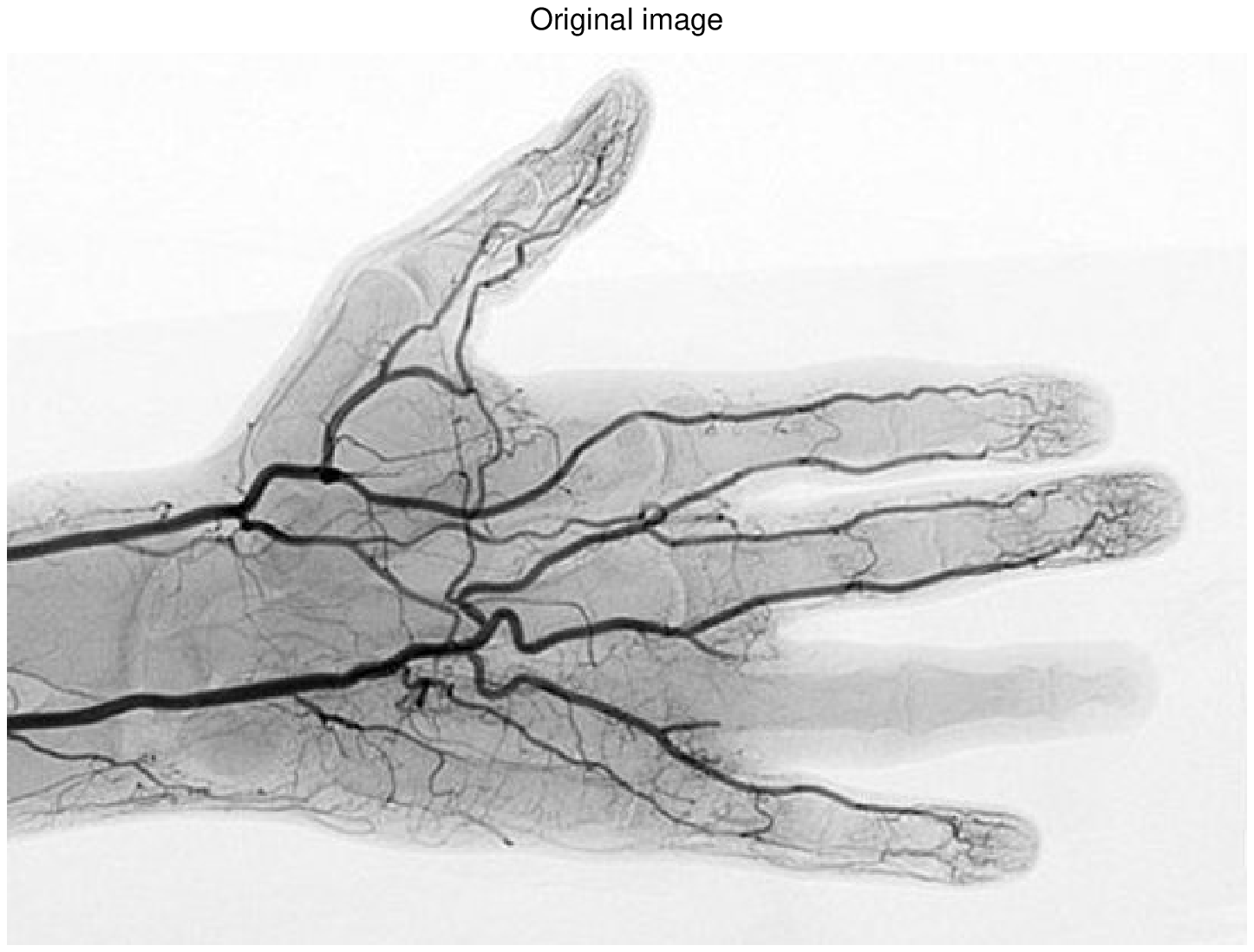}
  \end{minipage}
  \begin{minipage}[t]{0.4\linewidth}
  \centering
  \includegraphics[width=1\textwidth]{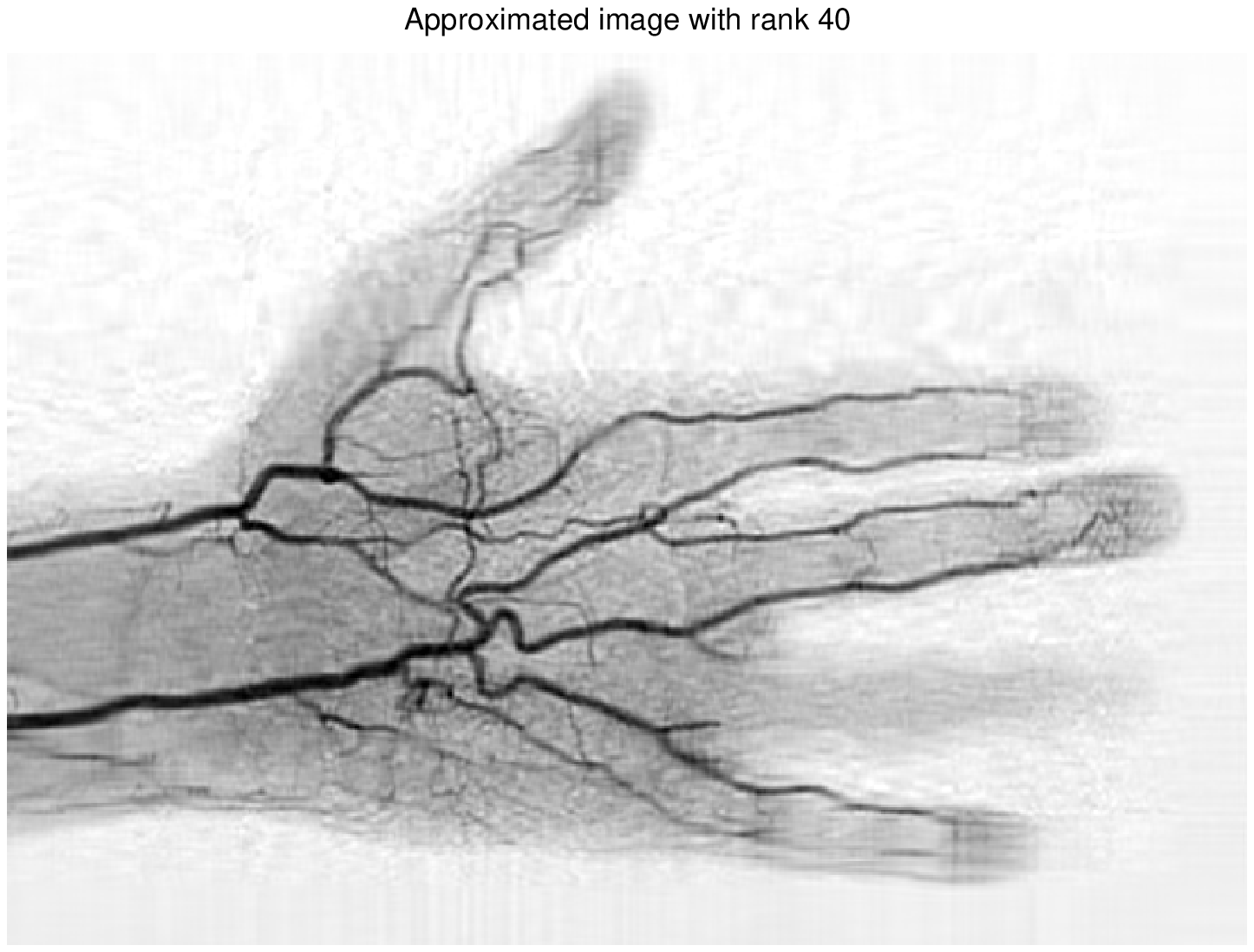}
  \end{minipage}
  \caption{Original $395\times549$ HAI and its approximation with rank 40.} \label{figure4}
\end{figure}

\begin{figure}[h]
  \centering
  \begin{minipage}[t]{0.4\linewidth}
  \centering
  \includegraphics[width=1\textwidth]{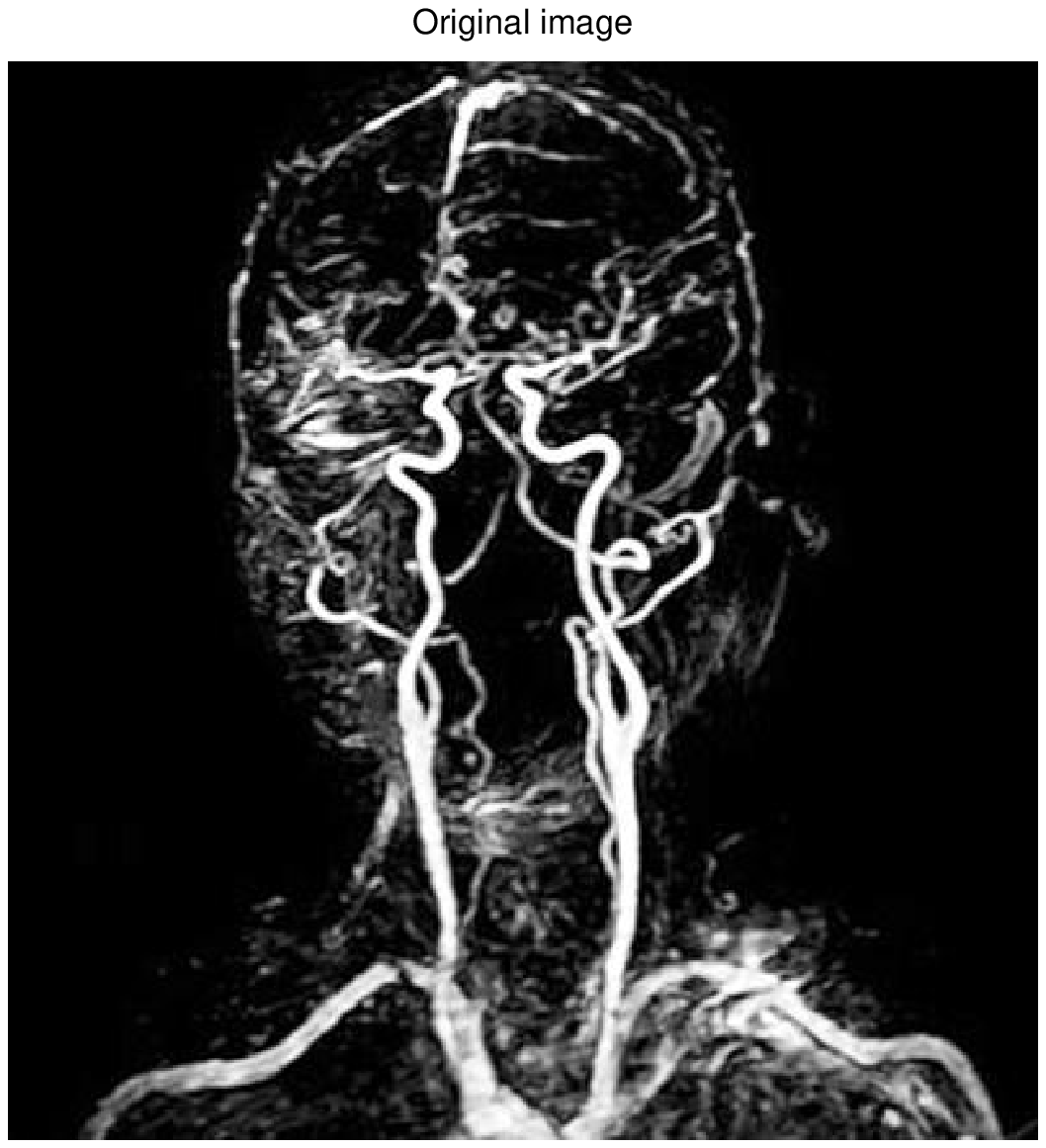}
  \end{minipage}
  \begin{minipage}[t]{0.4\linewidth}
  \centering
  \includegraphics[width=1\textwidth]{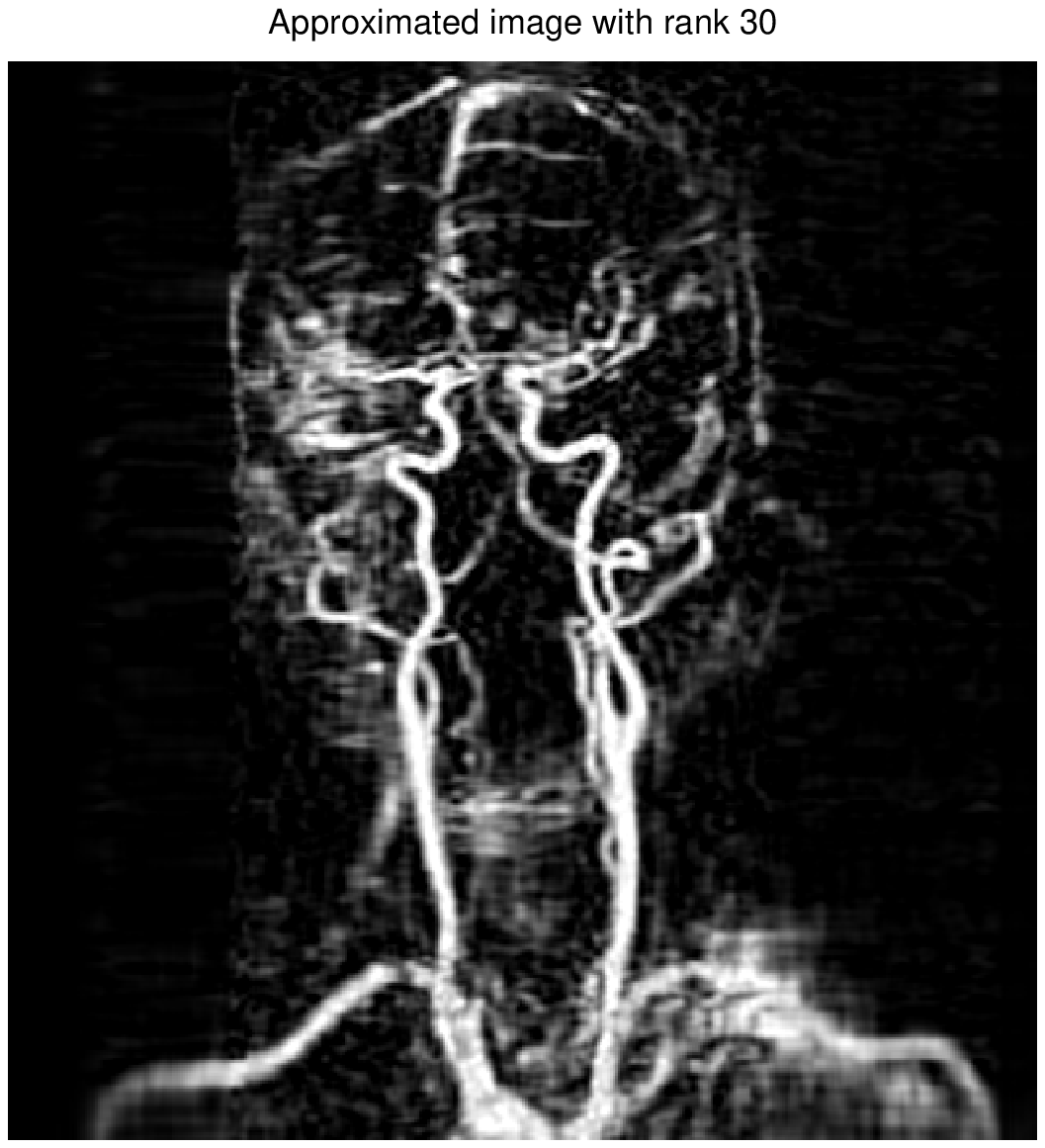}
  \end{minipage}
  \caption{Original $419\times400$ IVI and its approximation with rank 30.} \label{figure5}
\end{figure}

Table 5, 6 show that ISVTA performs powerful in finding a low-rank matrix on image inpainting problems. Indeed, we could get an exact
low-rank image by the ISVTA by choosing proper $a$. Moreover, it is necessary to point out that our method does not work well for all $a>0$,
and we can find that $a=100$ is not a good strategy for the low-rank IVI either $\mathrm{SR}=0.40$ or $\mathrm{SR}=0.50$. The numerical results
of ISVT, SVTA and SVPA compared in Table 5, 6, 7, 8, 9, 10 under same circumstance show that the ISVT algorithm performs far more better than
ISTA and SVPA on image inpainting problems for some proper $a>0$.

\begin{table}[htbp]
\centering
\begin{tabular}{|c||l|l|l|l|l|l|l|}\hline
\multicolumn{7}{|c|}{SR=0.50}\\\hline
Image&\multicolumn{2}{c}{a=1}&\multicolumn{2}{|c}{a=3}&\multicolumn{2}{|c|}{a=5}\\
\hline
(Name,\,rank,\,FR)&RE&Time&RE&Time&RE&Time\\
\hline
(BAI, 30, 1.9568)& 9.99e-05& 10.62& 9.93e-05& 10.83& 9.95e-05& 11.30\\
\hline
(HAI, 40, 2.9985)& 9.96e-05& 38.99& 9.83e-05& 44.77& 9.87e-05& 54.01\\
\hline
(IVI, 30, 3.5403)& 9.96e-05& 37.26& 9.94e-05& 41.73& 9.97e-05& 68.64\\
\hline
\multicolumn{7}{|c|}{SR=0.50}\\\hline
Image&\multicolumn{2}{c}{a=7}&\multicolumn{2}{|c}{a=30}&\multicolumn{2}{|c|}{a=100}\\
\hline
(Name,\,rank,\,FR)&RE&Time&RE&Time&RE&Time\\
\hline
(BAI, 30, 1.9568)& 9.93e-05& 11.20& 9.91e-05& 9.93& 9.97e-05& 19.66\\
\hline
(HAI, 40, 2.9985)& 9.96e-05& 76.77& 9.99e-05& 227.39& 5.10e-02& 283.21\\
\hline
(IVI, 30, 3.5403)& 9.91e-05& 42.93& 9.93e-05& 56.27& 9.96e-05& 37.02\\
\hline
\end{tabular}
\caption{\scriptsize Numerical results of ISVTA with a=1, 3, 5, 7, 30, 100 for image inpainting problems}\label{table5}
\end{table}

\begin{table}[htbp]
\centering
\begin{tabular}{|c||l|l|l|l|l|l|l|}\hline
\multicolumn{7}{|c|}{SR=0.40}\\\hline
Image&\multicolumn{2}{c}{a=1}&\multicolumn{2}{|c}{a=3}&\multicolumn{2}{|c|}{a=5}\\
\hline
(Name,\,rank,\,FR)&RE&Time&RE&Time&RE&Time\\
\hline
(BAI, 30, 1.5655)& 9.98e-05& 27.85& 9.98e-05& 27.87& 9.98e-05& 31.11\\
\hline
(HAI, 40, 2.3988)& 9.96e-05& 89.43& 9.99e-05& 107.73&  9.96e-05& 123.44\\
\hline
(IVI, 30, 2.8323)& 9.97e-05& 73.25& 9.96e-05& 91.02& 9.98e-05& 91.21\\
\hline
\multicolumn{7}{|c|}{SR=0.40}\\\hline
Image&\multicolumn{2}{c}{a=7}&\multicolumn{2}{|c}{a=30}&\multicolumn{2}{|c|}{a=100}\\
\hline
(Name,\,rank,\,FR)&RE&Time&RE&Time&RE&Time\\
\hline
(BAI, 30, 1.5655)& 9.97e-05& 29.36& 9.97e-05& 32.41& 9.97e-05& 59.47\\
\hline
(HAI, 40, 2.3988)& 9.98e-05& 122.00& 9.94e-05& 327.87& 8.90e-02& 266.88\\
\hline
(IVI, 30, 2.8323)& 9.98e-05& 98.29& 9.96e-05& 98.49& 9.96e-05& 98.58\\
\hline
\end{tabular}
\caption{\scriptsize Numerical results of ISVTA with a=1, 3, 5, 7, 30, 100 for image inpainting problems}\label{table6}
\end{table}

\begin{table}[htbp]
\centering
\begin{tabular}{|c||l|l|l|l|l|l|l|}\hline
\multicolumn{7}{|c|}{SR=0.35}\\\hline
Image&\multicolumn{2}{c}{a=1}&\multicolumn{2}{|c}{a=3}&\multicolumn{2}{|c|}{a=5}\\
\hline
(Name,\,rank,\,FR)&RE&Time&RE&Time&RE&Time\\
\hline
(BAI, 30, 1.3698)& 9.99e-04& 53.40& 9.99e-04& 46.55& 9.99e-04& 58.41\\
\hline
(HAI, 40, 2.0990)& 9.99e-04& 91.03& 9.98e-04& 86.91& 9.97e-04& 102.03\\
\hline
(IVI, 30, 2.4782)& 9.98e-04& 104.57& 9.99e-04& 90.67& 9.99e-04& 94.38\\
\hline
\multicolumn{7}{|c|}{SR=0.35}\\\hline
Image&\multicolumn{2}{c}{a=7}&\multicolumn{2}{|c}{a=30}&\multicolumn{2}{|c|}{a=100}\\
\hline
(Name,\,rank,\,FR)&RE&Time&RE&Time&RE&Time\\
\hline
(BAI, 30, 1.3698)& 9.99e-04& 45.16& 9.99e-04& 54.42& 2.68e-01& 73.19 \\
\hline
(HAI, 40, 2.0990)& 9.97e-04& 136.50& 9.98e-04& 348.71& 7.49e-02& 443.80\\
\hline
(IVI, 30, 3.5403)& 9.97e-04& 83.67& 9.97e-04& 107.98& 9.98e-04& 134.67\\
\hline
\end{tabular}
\caption{\scriptsize Numerical results of ISVTA with a=1, 3, 5, 7, 30, 100 for image inpainting problems}\label{table7}
\end{table}

\begin{table}[htbp]
\centering
\begin{tabular}{|c||l|l|l|l|l|l|l|}\hline
\multicolumn{7}{|c|}{SR=0.30}\\\hline
Image&\multicolumn{2}{c}{a=1}&\multicolumn{2}{|c}{a=3}&\multicolumn{2}{|c|}{a=5}\\
\hline
(Name,\,rank,\,FR)&RE&Time&RE&Time&RE&Time\\
\hline
(BAI, 30, 1.1741)& 1.91e-01& 72.77& 3.97e-01&  71.37& 3.53e-01 & 71.52 \\
\hline
(HAI, 40,  1.7991)& 9.99e-04& 189.38& 9.98e-04& 202.80& 9.98e-04& 223.22\\
\hline
(IVI, 30, 2.1242)& 9.99e-04& 187.15& 9.99e-04& 184.93& 9.99e-04& 184.69\\
\hline
\multicolumn{7}{|c|}{SR=0.30}\\\hline
Image&\multicolumn{2}{c}{a=7}&\multicolumn{2}{|c}{a=30}&\multicolumn{2}{|c|}{a=100}\\
\hline
(Name,\,rank,\,FR)&RE&Time&RE&Time&RE&Time\\
\hline
(BAI, 30, 1.1741)& 3.18e-01& 71.84& 2.77e-01 & 71.55& 4.03e-01& 71.33\\
\hline
(HAI, 40, 1.7991)& 9.99e-04& 322.45& 1.53e-02& 447.79& 3.54e-01& 437.48\\
\hline
(IVI, 30, 2.1242)& 9.99e-04& 228.20& 1.20e-03& 320.00& 5.10e-03& 322.57\\
\hline
\end{tabular}
\caption{\scriptsize Numerical results of ISVTA with a=1, 3, 5, 7, 30, 100 for image inpainting problems}\label{table8}
\end{table}

\begin{table}[htbp]
\centering
\begin{tabular}{|c||l|l|l|l|l|l|l|}\hline
\multicolumn{7}{|c|}{SVTA for image inpainting}\\\hline
Image&\multicolumn{2}{c}{SR=0.5}&\multicolumn{2}{|c}{SR=0.4}&\multicolumn{2}{|c|}{SR=0.35}\\
\hline
(Name,\,rank)&RE&Time&RE&Time&RE&Time\\
\hline
(BAI, 30)& 3.42e-02& 75.01& 1.21e-01& 73.79& 1.75e-01& 73.42\\
\hline
(HAI, 40)& 1.90e-03& 456.63& 2.11e-02& 459.69& 3.36e-02& 444.90\\
\hline
(IVI, 30)& 9.99e-04& 46.76& 6.97e-02& 348.90& 1.40e-01& 346.91\\
\hline
\end{tabular}
\caption{\scriptsize Numerical results of SVTA with SR=0.5, 0.4, 0.35 for image inpainting problems}\label{table9}
\end{table}

\begin{table}[htbp]
\centering
\begin{tabular}{|c||l|l|l|l|l|l|l|}\hline
\multicolumn{7}{|c|}{SVPA for image inpainting}\\
\hline
Image&\multicolumn{2}{c}{SR=0.5}&\multicolumn{2}{|c}{SR=0.4}&\multicolumn{2}{|c|}{SR=0.35}\\
\hline
(Name,\,rank)&RE&Time&RE&Time&RE&Time\\
\hline
(BAI, 30)& 6.24e-01& 73.96& 7.18e-01& 71.26& 7.63e-01& 71.31\\
\hline
(HAI, 40)& 7.01e-01& 439.22& 7.71e-01& 429.31& 8.03e-01& 422.59\\
\hline
(IVI, 30)& 6.42e-01& 327.01& 7.42e-01& 314.19& 7.77e-01& 313.49\\
\hline
\end{tabular}
\caption{\scriptsize Numerical results of SVPA with SR=0.5, 0.4, 0.35 for image inpainting problems}\label{table10}
\end{table}

\begin{figure}
  \centering
  \begin{minipage}[t]{0.4\linewidth}
  \centering
  \includegraphics[width=1\textwidth]{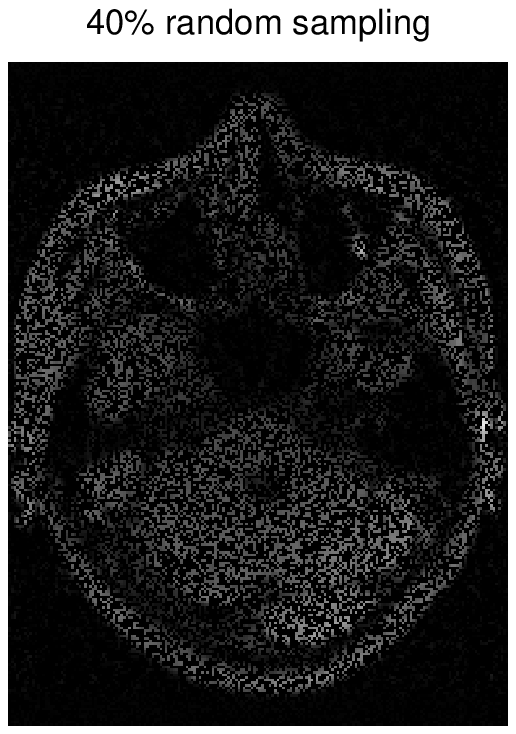}
  \end{minipage}\\
  \begin{minipage}[t]{0.4\linewidth}
  \centering
  \includegraphics[width=1\textwidth]{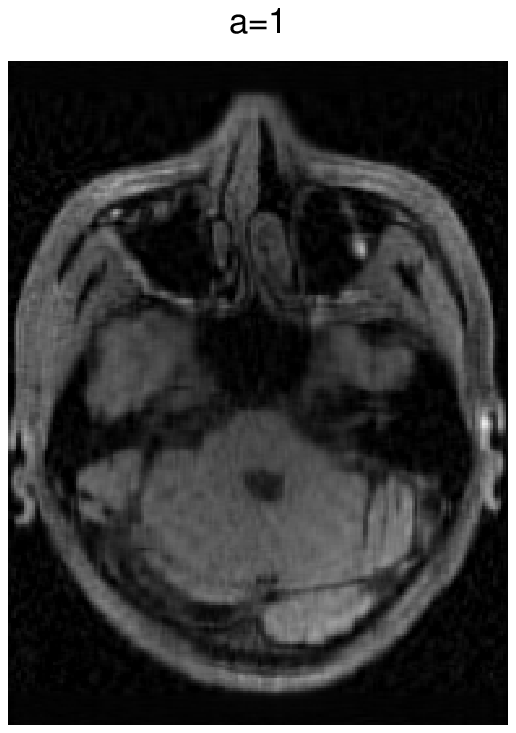}
  \end{minipage}
  \begin{minipage}[t]{0.4\linewidth}
  \centering
  \includegraphics[width=1\textwidth]{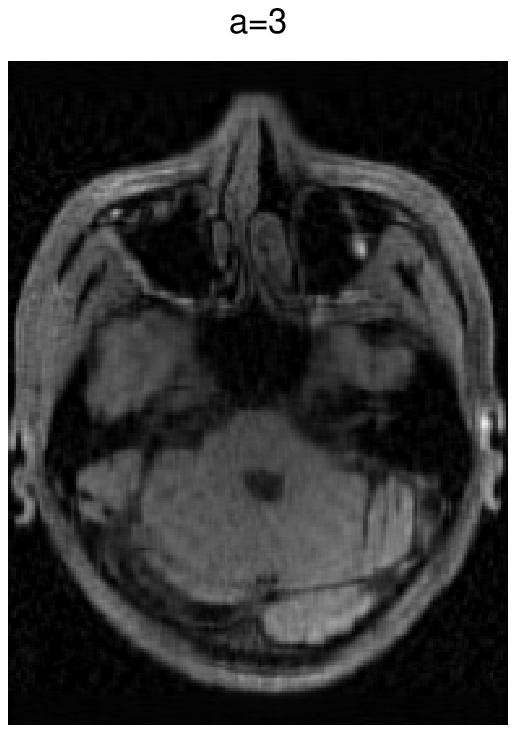}
  \end{minipage}
  \begin{minipage}[t]{0.4\linewidth}
  \centering
  \includegraphics[width=1\textwidth]{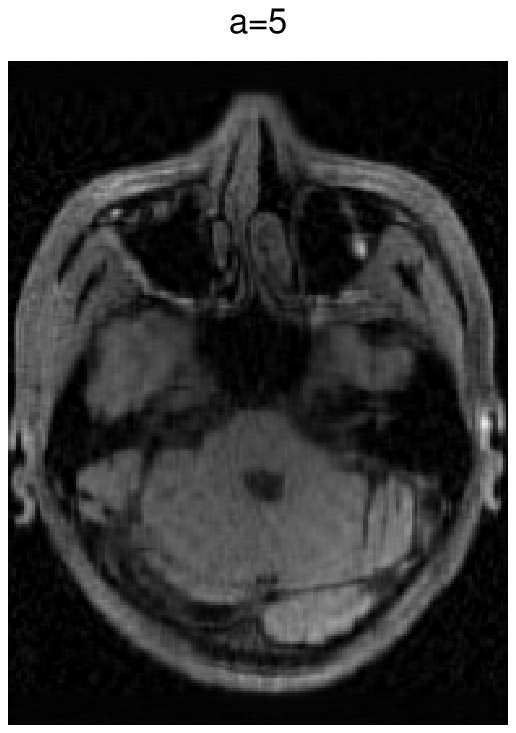}
  \end{minipage}
  \begin{minipage}[t]{0.4\linewidth}
  \centering
  \includegraphics[width=1\textwidth]{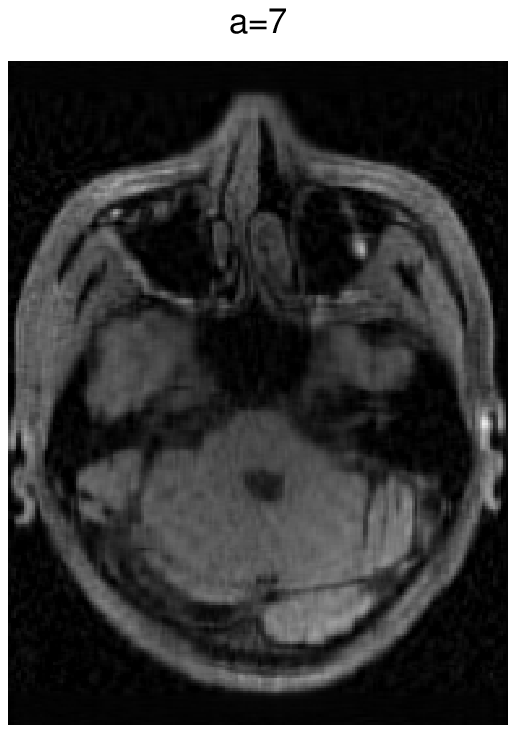}
  \end{minipage}
  \begin{minipage}[t]{0.4\linewidth}
  \centering
  \includegraphics[width=1\textwidth]{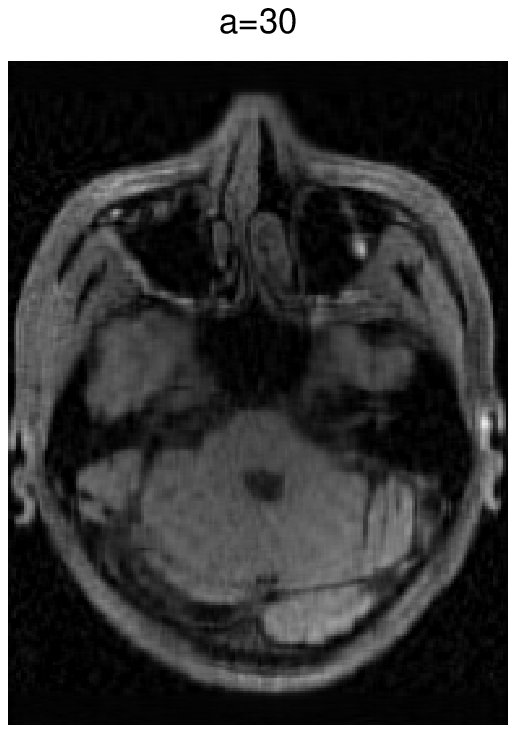}
  \end{minipage}
  \begin{minipage}[t]{0.4\linewidth}
  \centering
  \includegraphics[width=1\textwidth]{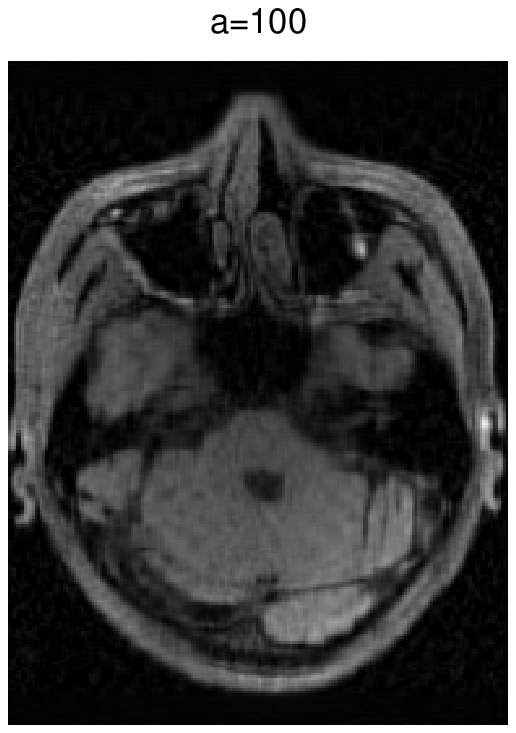}
  \end{minipage}
  \begin{minipage}[t]{0.4\linewidth}
  \centering
  \includegraphics[width=1\textwidth]{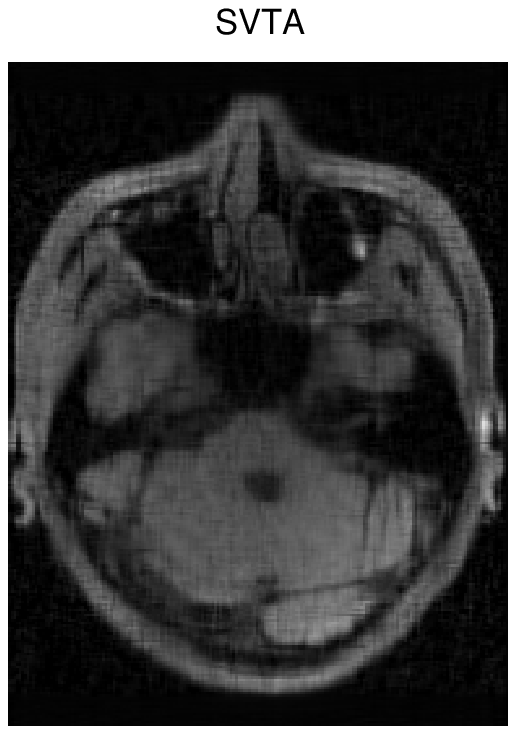}
  \end{minipage}
  \begin{minipage}[t]{0.4\linewidth}
  \centering
  \includegraphics[width=1\textwidth]{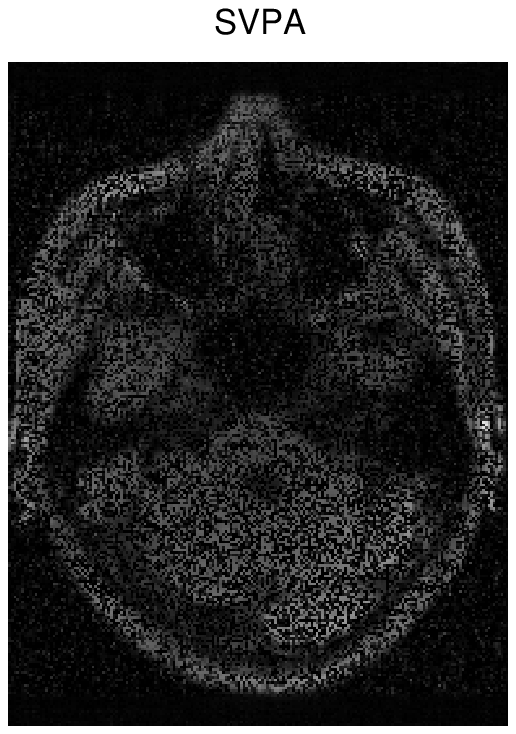}
  \end{minipage}
  \caption{Comparisons of ISVTA, SVTA and SVPA for recovering the approximated low-rank BAI with SR=0.40.} \label{figure6}
\end{figure}

\begin{figure}
  \centering
  \begin{minipage}[t]{0.4\linewidth}
  \centering
  \includegraphics[width=1\textwidth]{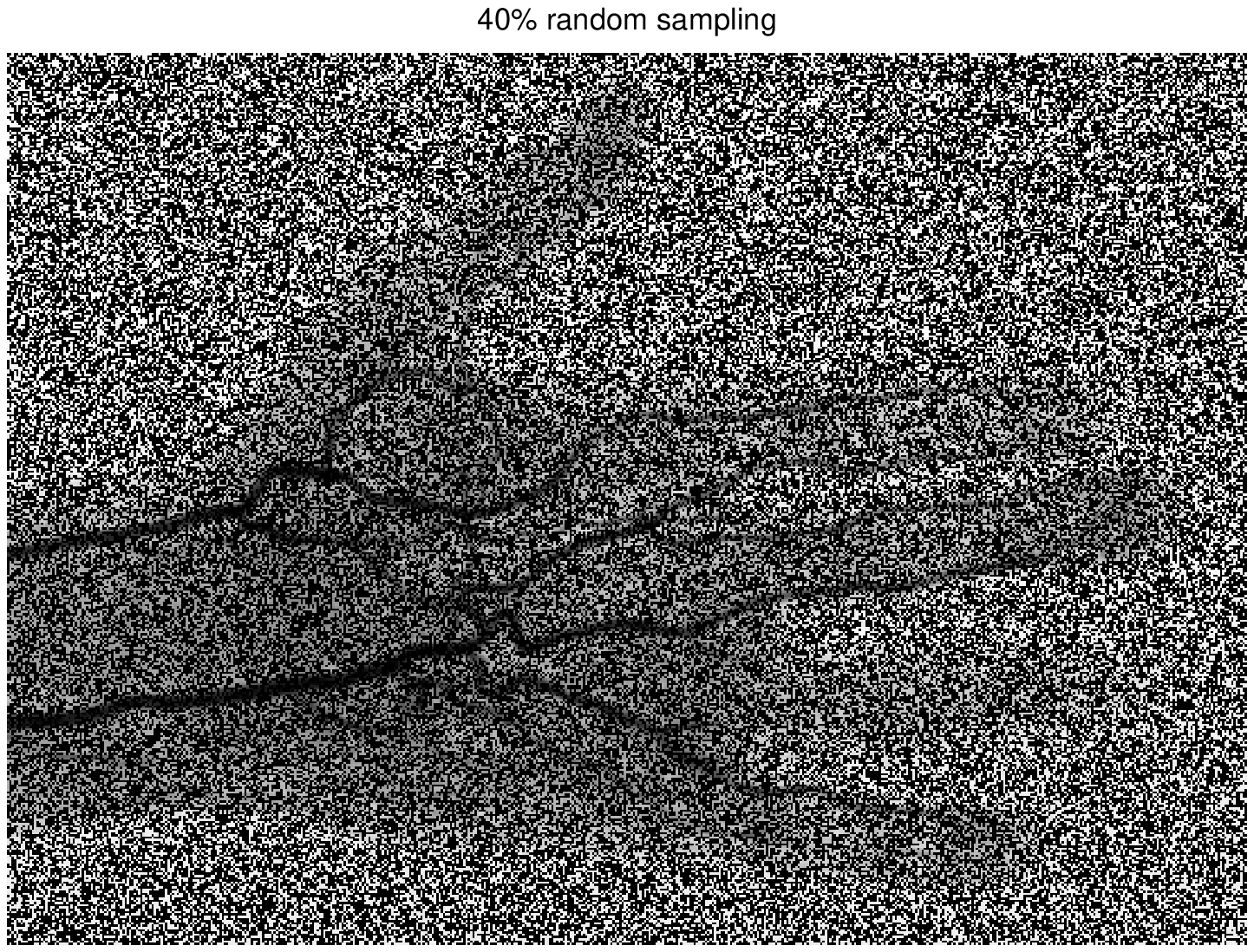}
  \end{minipage}\\
  \begin{minipage}[t]{0.4\linewidth}
  \centering
  \includegraphics[width=1\textwidth]{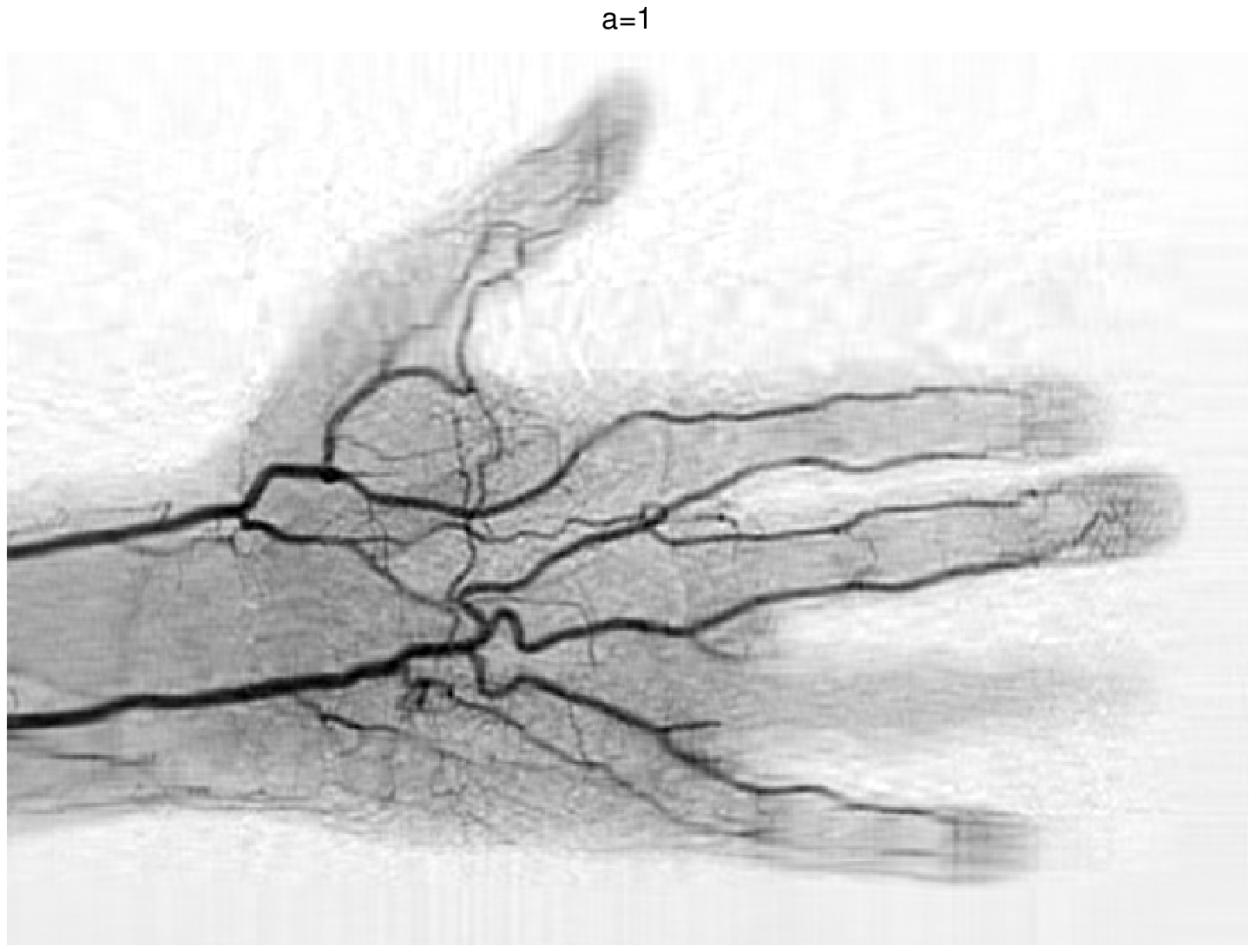}
  \end{minipage}
  \begin{minipage}[t]{0.4\linewidth}
  \centering
  \includegraphics[width=1\textwidth]{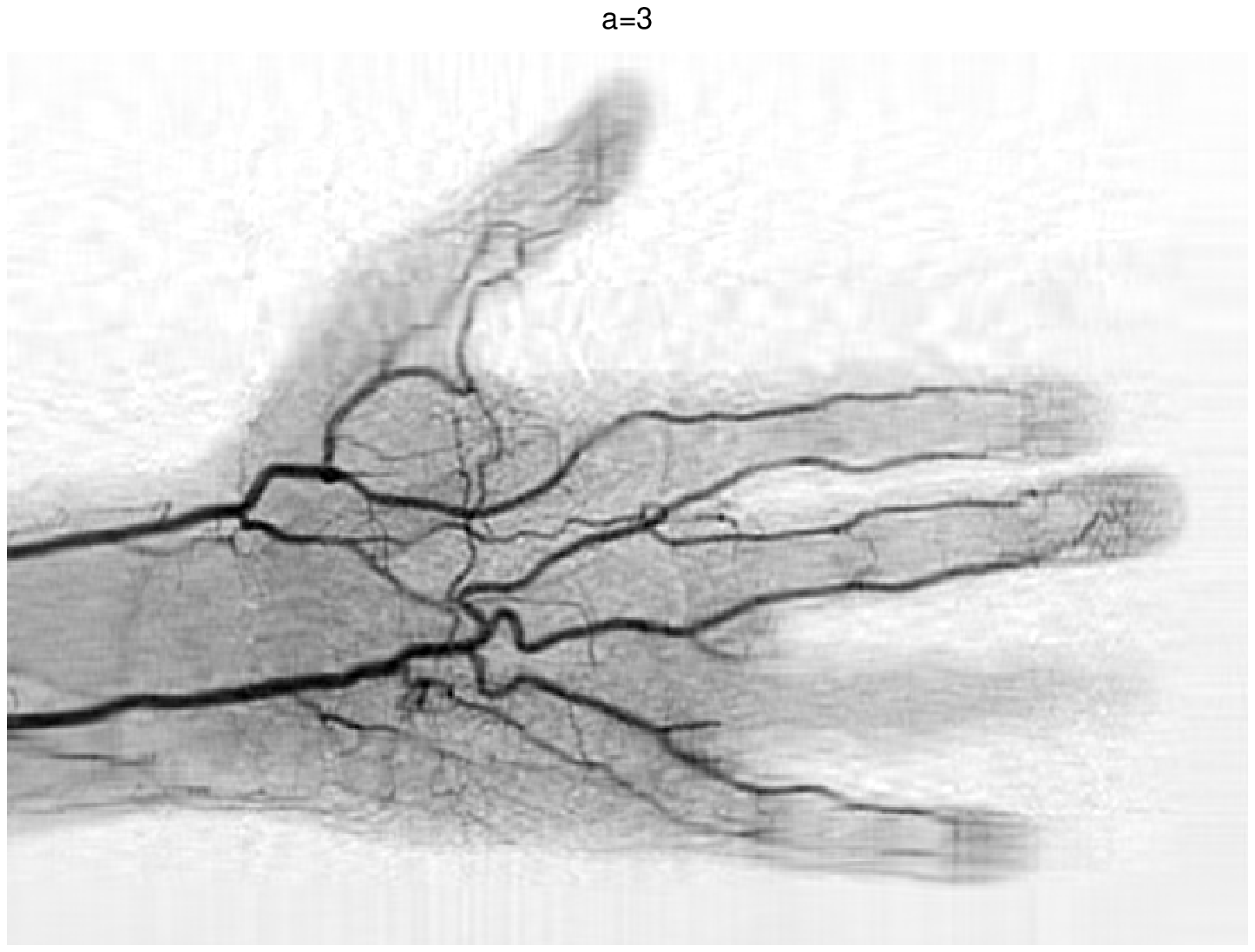}
  \end{minipage}
  \begin{minipage}[t]{0.4\linewidth}
  \centering
  \includegraphics[width=1\textwidth]{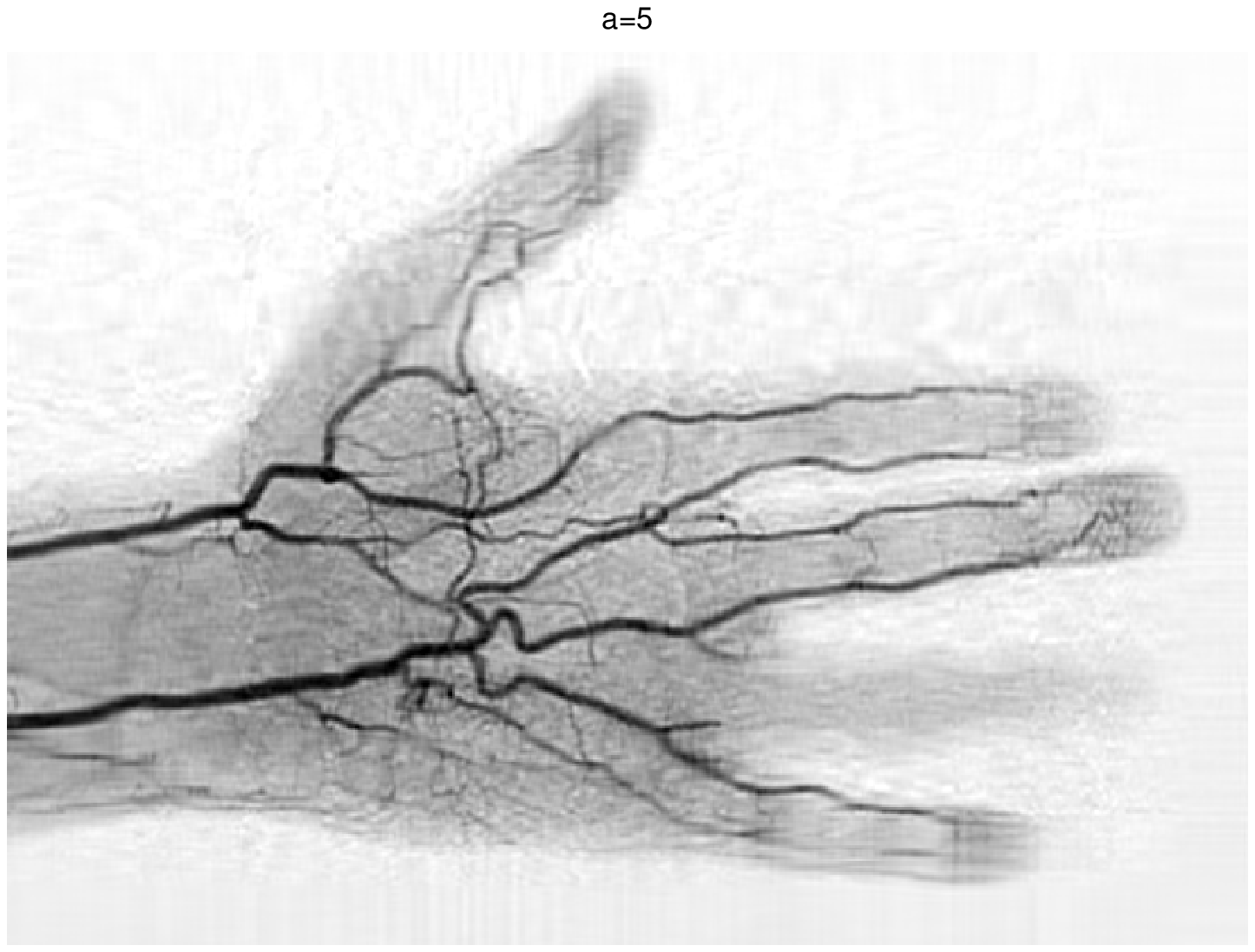}
  \end{minipage}
  \begin{minipage}[t]{0.4\linewidth}
  \centering
  \includegraphics[width=1\textwidth]{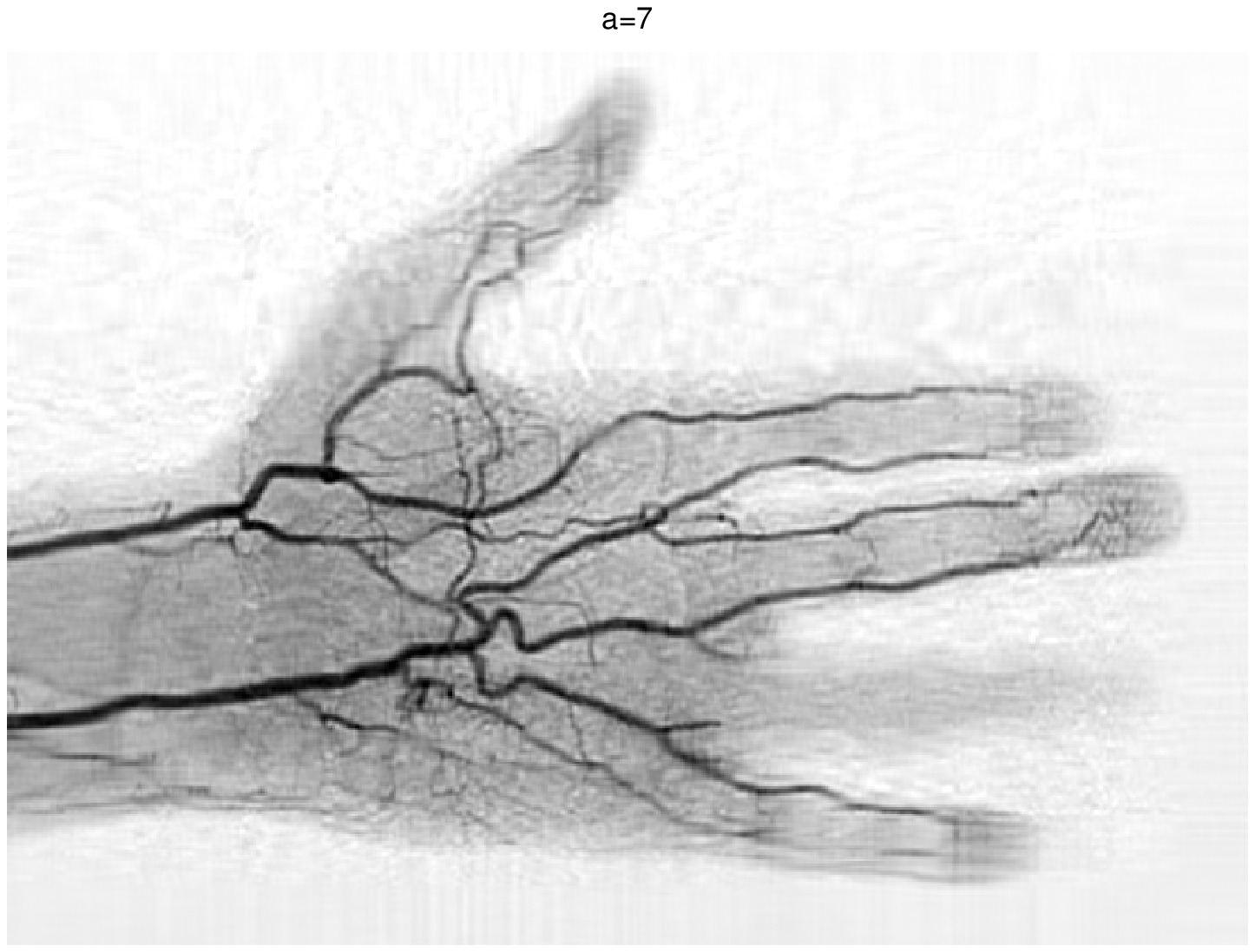}
  \end{minipage}
  \begin{minipage}[t]{0.4\linewidth}
  \centering
  \includegraphics[width=1\textwidth]{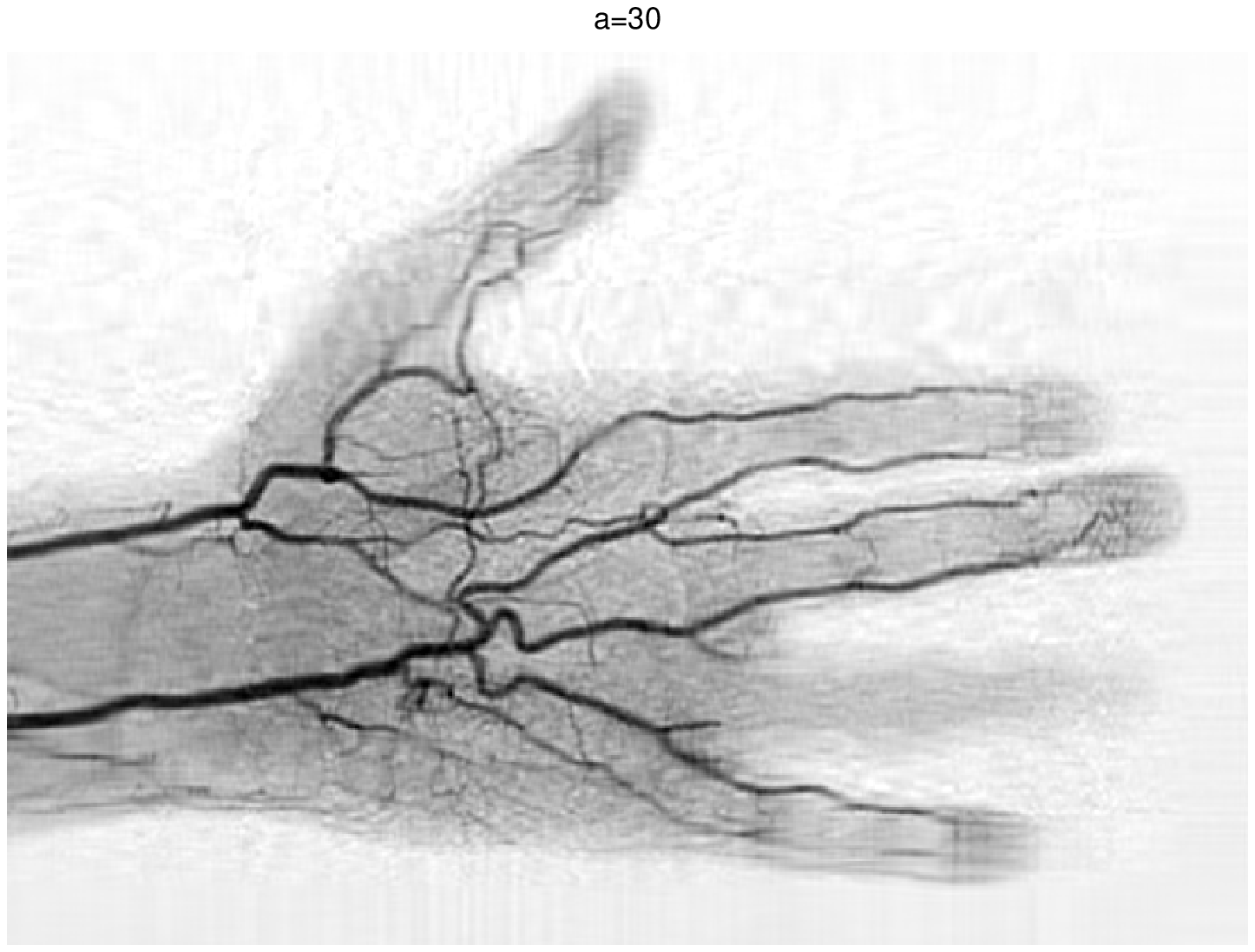}
  \end{minipage}
  \begin{minipage}[t]{0.4\linewidth}
  \centering
  \includegraphics[width=1\textwidth]{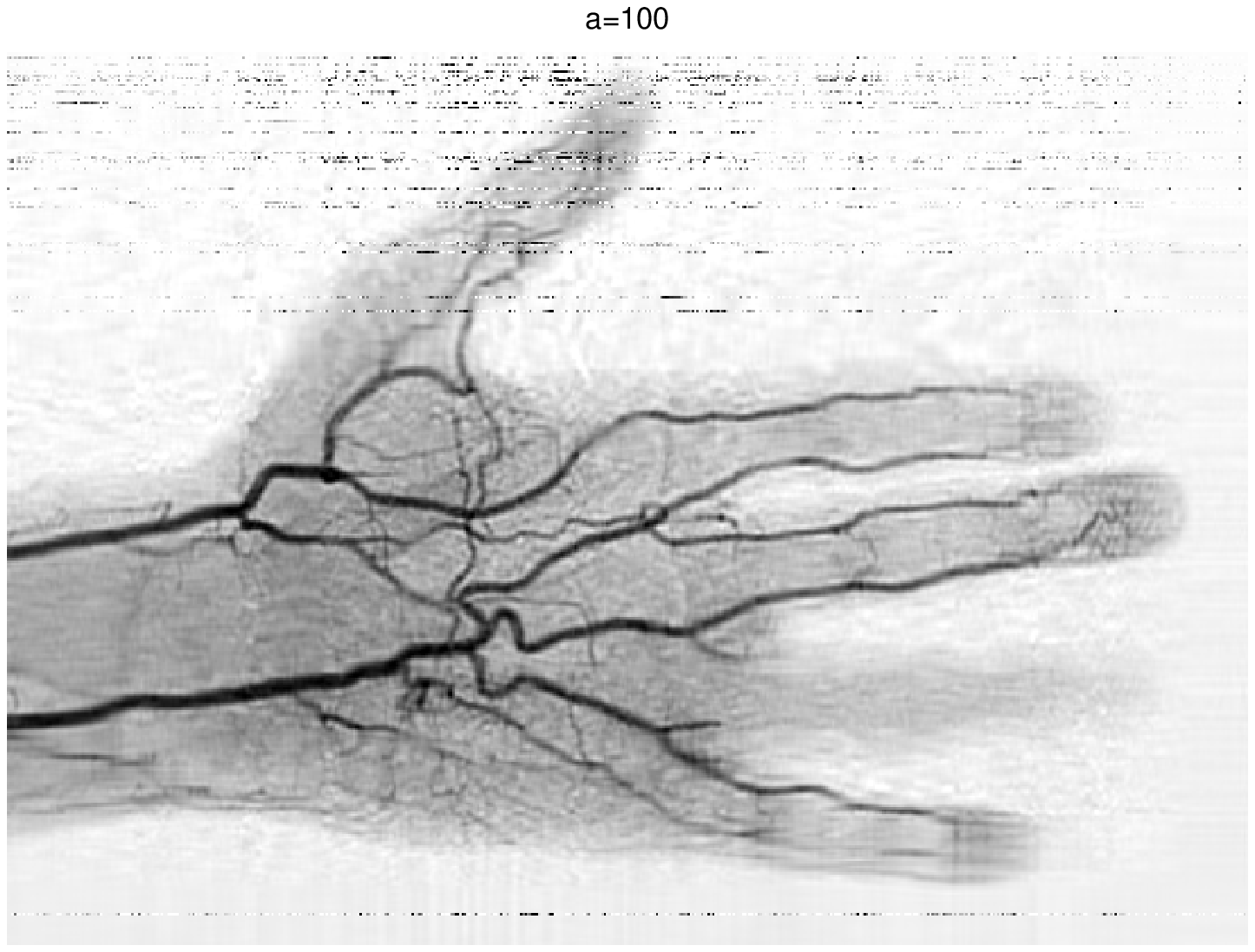}
  \end{minipage}
  \begin{minipage}[t]{0.4\linewidth}
  \centering
  \includegraphics[width=1\textwidth]{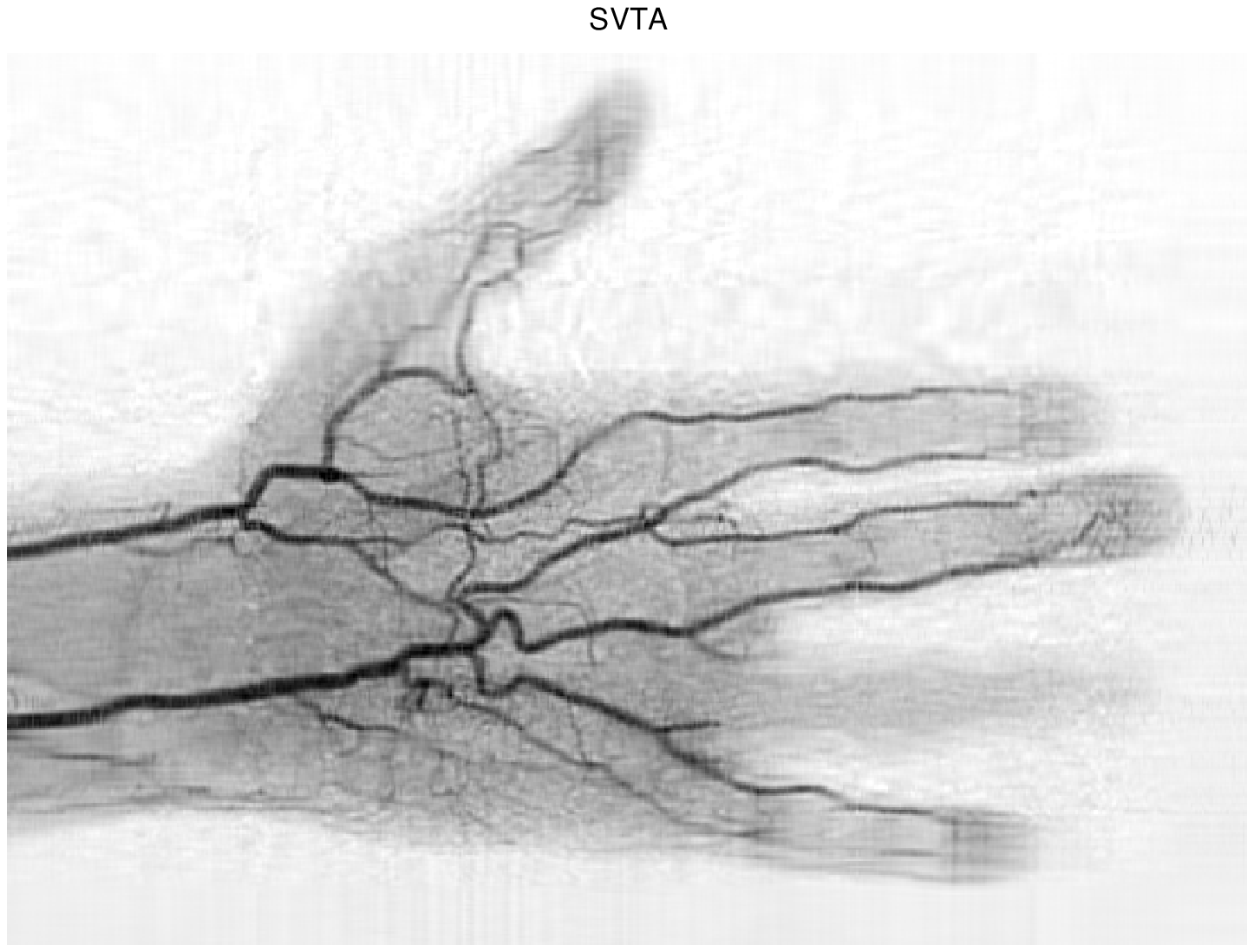}
  \end{minipage}
  \begin{minipage}[t]{0.4\linewidth}
  \centering
  \includegraphics[width=1\textwidth]{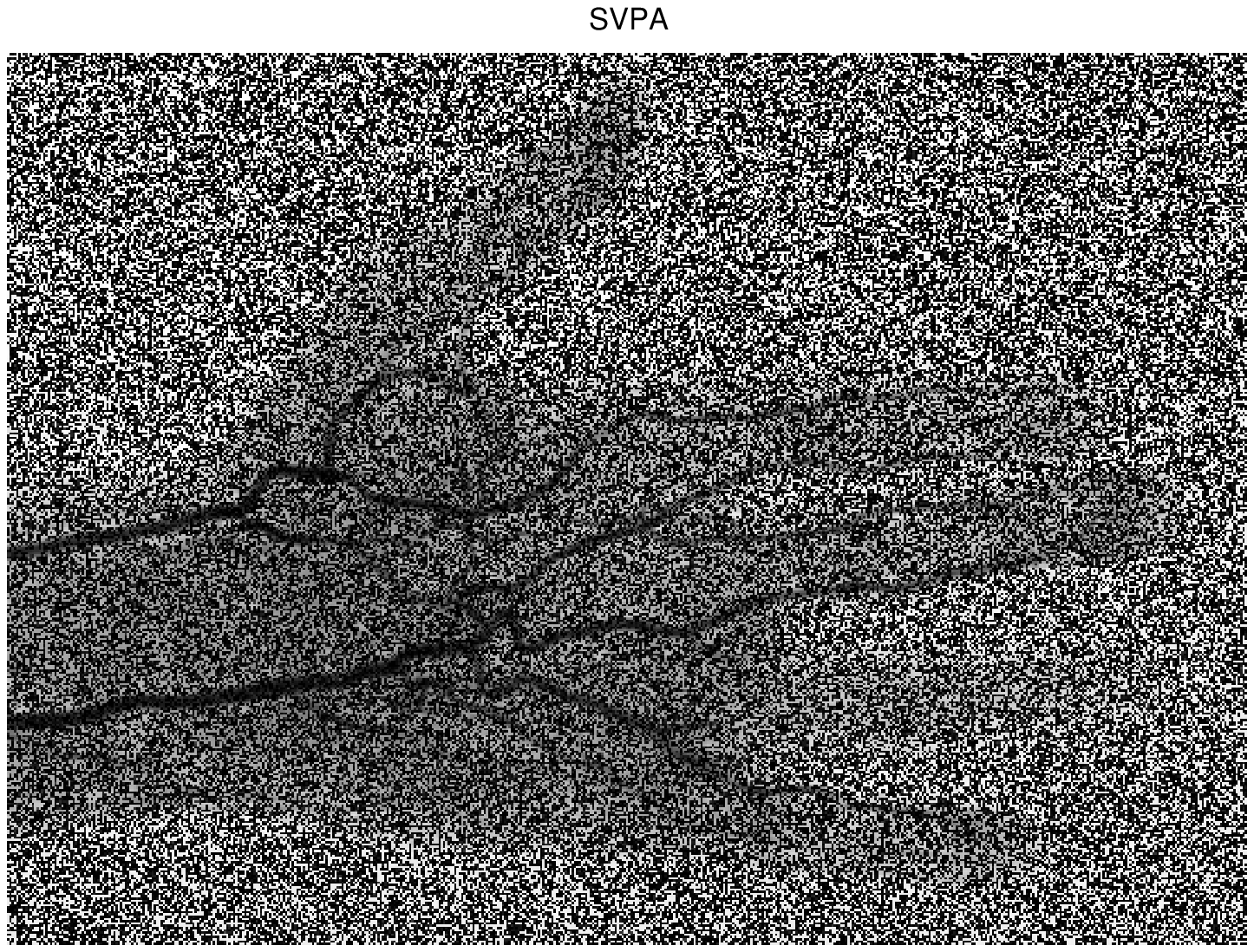}
  \end{minipage}
  \caption{Comparisons of ISVTA, SVTA and SVPA for recovering the approximated low-rank HAI with SR=0.40.} \label{figure7}
\end{figure}

\begin{figure}
  \centering
  \begin{minipage}[t]{0.4\linewidth}
  \centering
  \includegraphics[width=1\textwidth]{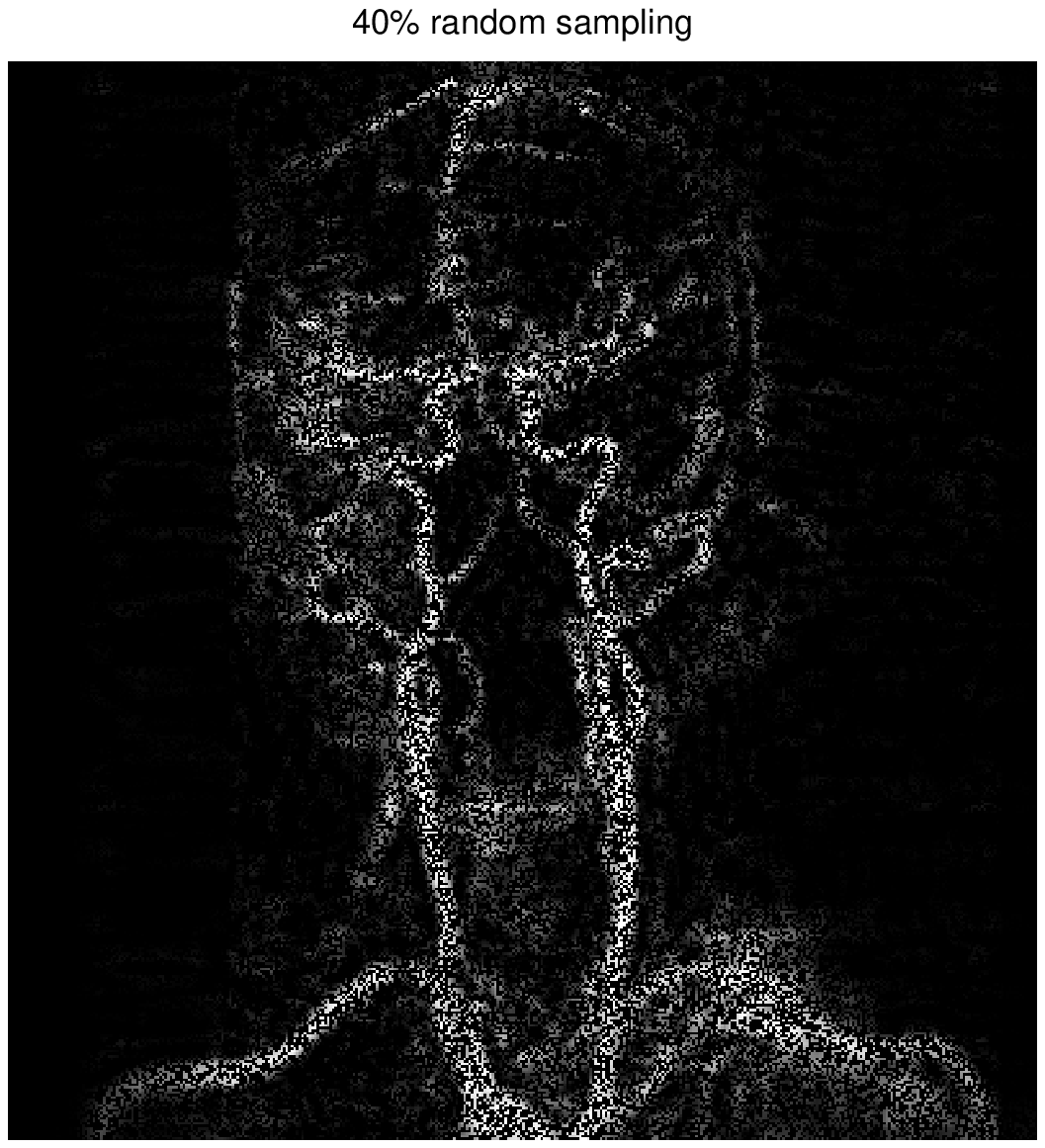}
  \end{minipage}\\
  \begin{minipage}[t]{0.4\linewidth}
  \centering
  \includegraphics[width=1\textwidth]{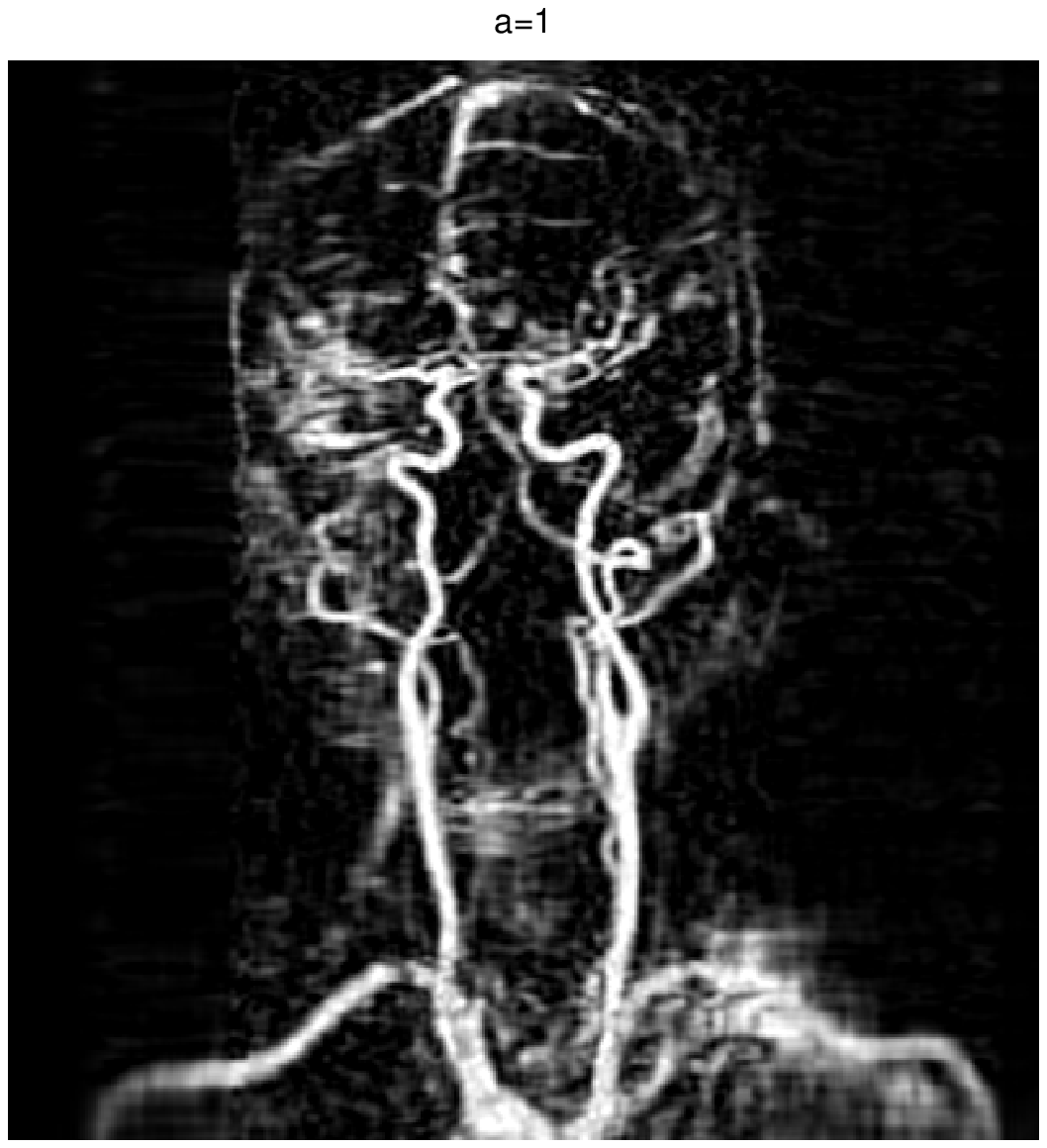}
  \end{minipage}
  \begin{minipage}[t]{0.4\linewidth}
  \centering
  \includegraphics[width=1\textwidth]{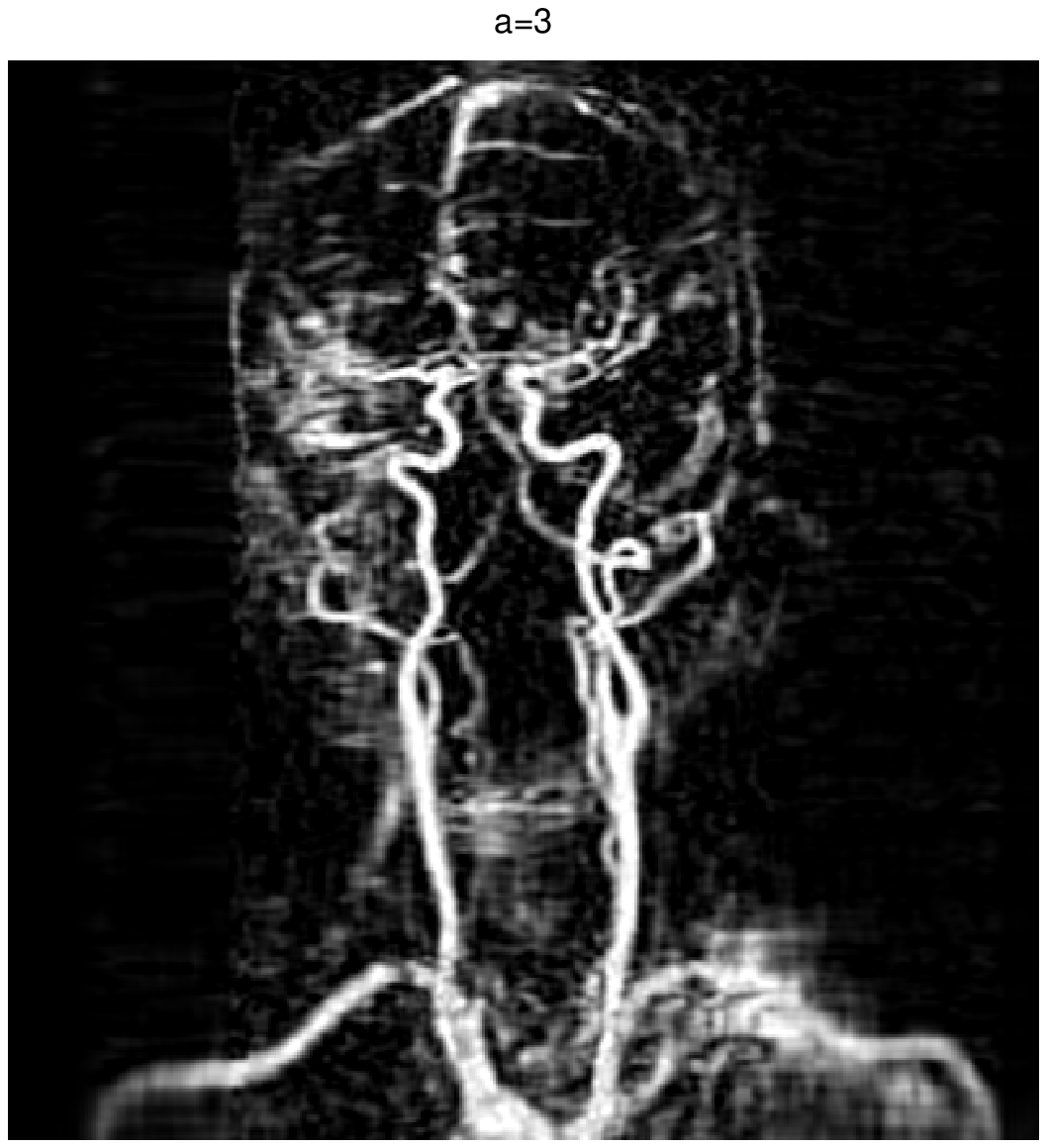}
  \end{minipage}
  \begin{minipage}[t]{0.4\linewidth}
  \centering
  \includegraphics[width=1\textwidth]{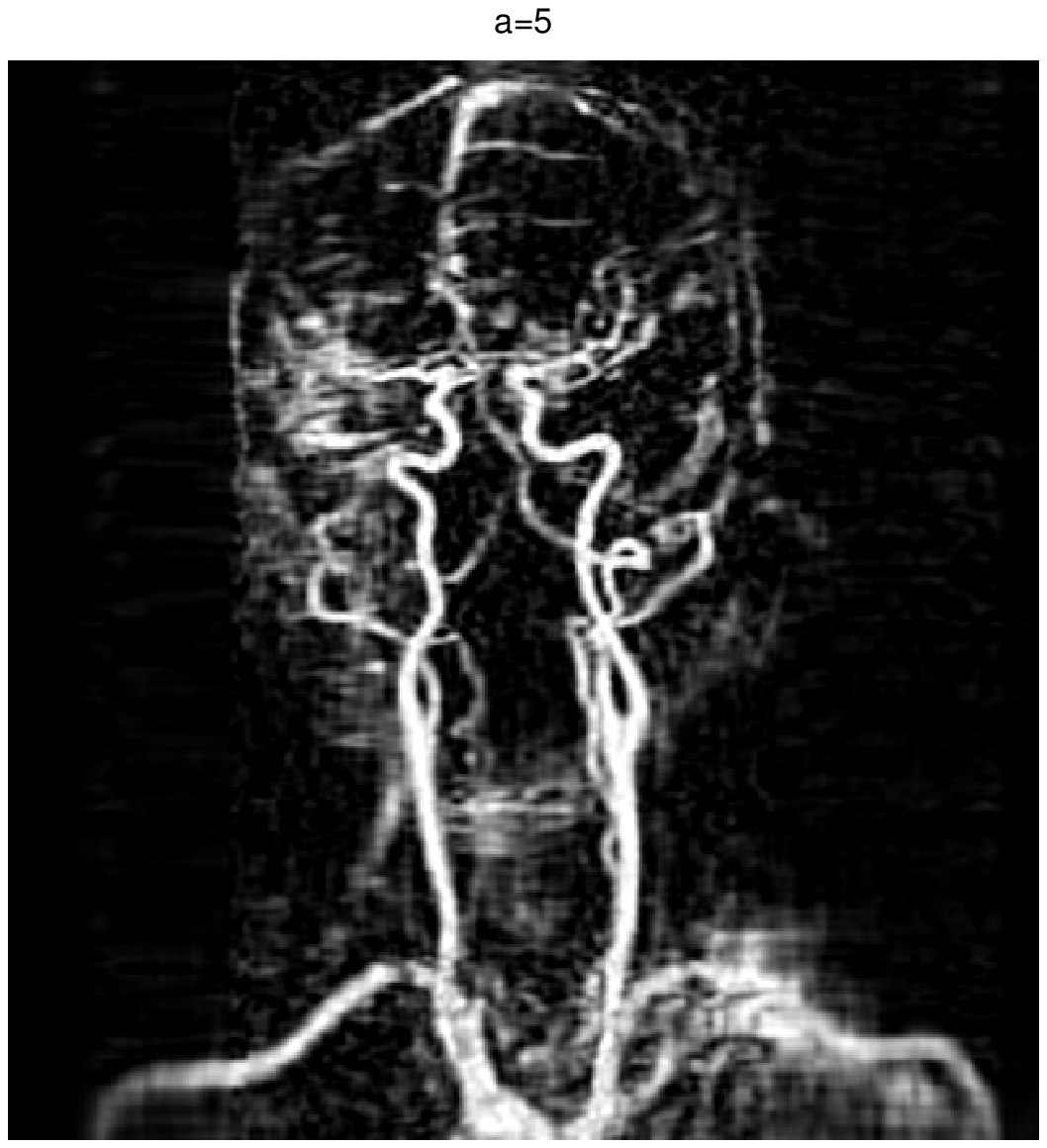}
  \end{minipage}
  \begin{minipage}[t]{0.4\linewidth}
  \centering
  \includegraphics[width=1\textwidth]{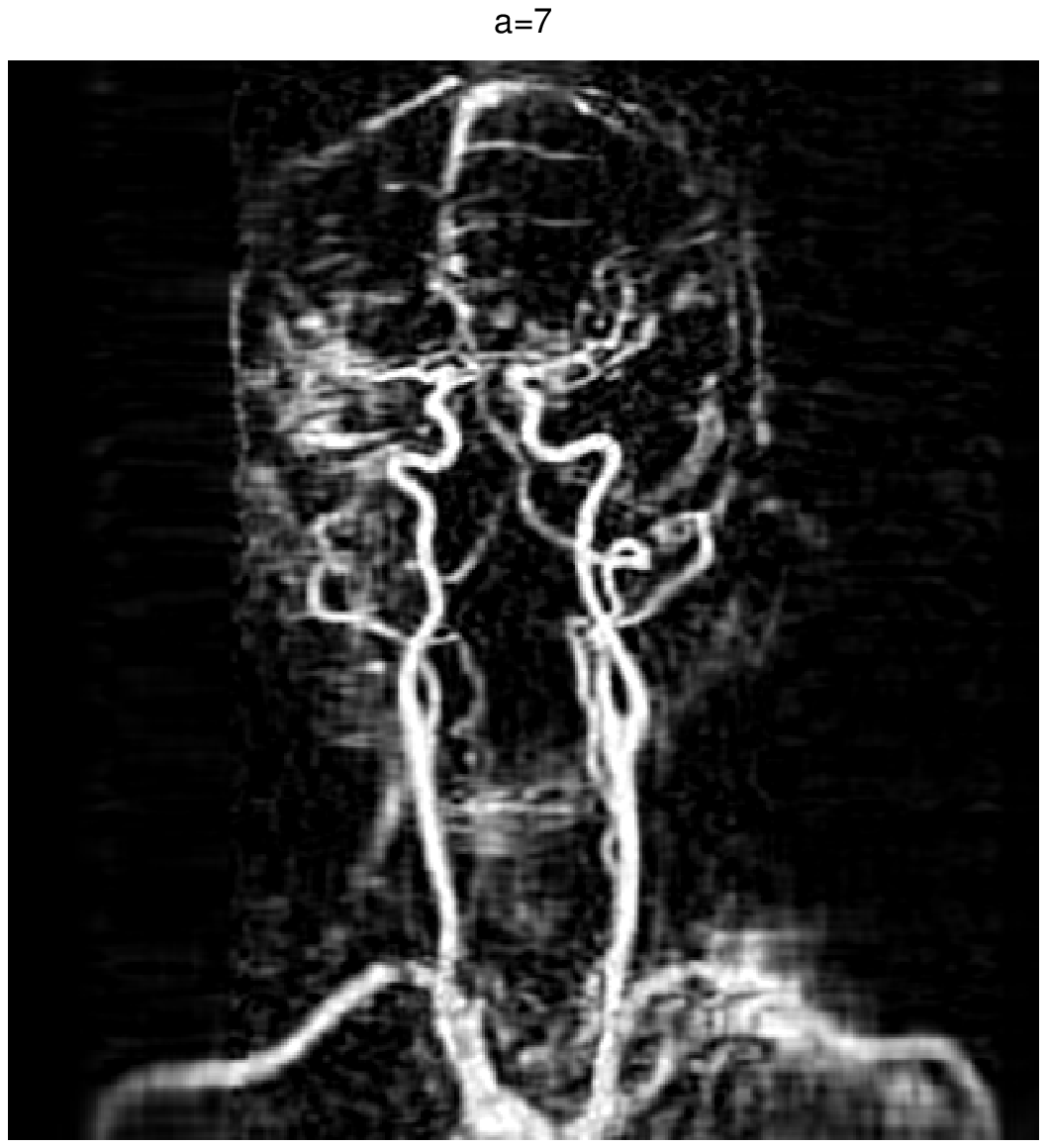}
  \end{minipage}
  \begin{minipage}[t]{0.4\linewidth}
  \centering
  \includegraphics[width=1\textwidth]{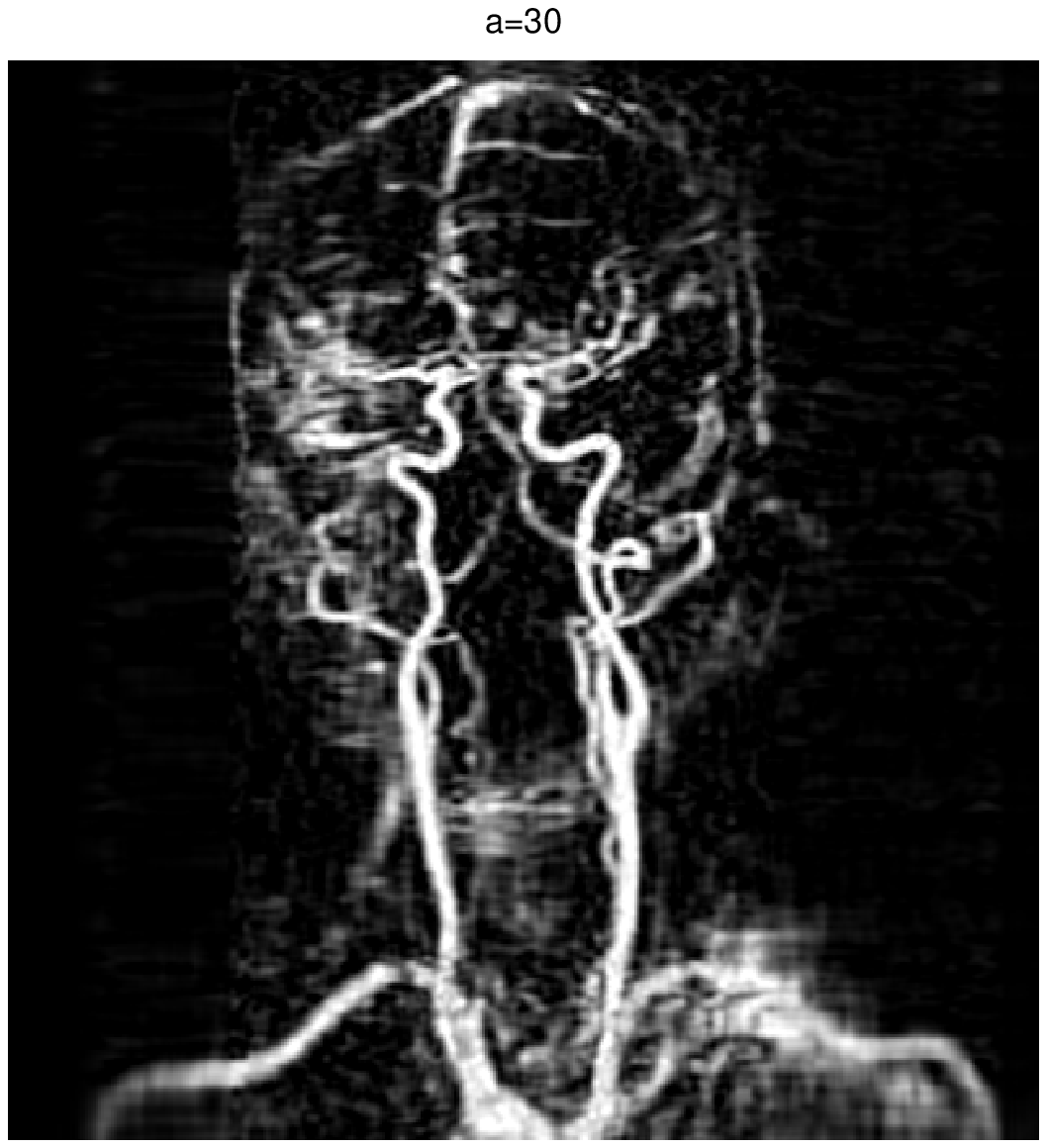}
  \end{minipage}
  \begin{minipage}[t]{0.4\linewidth}
  \centering
  \includegraphics[width=1\textwidth]{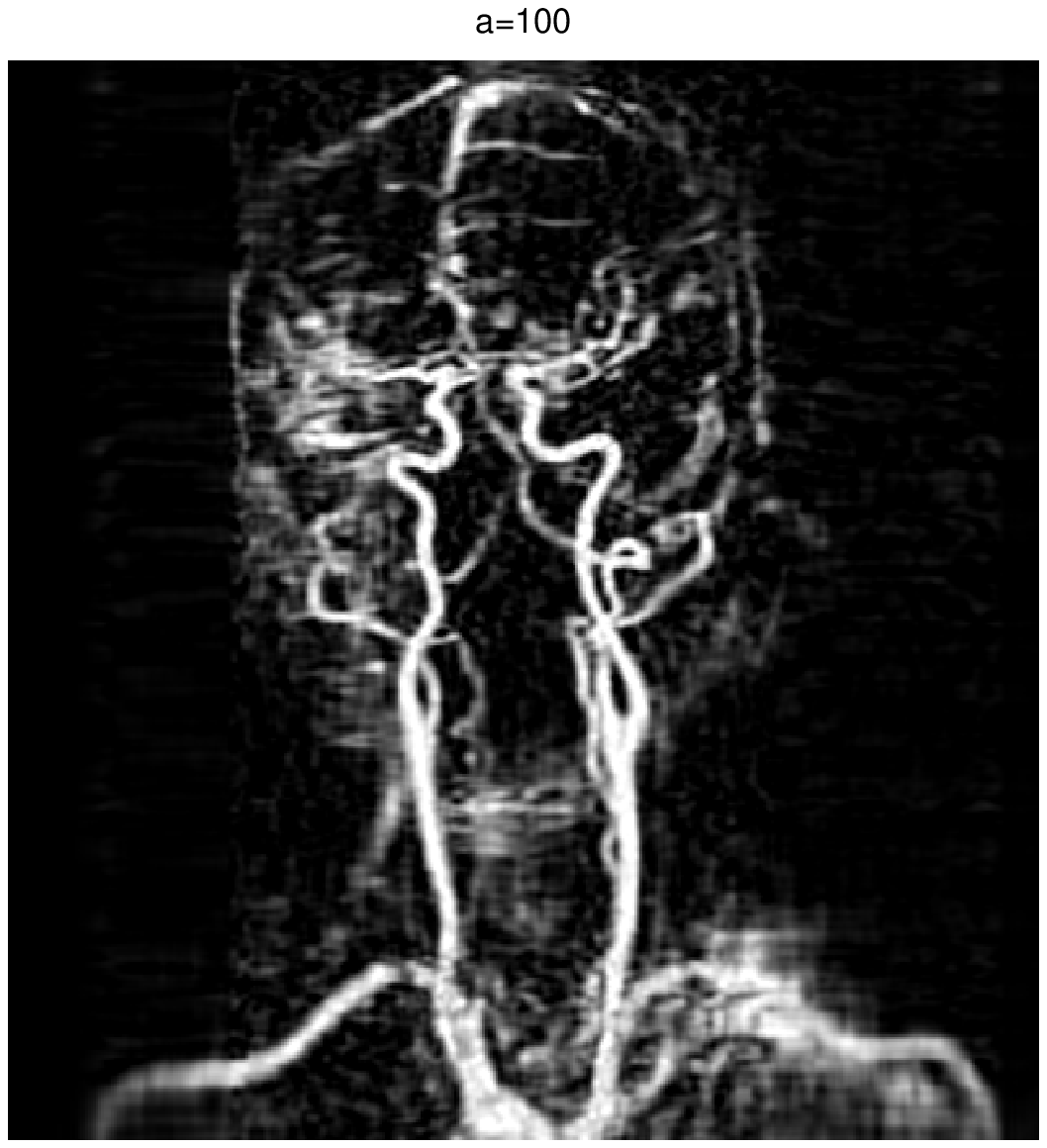}
  \end{minipage}
  \begin{minipage}[t]{0.4\linewidth}
  \centering
  \includegraphics[width=1\textwidth]{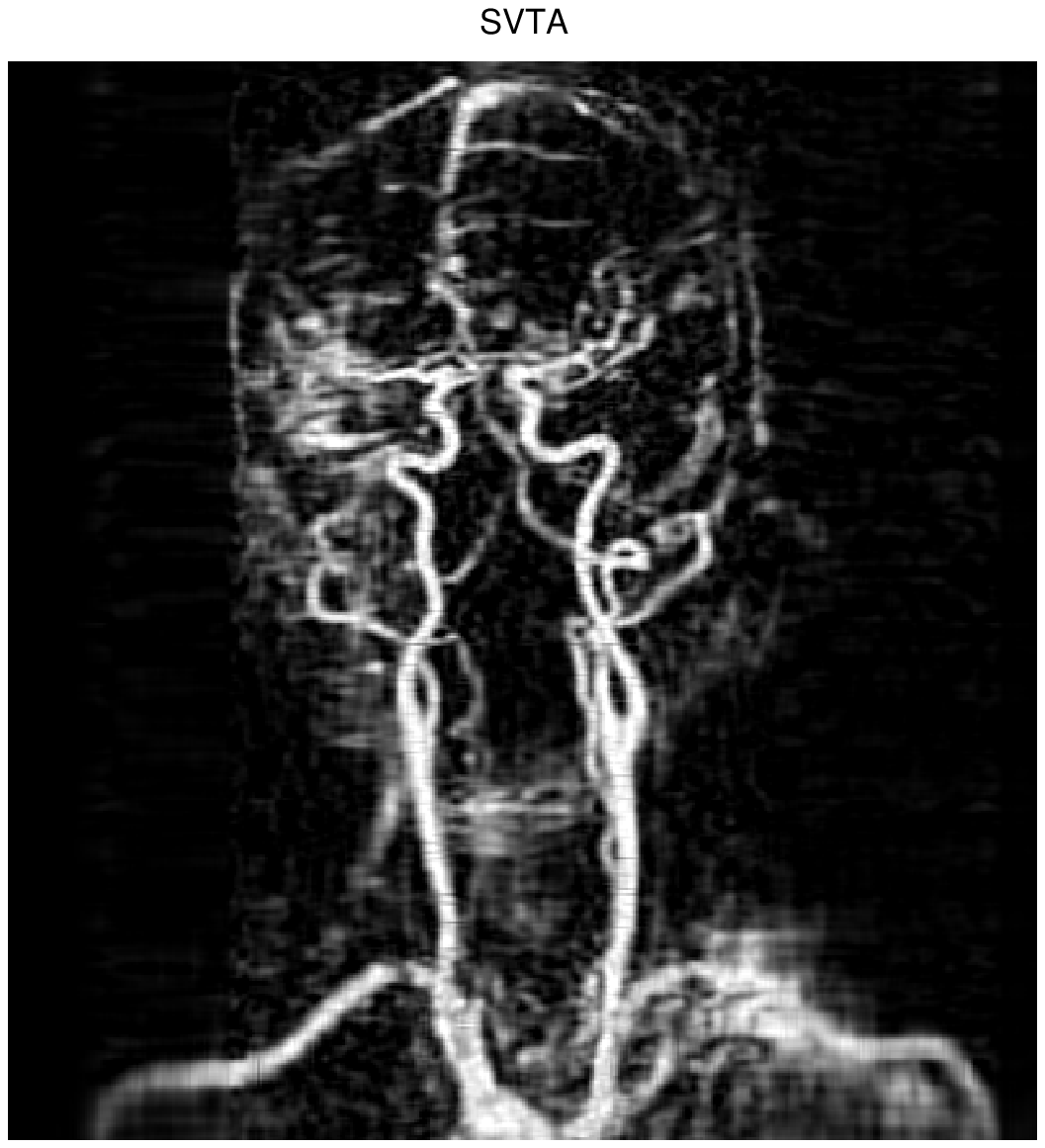}
  \end{minipage}
  \begin{minipage}[t]{0.4\linewidth}
  \centering
  \includegraphics[width=1\textwidth]{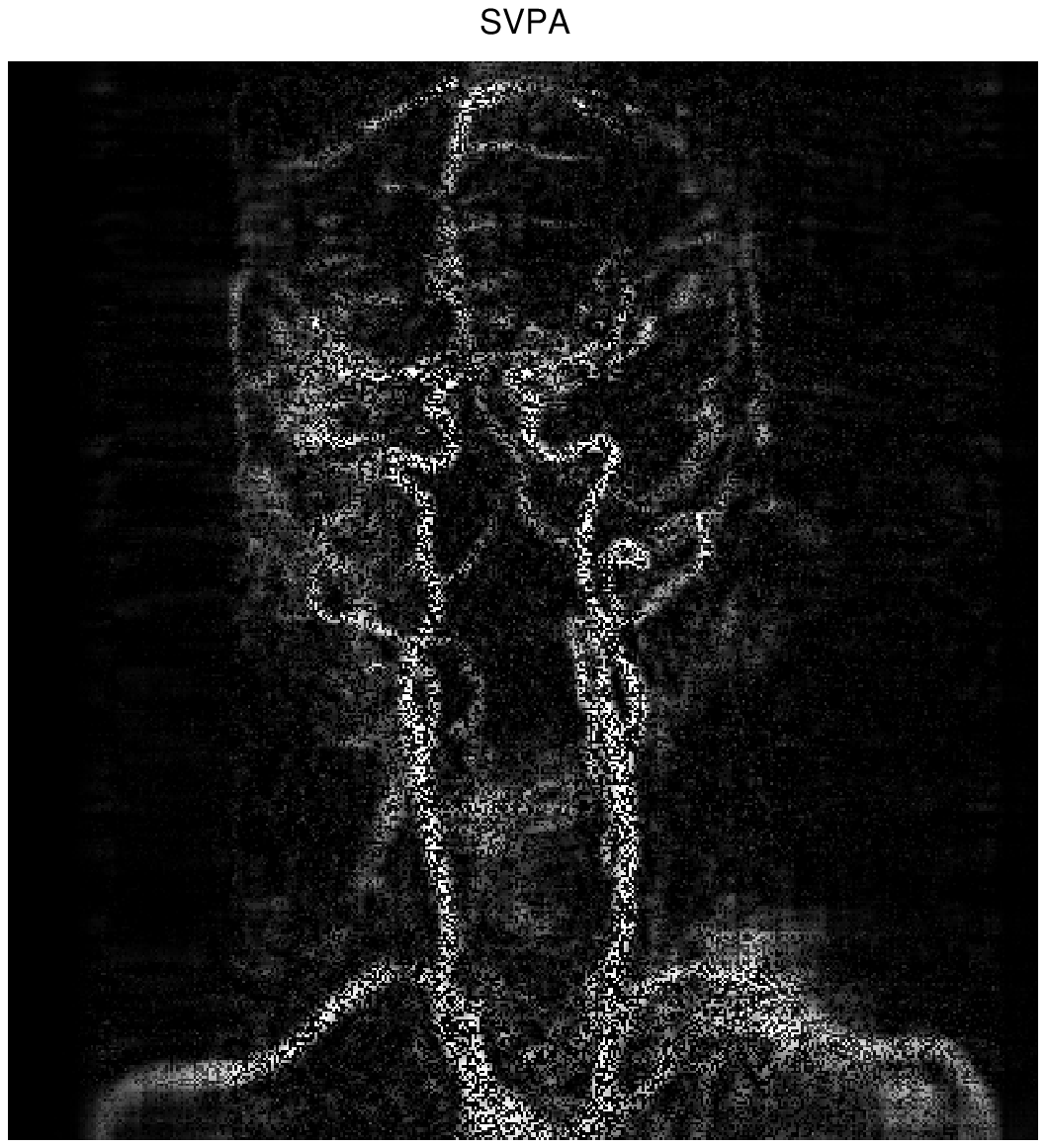}
  \end{minipage}
  \caption{Comparisons of ISVTA, SVTA and SVPA for recovering the approximated low-rank IVI with SR=0.40.} \label{figure8}
\end{figure}

\section{Conclusions}
It is well known that affine matrix rank minimization problem is combinatorial and NP-hard in general. Therefore, it is important to choose a
suitable substitution for it. In this paper, a continuous promoting low-rank non-convex fraction function is studied to replace the rank function
in this NP-hard problem, and then the NP-hard affine matrix rank minimization problem can be translated into a transformed affine matrix rank
minimization problem. Inspired by our former work in compressive sensing, the iterative singular value thresholding algorithm is proposed to
solve the regularization transformed affine matrix rank minimization problem. For different $a>0$, we can get a far more better result by adjusting the
values of the parameter $a$, which is one of the advantages for the iterative singular value thresholding algorithm compared with some state-of-art methods.
We proved that the value of the regularized parameter $\lambda>0$ can not be chosen too large. Indeed, there exists $\bar{\lambda}>0$ such that the
optimal solution of the regularization transformed affine matrix rank minimization problem is equal to zero for any $\lambda>\bar{\lambda}$. Moreover,
some convergence results are established and numerical experiments show that this thresholding algorithm is feasible for solving the regularization transformed
affine matrix rank minimization problem. Numerical experiments on completion of low-rank random matrices show that our method performs powerful in finding a
low-rank matrix and the numerical experiments for the image inpainting problems show that our algorithm have better performances than some state-of-art methods.

\begin{acknowledgements}
The work was supported by the National Natural Science Foundations of China (11131006, 11271297) and the Science
Foundations of Shaanxi Province of China (2015JM1012).
\end{acknowledgements}


\begin{thebibliography}{}
%
%
\bibitem{1}
E. J. Cand\`{e}s, B. Recht, Exact matrix completion via convex optimization. Foundations of Computational Mathematics, 9, 717-772 (2009)

\bibitem{2}
D. Jannach, M. Zanker, A. Felfernig and G. Friedrich, Recommender Systerm: An Introduction. Cambridge university press, New York (2012)

\bibitem{3}
M. Fazel, H. Hindi and S. Boyd, A rank minimization heuristic with application to minimum order system approximation. In proceedings of
American Control Conference, Arlington, VA, 6, 4734-4739 (2001)

\bibitem{4}
M. Fazel, H. Hindi and S. Boyd, Log-det heuristic for matrix minimization with applications to Hankel and Euclidean distance matrices.
In Proceedings of American Control Conference, Denever, Colorado, 3, 2156-2162 (2003)

\bibitem{5}
S. Ji, K. F. Sze and Z. Zhou, Beyond Convex Relaxation: A polynomial-time nonconvex optimization approach to network localization.
INFOCOM, 2013 Proceedings IEEE, 12, 2499-2507 (2013)

\bibitem{6}
B. Recht, M. Fazel and P. A. Parrilo, Guaranteed minimum-rank solution of linear matrix equations via nuclear norm minimization.
SIAM Review, 52, 471-501 (2010)

\bibitem{7}
E. J. Cand\`{e}s, T. Tao, The power of convex relaxation: Near-optimal matrix completion. IEEE Transactions on Information Theory, 56,  2053-2080 (2010)

\bibitem{8}
M. Fazel, Matrix Rank Minimization with Applications. PhD thesis, Stanford University (2002)

\bibitem{9}
E. J. Cand\`{e}s, Y. Plan, Matrix completion with noise. Proceedings of the IEEE, 98, 925-936 (2010)

\bibitem{10}
Y. Liu, D. Sun and K. C. Toh, An implementable proximal point algorithmic framewprk for nuclear norm minimization.
Mathematical Programming, 133, 399-436 (2012)

\bibitem{11}
J. Cai, E. J. cand\`{e}s and Z. W. Shen, A singular value thresholding algorithm for matrix completion. SIAM Journal
on Optimization, 20, 1956-1982 (2010)

\bibitem{12}
K. C. Toh, S. Yun, An accelerated proximal gradient algorithm for nuclear norm regularized linear least squares problems.
Pacific Journal of Optimization, 6, 615-640 (2012)

\bibitem{13}
S. Ma, D. Goldfarb and L. Chen, Fixed point and Bregman iterative methods for matrix rank minimization.
Mathematical Programming, 128, 321-353 (2011)

\bibitem{14}
I. Daubechies, M. Defrise and D. M. Christine, An iterative thresholding algorithm for linear inverse problems with a
sparsity constraint. Communications on Pure and Applied Mathematics, 57(11), 1413-1457 (2004)

\bibitem{15}
R. Chartrand, Exact reconstruction of sparse signals via nonconvex minimization. IEEE Signal Processing Letters, 14,707-710 (2007)

\bibitem{16}
R. Chartrand, V. Staneva, Restricted isometry isometry properties and nonconvex compressive sensing. Inverse Problems, 147, 657-682 (2008)

\bibitem{17}
S. Foucart, M. Lai, Sparsest solutions of underdetermined linear systems via $\ell_{q}$ minimization for $0<q\leq1$.
Applied and Computational Harmonic Analysis, 26, 395-407 (2009)

\bibitem{18}
M. Lai, J. Wang, An unconstrained $\ell_{q}$ minimization with $0<q\leq1$ for sparse solution of underdetermined linear systems.
SIAM Journal on Optimization, 21, 82-101 (2011)

\bibitem{19}
X. Chen, F. Xu and Y. Ye, Lower bound theory of nonzero entries in solutions of $\ell_{2}$-$\ell_{p}$ minimization.
SIAM Journal on Scientific Computing, 32, 2832-2852 (2010)

\bibitem{20}
I. Daubechies, R. Devore, M. Fornasier and C. S. Gunturk, Iteratively reweighted least squares minimization for sparse
recovery. Communications on Pure and Applied Mathematics, 63, 1-38 (2010)

\bibitem{21}
N. Mourad, J. P. Reilly, Minimizing nonconvex functions for sparse vector reconstruction. IEEE Transactions on Signal
Processing, 58, 3485-3496 (2010)

\bibitem{22}
Q. Sun, Recovery of sparsest signals via $\ell_{q}$ minimization. Applied and Computational Harmonic Analisis, 32, 329-341 (2010)

\bibitem{23}
R. Chartrand, V. Staneva, Restricted isometry properties and nonconvex compressive sensing. Inverse Problems,
24, 657-682 (2008)

\bibitem{24}
J. Peng, S. Yue and H. Li, NP/CMP Equivalence: A phenomenon hidden among sparsity models $\ell_{0}$ minimization and $\ell_{p}$
minimization for information processing. IEEE Transaction on Information Theory, 61, 4028-4033 (2015)

\bibitem{25}
Z. Xu, H. Zhang, Y. Wang, X. Chang and Y. Liang, L1/2 regularization. Science China Information Sciences,
53, 1159-1169 (2010)

\bibitem{26}
H. MOHIMANI, Z. M. BABAIE and C. JUTTEN, A fast approach for overcomplete sparse decomposition based on smoothed L0-norm[J].
IEEE Transctions on signal Processing, 57(1), 289-301 (2009)

\bibitem{27}
D. Geman and G. Reynolds. Constrained restoration and recovery of discontinuities. IEEE Transactions on Pattern Analysis and Machine Intelligence, 14(3), 367-383 (1992)

\bibitem{28}
M. Nikolova. Local strong homogeneity of a regularized estimator. SIAM Journal on Applied Mathematics, 61(2), 633-658 (2000)

\bibitem{29}
H. Li, Q. Zhang, A. Cui and J. Peng, Minimization of fraction function penalty in compressed sensing. Submitted

\bibitem{30}
J. Zeng, S. Lin, Y. Wang, and Z. Xu, L1/2 Regularization: Convergence of iterative half thresholding algorithm, IEEE transaction on signal processing, 62(9), 2317-2329 (2014)

\bibitem{31}
R. Meka, P. Jain and I. Dhillon, Guaranteed rank minimization via singular value projection. Proceeding of the Neural Information
Processing Systems Conference (NIPS), 937-945 (2010)

\end{thebibliography}


\end{document}